\documentclass[a4paper,11pt]{amsart}
\usepackage[left=2.7cm,right=2.7cm,top=3.5cm,bottom=3cm]{geometry}
\usepackage{amsthm,amssymb,amsmath,amsfonts,mathrsfs,amscd}
\usepackage[latin1]{inputenc}
\usepackage[all]{xy}
\usepackage{latexsym}
\usepackage{longtable}

\newcommand{\Q}{\mathbb{Q}}

\newcommand{\qbar }{{\bar \Q }}

\newfont{\gotip}{eufb10 at 12pt}

\DeclareMathOperator{\Hom}{Hom}

\theoremstyle{plain}
\newtheorem{theorem}{Theorem}[section]
\newtheorem{corollary}[theorem]{Corollary}
\newtheorem{lemma}[theorem]{Lemma}
\newtheorem{proposition}[theorem]{Proposition}
\newtheorem{conjecture}[theorem]{Conjecture}

\theoremstyle{definition}
\newtheorem{definition}[theorem]{Definition}

\newtheorem{examplewr}[theorem]{Examples}

\theoremstyle{remark}
\newtheorem{obswr}[theorem]{Observation}
\newtheorem{remarkwr}[theorem]{Remark}

\newenvironment{remark}{\begin{remarkwr}\begin{upshape}}{\end{upshape}\end{remarkwr}}

\input{tcilatex}

\begin{document}
\title[$\mathcal{L}$-invariants and Darmon cycles attached to modular forms]{%
$\mathcal{L}$-invariants and Darmon cycles attached to modular forms}
\author{Victor Rotger, Marco Adamo Seveso}
\maketitle

\begin{abstract}
Let $f$ be a modular eigenform of even weight $k\geq 2$ and new at a prime $%
p $ dividing exactly the level with respect to an indefinite quaternion
algebra. The theory of Fontaine-Mazur allows to attach to $f$ a monodromy
module $\mathbf{D}^{FM}_f$ and an $\mathcal{L}$-invariant $\mathcal{L}%
^{FM}_f $. The first goal of this paper is building a suitable $p$-adic
integration theory that allows us to construct a new monodromy module $\mathbf{D}%
_f$ and ${\mathcal{L}}$-invariant ${\mathcal{L}}_f$, in the spirit of
Darmon. We conjecture both monodromy modules are isomorphic, and in
particular the two ${\mathcal{L}}$-invariants are equal.

For the second goal of this note we assume the conjecture is true. Let $K$
be a real quadratic field and assume the sign of the functional equation of
the $L$-series of $f$ over $K$ is $-1$. The Bloch-Beilinson conjectures
suggest that there should be a supply of elements in the Mordell-Weil group
of the motive attached to $f$ over the tower of narrow ring class fields of $%
K$. Generalizing work of Darmon for $k=2$, we give a construction of local
cohomology classes which we expect to arise from global classes and satisfy
an explicit reciprocity law, accounting for the above prediction.
\end{abstract}

\tableofcontents

\section{Introduction}

\label{S0}

Let $X/\mathbb{Q}$ denote the canonical model of the smooth projective
Shimura curve attached to an Eichler order $\mathcal{R}$ in an indefinite
quaternion algebra $\mathcal{B}$ over $\mathbb{Q}$. When $\mathcal{B}\simeq {%
\ \mathrm{M}}_2(\mathbb{Q})$ (respectively $\mathcal{B}$ is division), $X$
is the coarse moduli space parametrizing generalized elliptic curves (resp.
abelian surfaces with multiplication by a maximal order in $\mathcal{B}$)
together with a $\Gamma_0$-level structure.

Let $k\geq 2$ be an even integer and let $n:=k-2$ and $m:=n/2$. As explained
in \cite{Ja}, \cite{Sc} for $\mathcal{B}\simeq {\mathrm{M}}_2( \mathbb{Q})$
and in \cite[§10.1]{IS} when $\mathcal{B}$ is division, there exists a Chow
motive $\mathcal{M}_{n}$ over $\mathbb{Q}$ attached to the space $S_k(X)$ of
cusp forms of weight $k$ on $X$. Attached to any eigenform $f\in S_{k}(X)$, there
exists a Grothendieck motive $\mathcal{M}_{n,f}$ over $\mathbb{Q}$ with
coefficients over the field $L_{f}:=\mathbb{Q}(\{a_{\ell }(f)\})$ generated
by the eigenvalues of $f$ under the action of the Hecke operators $T_{\ell }$
for all prime $\ell $, which is constructed as the $f$-isotypical factor of $\mathcal{M}_{n}$ in the category of Grothendieck motives (cf.\,\cite[Thm. 1.2.4]{Sc}).

Fix a prime $p$ and let $H_p(\mathcal{M}_n)$ denote the $p$-adic ètale
realization of $\mathcal{M}_{n}$ obtained as the $(m+1)$-th Tate twist of
the $p$-adic \`{e}tale cohomology of a suitable Kuga-Sato variety. It is a
finite dimensional continuous representation of $G_{ \mathbb{Q}}=\func{Gal\,}
(\qbar/\mathbb{Q})$ over $\mathbb{Q}_p$, endowed with a compatible action of
a Hecke algebra. Similarly, for any eigenform $f\in S_k(X)$ let $V_p(f)$
denote the $p$-adic realization of ${\mathcal{M}}_{n,f}$, a two-dimensional
representation over $L_{f,p}:=L_f\otimes \mathbb{Q}_p$.

Assume now that $p$ divides exactly the level of $\mathcal{R}$. Let $\mathbb{%
\ T}$ denote the maximal quotient of the algebra generated by the Hecke
operators acting on $S_k(X)^{p-new}$ and let $V_p:=H_p( \mathcal{M}_n)
^{p-new}$ denote the $p$-new quotient of $H_p(\mathcal{M}_n)$.

The restriction of $V_p$ to a decomposition subgroup $D_p\simeq \func{Gal\,}
(\qbar_p/\mathbb{Q}_p)$ is a \emph{semistable} representation and the
filtered $(\phi,N)$-module $\mathbf{D}^{FM}=D_{st}(V_p)$ attached by
Fontaine and Mazur to $V_p$ is a two-dimensional monodromy $\mathbb{T}
\otimes \mathbb{Q}_p$-module over $\mathbb{Q}_p$ in the sense of \cite[
Definition 2.2]{IS}. An important invariant of its isomorphism class is the $%
{\mathcal{L}}$-invariant $\mathcal{L}^{FM} := {\mathcal{L}}(\mathbf{D}
^{FM})\in \mathbb{T} \otimes \mathbb{Q}_p$ that one may associate to it. We
refer the reader to \cite{Ma} and \cite[§2]{IS} (and to Proposition \ref{UL}
below) for details. Similarly, let $\mathbf{D}^{FM}_f$ and ${\mathcal{L}}
^{FM}_f\in L_{f,p}$ respectively denote the two-dimensional monodromy module
over $L_{f,p}$ and ${\mathcal{L}}$-invariant associated with $f$.

An illustrative explicit example arises when $k=2$, since then $n=0$ and ${\ 
\mathcal{M}}_0$ can simply be interpreted as the Jacobian $J$ of $X$. Then ${%
\mathcal{M}}_{0,f}=A_f$ is the abelian variety attached to $f$ by Shimura
(cf.\,\cite{Sh1}). As is well-known, if $f$ is an eigenform in $%
S_k(X)^{p-new}$ then $A_f$ has purely toric reduction at $p$ and
Tate-Morikawa's theory allows to attach to it an ${\mathcal{L}}$-invariant ${%
\ \mathcal{L}} (A_f)\in L_{f,p}$ purely in terms of the $p$-adic rigid
analytic description of this variety. When $E=A_f$ is an elliptic curve, for
instance, this ${\mathcal{L}}$-invariant is simply 
\begin{equation*}
{\mathcal{L}}(E)=\frac{\log(q)}{{\func{ord}}_p(q)},
\end{equation*}
where $q=q(E)$ is the Tate period of $E$.

Thanks to the work of several authors (Greenberg-Stevens,
Kato-Kurihara-Tsuji, Coleman-Iovita, Colmez) we now know that ${\mathcal{L}}
^{FM}_f={\mathcal{L}}(A_f)$. The importance of this invariant partly relies
on the fact that, when $a_p=1$, it accounts for the discrepancy between the
special values of the classical L-series $L(f,s)$ and the $p$-adic
L-function $L_p(f,s)$ at $s=1$. This phenomenon was predicted by Mazur, Tate
and Teitelbaum as the \emph{exceptional zero conjecture} and was first
proved by Greenberg and Stevens.

For higher weights $k\geq 4$ similar phenomena occur, and several, a priori
different, ${\mathcal{L}}$-invariants attached to a $p$-new eigenform $f$
were defined by several authors (Teitelbaum, Coleman, Darmon and Orton,
Breuil) besides the aforementioned Fontaine-Mazur's ${\mathcal{L}}^{FM}_f$.
Let us stress that the definition of all these invariants is not always
available in the general setting of this introduction. However, we again
know now, thanks to the previously mentioned works together with \cite{Br}, 
\cite{BDI} and \cite{IS}, that all these invariants are equal whenever are
defined. See the above references for a detailed account of the theory.

The ${\mathcal{L}}$-invariant ${\mathcal{L}}_{f}^{D}$ introduced by Darmon
in the foundational work \cite{Dar} (and generalized by Orton in \cite{Or}
and Greenberg in \cite{Gr}) is the one which is more germane to this article
(cf.\ also §\ref{A}). Darmon's ${\mathcal{L}}$-invariant is only available
when $\mathcal{B}\simeq {\mathrm{M}}_{2}(\mathbb{Q})$ and when $\mathcal{B} $
is an indefinite quaternion algebra but $k=2$. Note that when $\mathcal{B}
\simeq {\mathrm{M}}_{2}(\mathbb{Q})$ its construction heavily relies on the
theory of modular symbols, which in turn is based on the presence of
cuspidal points on the modular curve $X$. This feature is simply absent when 
$\mathcal{B}$ is a division algebra.

The first goal of this paper is providing a construction of an ${\mathcal{L}}
$-invariant ${\mathcal{L}}$ attached to the space of $p$-new cusp forms for
all quaternion algebras $\mathcal{B}$ in the spirit of Darmon, Greenberg and
Orton even in the case $k>2$. This is achieved in §\ref{S22} as a
culmination of the results gathered in §\ref{S1} and §\ref{S2}, which show
the existence of a suitable $p$-adic integration theory and are the
technical core of this paper. One of the main results of this first part of
the article is Theorem \ref{2parts}, which the reader may find of
independent interest. It is an avatar of the classical Amice-Velu-Vishik
theorem and the comparison theorem of Stevens in \cite{St}. The proof
exploits the modular representations of the quaternion algebra $\mathcal{B}$
studied intensively by Teitelbaum and others (cf.\thinspace §\ref{S1} below
for details).

In view of the above discussion it is natural to conjecture that our
invariant ${\mathcal{L}}$ equals ${\mathcal{L}}^{FM}$; cf.\,Conjecture \ref%
{Conjecture equality of the L-invariants 1} for the precise statement. And
indeed, this equality should follow as a consequence of a $p$-adic
Jacquet-Langlands correspondence, which is already available in the
literature when $\mathcal{B}\simeq {\mathrm{M}}_2(\mathbb{Q} )$ thanks to
the work of Stevens. See the forthcoming work \cite{Se} and the references
therein for progress in this direction. In §\ref{S42} we construct a
monodromy module $\mathbf{D}$ out of the ${\mathcal{L}}$-invariant ${\ 
\mathcal{L}}$ which, granting Conjecture \ref{Conjecture equality of the
L-invariants 1}, is shown to be isomorphic to $\mathbf{D}^{FM}$.

Let us now describe the second goal and main motivation of this article, to
which §\ref{S3} is devoted as an application of the material in §\ref{S1}, § %
\ref{S2} and §\ref{SMonodromy}.

Let $K$ be a number field, which for simplicity we assume to be unramified
over $p$. As in \cite{BK} and \cite{Ne2}, for every place $v$ of $K$ define $%
H_{st}^{1}(K_v,V_p)$ to be the kernel of 
\begin{equation}  \label{Hst}
H^1( K_{v},V_p) \rightarrow 
\begin{cases}
H^1( K_v^{unr},V_p) & \text{ if } v\nmid p \\ 
H^1(K_v,\mathbf{B}_{st}\otimes _{\mathbb{Q}_p} V_p) & \text{ if } v\mid p%
\end{cases}%
\end{equation}
where $K_v^{unr}$ is the maximal unramified extension of $K_v$ and $\mathbf{%
\ \ B }_{st}$ stands for Fontaine's ring (cf.\,loc.\,cit.). Define the
(semistable) \emph{Mordell-Weil group} of the representation $V_p$ as 
\begin{equation}  \label{MW}
\mathrm{MW}_{st}(K,V_p) := \ker(H^1(K,V_p)\, \overset{\prod \func{res}_v}{
\longrightarrow} \,\prod\nolimits_{v} \frac{H^1(K_{v},V_p)}{%
H_{st}^1(K_v,V_p) }).
\end{equation}

For any motive ${\mathcal{M}}$ over a field $\kappa $ and any integer $j$,
let $\func{CH}^{j}(\mathcal{M})$
denote the Chow group of cycles on ${\mathcal{M}}$ of codimension $j$
with rational coefficients and let $\func{CH}^{j}(\mathcal{M})_{0}$ denote its subgroup of null-homologous cycles. By the work of Nekov\'a\v{r} (cf.\thinspace \cite[§7]{IS} for precise statements in our general
quaternionic setting), the classical $p$-adic ètale Abel-Jacobi map induces
a commutative diagram: 
\begin{equation}
\begin{array}{ccc}
\func{CH}^{m+1}(\mathcal{M}_{n}\otimes K)_{0} & \overset{cl_{0,K}^{m+1}}{
\rightarrow } & \mathrm{MW}_{st}(K,V_{p}) \\ 
&  &  \\ 
\downarrow &  & \text{ \ }\downarrow res_{v} \\ 
\func{CH}^{m+1}(\mathcal{M}_{n}\otimes K_{v})_{0} & \overset{
cl_{0,K_{v}}^{m+1}}{\rightarrow } & H_{st}^{1}(K_{v},V_{p})%
\end{array}
\label{Diagram Abel-Jacobi maps}
\end{equation}
for any place $v$ of $K$ over $p$. Note that in this situation we have $%
\func{CH}^{m+1}({\mathcal{M}}_{n})_{0}=\func{CH}^{m+1}({\mathcal{M}}_{n})$,
as proved in \cite[Lemma 10.1]{IS}. Composing with the natural projection $%
V_{p}{\rightarrow }V_{p}(f)$, we obtain a map $\func{CH}^{m+1}({\mathcal{M}}
_{n}\otimes K)_{0}\overset{cl_{0,K}^{m+1}(f)}{\rightarrow }\mathrm{MW}
_{st}(K,V_{p}(f))$. As a natural generalization of the conjecture of Birch
and Swinnerton-Dyer, the conjectures of Bloch and Beilinson predict that
(cf.\thinspace e.g.\thinspace \cite[§4]{Ne}): 
\begin{equation}
\func{rank}_{L_{f,p}}(cl_{0,K}^{m+1}(f))={\func{ord}}_{s=k/2}L(f\otimes K,s).
\label{BBconj}
\end{equation}

Let $N^-=\func{disc}(\mathcal{B})\geq 1$ denote the reduced discriminant of $%
\mathcal{B}$ and let $pN^+$ denote the level of $\mathcal{R}$. We have $%
(N^-,pN^+)=1$ and, as we already mentioned, $p\nmid N^+$.

Assume now that $K$ is quadratic, either real or imaginary, satisfying the
following \emph{Heegner hypothesis}:

\begin{itemize}
\item The discriminant $D_K$ of $K$ is coprime to $N:=pN^+ N^-$.

\item All prime factors of $N^-$ remain inert in $K$.

\item All prime factors of $N^+$ split in $K$.

\item $p$ splits (remains inert) in $K$, if $K$ is imaginary (real,
respectively).
\end{itemize}

Thanks to the first condition, the sign of the functional equation of $%
L(f\otimes K,s)$ is simply $(\frac{-N}{K})$. The last three conditions imply
that this sign is $-1$. In particular, $L(f\otimes K,k/2)=0$. Let now $c\geq
1$ be a positive integer and let $H_{c}/K$ denote the (narrow) ring class
field of conductor $c$, whose Galois group $G_{c}:=\func{Gal}(H_{c}/K)$ is
canonically isomorphic via Artin's reciprocity map to the (narrow) Picard
group ${\mathrm{Pic}}({\mathcal{O}}_{c})$ of the order ${\mathcal{O}}
_{c}\subset K$ of conductor $c$. Assuming $(c,N)=1$, for any character $\chi
:G_{c}{\rightarrow }\mathbb{C}^{\times }$ the root number of the twisted
L-series $L(f\otimes K,\chi ,s)$ continues to be $-1$ and the L-series of $%
f\otimes H_{c}$ admits the factorisation $L(f\otimes H_{c},s)=\prod_{\chi
\in G_{c}^{\vee }}L(f\otimes K,\chi ,s)$. It follows that ${\func{ord}}
_{s=k/2}L(f\otimes H_{c},s)\geq h({\mathcal{O}}_{c}):=|G_{c}|$ and the
Bloch-Beilinson conjecture \eqref{BBconj} predicts that $\func{rank}
_{L_{f,p}}(cl_{0,K}^{m+1}(f))=h({\mathcal{O}}_{c})$. In crude terms, there
should be a systematic way of producing a collection of nontrivial elements 
\begin{equation}
\{s_{c}\in \mathrm{MW}_{st}(H_{c},V_{p}(f))\}
\end{equation}
in the Mordell-Weil group of $f$ with coefficients on the tower of class
fields $H_{c}/K$ for $c\geq 1$,$(c,N)=1$.

When $K$ is imaginary, and $N^-=1$, Nekov\'a\v{r} \cite{Ne} was able to
construct these sought-after elements as images by the $p$-adic ètale
Abel-Jacobi map $cl_{0,K}^{m+1}(f)$ of certain Heegner cycles on ${\mathcal{%
\ \ M }}_{n}$ whose construction exploits, as in the classical case $k=2$,
the theory of complex multiplication on elliptic curves. This construction
was later extended to arbitrary discriminants $N^- \geq 1$ by Besser (cf.\,%
\cite[ §8]{IS} for a review).

Assume for the rest of this introduction (and of the article) that $K$ is
real. The aim of §\ref{S3} is exploiting the $p$-adic integration theory
established in §\ref{S2} in order to propose a conjectural construction of
suitable analogues of Heegner cycles for real quadratic fields.

Namely, our construction yields local cohomology classes $s_c\in
H_{st}^1(K_p,V_p)$ which we expect to arise from global cohomology classes
in $\mathrm{MW}_{st}(H_c,V_p) $. Notice that this makes sense, as $H_c$
naturally embeds in $K_p$ because $p$ is inert in $K$.

More precisely, we produce local cohomology classes $s_{\Psi}\in
H_{st}^1(K_p,V_p)$ for every oriented optimal embedding $\Psi: {\mathcal{O}}
_c\hookrightarrow \mathcal{R}$. We expect them to be global over $H_c$ and
we conjecture that they satisfy a reciprocity law which describes the Galois
action of $G_c $ on them. In addition, one further expects these classes to
be related, via a Gross-Zagier formula, to the first derivative of $%
L(f\otimes K,s)$ at $s=k/2$. See §\ref{S3} for precise statements. This
provides a higher weight generalization of the theory of points due to
Darmon \cite{Dar} and continued in \cite{Das}, \cite{Gr}, \cite{DG}, \cite%
{LRV} and \cite{LRV2}.

A fundamental difference (or perhaps one should say inconvenient, as this is
the reason why the rationality of $s_{\Psi}$ over $H_c$ remains highly
conjectural) of this construction when compared with Nekov\'a\v{r}'s
approach is that these cohomology classes are not defined (at least not a
priori) as the image of any cycles on $\func{CH}^{m+1}( \mathcal{M}
_{n}\otimes K_p)_0$.

Instead, letting $\mathbf{P}_{n}$ denote the space of polynomials of degree $%
\leq n$ with coefficients in $K_p$, the role of the Chow group in our
setting is played by the module $H_1(\Gamma ,\func{Div}( \mathcal{H}
_p)(K_p)\otimes \mathbf{P}_n)$. The choice of this module is motivated by
the fact that one can naturally attach a $1$-cycle $y_{\Psi}$ to each optimal
embedding $\Psi$, in a manner which is reminiscent of the $p$-adic
construction of Heegner points for imaginary quadratic fields (cf.\,\cite[§5]%
{BD}), and is a straightforward generalization of the points defined by M. Greenberg in 
\cite{Gr} (cf.\,also \cite{LRV2}). For this reason, the cycles $y_{\Psi}$ may be called Stark-Heegner cycles (following loc. cit.) or also, 
as we suggest here, {\em Darmon cycles}.

Here, $\Gamma \subseteq (\mathcal{B} \otimes \mathbb{Q}_p)^{\times}$ is a
group whose definition is recalled in § \ref{S1} and already makes its
appearance in classical works of Ihara and in \cite{Dar}. The module $\func{
Div}(\mathcal{H}_p)(K_p)$ is the subgroup of divisors with coefficients on $%
\mathcal{H}_p(\bar{K}_p):=\bar{K}_p\setminus \mathbb{Q}_p$ which are
invariant under the action of the Galois group $\func{Gal\,}(\bar{K}_p/K_p)$.

We define $s_{\Psi}$ as the image of $y_{\Psi}$ by a composition of
morphisms 
\begin{equation}
H_1(\Gamma ,\func{Div}(\mathcal{H}_p)(K_p)\otimes \mathbf{P}_n) \overset{
\Phi^{AJ}}{{\longrightarrow }} \frac{D\otimes K_p}{F^m(\mathbf{D}\otimes
K_p) } \overset{?}{\simeq} \frac{D^{FM}\otimes K_p}{F^m(\mathbf{D}%
^{FM}\otimes K_p)}\simeq H_{st}^1(K_p,V_p)
\end{equation}
where the first map is introduced in \eqref{AJ2} and should be regarded as
an analogue of the $p$-adic Abel-Jacobi map; the second map is the
conjectural isomorphism predicted by Conjecture \ref{Conjecture equality of
the L-invariants 1}; the last map is the isomorphism provided by
Bloch-Kato's exponential. Cf.\,§\ref{S34} for more details.

The last section of this manuscript is devoted to the particular cases $k=2$
in §\ref{A1} and $N^{-}=1$ in §\ref{A2}. For $k=2$ we quickly review the
work of \cite{Gr} and \cite{LRV}, comparing it to our constructions. For $%
N^{-}=1$ we rephrase the theory in the convenient language of modular
symbols. This formulation is employed in \cite{Se}, where Conjecture \ref{Conjecture rationality} (3) and Conjecture \ref{Conjecture non-trivial
cases} are proved for suitable genus characters of $K$.

\vspace{0.3cm}

\emph{Acknowledgements.} It is a pleasure to thank Michael Spiess for his
wise comments and advice at delicate stages of this project. The authors also thank Stefano Vigni for his comments, which helped to improve the exposition of our results, and the Centre de Recerca Matem{\`a}tica (Bellaterra, Spain) for its warm hospitality in Winter 2010, when part of this research was carried out.

\section{Modular representations of quaternion algebras}

\label{S1}

\subsection{Quaternion algebras and Hecke modules}

\label{S11}

Let $\mathcal{B}$ be a quaternion algebra over $\mathbb{Q}$ and let $N^-\geq
1$ denote its reduced discriminant. Let $b\mapsto \overline{b}$ denote the
canonical anti-involution of $\mathcal{B}$ and write $\mathrm{Tr}(b)=b+ 
\overline{b}$, $\mathrm{n}(b)=b\overline{b}$ for the reduced trace (resp.
norm) of elements of $\mathcal{B}$.

Assume $\mathcal{B}$ is indefinite, that is, $N^-$ is the square-free
product of an even number of primes. Equivalently, there is an isomorphism $%
\iota _{\infty }:\mathcal{B}\otimes \mathbb{R}\simeq \mathrm{M}_{2}(\mathbb{%
R })$, which we fix for the rest of the article.

For any prime $\ell$, write $\mathcal{B}_{\ell}=\mathcal{B}\otimes \mathbb{Q}
_{\ell}$ and fix isomorphisms $\iota_{\ell}:\,\mathcal{B}_{\ell}\,\simeq \, 
\mathrm{M} _{2}(\mathbb{Q}_{\ell}) $ for $l\nmid N^-$ and $\mathcal{B}
_{\ell}\,\simeq \,\mathbb{H}_{\ell}$ for $\ell\mid N^-$. Here, $\mathbb{H}
_{\ell}$ stands for a fixed choice of a division quaternion algebra over $%
\mathbb{Q}_{\ell}$, which is unique up to isomorphism. Throughout, for each
place $l\leq \infty$ of $\mathbb{Q} $ we shall regard $\mathcal{B}$ as
embedded in $\mathrm{M}_2(\mathbb{Q}_{\ell})$ or $\mathbb{H}_{\ell}$ via the
above fixed isomorphisms.

Let $N^+\geq 1$ be a positive integer coprime with $N^-$ and fix a prime $%
p\nmid N^+N^-$. Write $N=pN^+N^-$ and let $\mathcal{R}_0(pN^+)\,\subset \, 
\mathcal{R}_0(N^+)$ be Eichler orders in $\mathcal{B}$ of level $pN^+$ and $%
N^+$. Let $\Gamma _0(pN^+)$ (resp. $\Gamma_0(N^+)$) denote the subgroup of $%
\mathcal{R}_0(pN^+)^{\times }$ (resp. $\mathcal{R}_0(N^+)^{\times }$) of
elements of reduced norm $1$. Choose an element ${\omega }_p\in \mathcal{R}
_0(pN^+)$ of reduced norm $p$ normalizing $\Gamma_0(pN^+)$ and set $\hat{
\Gamma }_0(N^+):= {\omega }_p \Gamma_0(N^+) {\omega }_p^{-1}$. In order to
lighten the notation, there is no reference to the discriminant $N^-$ in the
symbols chosen to denote these orders and groups; this should cause no
confusion, as the quaternion algebra $\mathcal{B}$ will always be fixed in
our discussion.

Both $\Gamma_0(pN^+)$ and $\Gamma_0(N^+)$ are naturally embedded in $\mathrm{%
\ SL}_2(\mathbb{R})$ and act discrete and discontinuously on Poincar\'e's
upper half-plane $\mathcal{H}$ through Moebius transformations, with compact
quotient if and only if $N^->1$. Let $X^{N^-}_0(pN^+)$, resp. $%
X^{N^-}_0(N^+) $, denote Shimura's canonical model over $\mathbb{Q} $ of
(the cuspidal compactification of, if $N^-=1$) these quotients (cf.\,\cite[§%
9.2]{Sh1}).

For reasons that will become clear later, it will also be convenient to
consider the Eichler $\mathbb{Z}[ 1/p] $-order $\mathcal{R}:= \mathcal{R}
_0(N^+)[ 1/p] $. Similarly as above, let $\Gamma$ denote the subgroup of
elements of reduced norm $1$ of $\mathcal{R}^{\times }$. This group was
first studied by Ihara and also makes an appearance in the works \cite{Dar}, 
\cite{Das}, \cite{Gr} and \cite{LRV}.

If $A$ is a module endowed with an action of $\mathcal{B}^{\times }$ and $G$
is either $\Gamma _{0}(pN^{+})$, $\Gamma _{0}(N^{+})$, $\hat{\Gamma}
_{0}(N^{+})$ or $\Gamma $, the homology and cohomology groups $H_{i}(G,A)$
and $H^{i}(G,A)$ are naturally modules over a Hecke algebra 
\begin{equation*}
\mathcal{H}(G):=\mathbb{Z}[T_{\ell }:\ell \nmid N_{G};U_{\ell }:\ell \mid
N_{G}^{+},W_{\ell }^{-}:\ell \mid N^{-},W_{p},W_{\infty }],
\end{equation*}
where $N_{G}^{+}=pN^{+}$ for $G=\Gamma _{0}(pN^{+}),\Gamma $ and $%
N_{G}^{+}=N^{+}$ otherwise, and $N_{G}=N_{G}^{+}N^{-}$. If $A_{1}{\
\rightarrow }A_{2}$ is a morphism of $\mathcal{B}^{\times }$ -modules, the
corresponding maps 
\begin{equation}
H^{i}(G,A_{1}){\rightarrow }H^{i}(G,A_{2})  \label{Hecke}
\end{equation}
are then morphisms of $\mathcal{H}(G)$-modules. Cf.\thinspace e.g.\thinspace 
\cite[§1]{AS}, \cite[§3]{Gr} and \cite[§2]{LRV} for details.

Choose an element $\omega _p\in \mathcal{R}_0(pN^+)$\ (resp. $\omega_{\infty
}$) of reduced norm $p$\ (resp. $-1$)\ that normalizes $\mathcal{R}_0(pN^+)$
; such elements exist and are unique up to multiplication by elements of $%
\Gamma _0(pN^+)$. The operators $W_p$ and $W_{\infty}$ mentioned above are
the (Atkin-Lehner) involutions defined as the double-coset operators
attached to ${\omega }_p$ and ${\omega }_{\infty}$, respectively. For any $%
\mathbb{Z}[W_\infty]$-module $A$ and sign $\epsilon\in\{\pm1\}$ we set $%
A^\epsilon:=A/(W_\infty-\epsilon)$. Up to $2$-torsion, $A \simeq A^+ \oplus
A^-$.

For any element $\gamma $ in $\mathrm{GL}_2(\mathbb{Q}_p)$ or $\mathrm{GL}
_2( {\mathbb{R}} )$, write $\hat{\gamma}:= {\omega }_p \gamma {\omega }
_p^{-1}$. For any subgroup $G$ of $\mathrm{GL}_2(\mathbb{Q}_p)$ or $\mathrm{%
GL}_2({\mathbb{R}} )$, write $\hat G = \{ \hat g, g\in G\}$. Note that $\hat{
\Gamma_0(pN^+)}=\Gamma_0(pN^+)$, $\hat{\Gamma}=\Gamma$, whereas $\hat{
\Gamma }_0(N^+) \ne \Gamma_0(N^+)$. In fact, 
\begin{equation}
\Gamma = \Gamma_0(N^+) \star_{\Gamma_0(pN^+)} \hat{\Gamma }_0(N^+)
\end{equation}
is the amalgamated product of $\Gamma_0(N^+)$ with $\hat{\Gamma }_0(N^+)$
over $\Gamma_0(pN^+) = \Gamma_0(N^+) \cap \hat{\Gamma }_0(N^+) $.

\subsection{The Bruhat-Tits tree}

\label{S12}

Let $\mathcal{T}$ denote Bruhat-Tits' tree attached to $\mathrm{PGL}_{2}( 
\mathbb{Q}_p)$, whose set $\mathcal{V}$ of vertices is the set of homothety
classes of rank two $\mathbb{Z}_p$-submodules of $\mathbb{Q}_p^2$. Write $%
\mathcal{E}$ for the set of oriented edges of the tree. Given $e\in \mathcal{%
E}$, write $s(e)$ and $t(e)$ for the source and target of the edge, and $%
\bar e$ for the edge in $\mathcal{E}$ such that $s(\bar e)=t(e)$ and $t(\bar
e)=s(e)$. Cf.\,e.g.\,\cite[§1.3.1]{DT} for more details.

Write $v_{\ast }$, $\hat{v}_{\ast }$ for the vertices associated with the
standard lattice $L_{\ast }:=\mathbb{Z}_p\times \mathbb{Z}_p$ and the
lattice $\hat{L}_{\ast}:=\mathbb{Z}_p\times p\mathbb{Z}_p$, respectively.
Note that $\omega _p$ acts on $\mathcal{T}$, mapping $v_{\ast }$ to $\hat{v}
_{\ast }$. In general, for any vertex $v\in \mathcal{V}$, write $\hat{v}
:=\omega _p(v)$.

Let $e_{\ast}$ be the edge with source $s(e_{\ast }) = v_{\ast}$ and $%
t(e_{\ast})=\hat{v}_{\ast}$. Let $\mathcal{V}^+$ (resp. $\mathcal{V}^-$)
denote the subset of vertices $v\in \mathcal{V}$ which lie at \emph{even}
(resp. \emph{odd}) distance from $v_{\ast }$. Similarly, write $\mathcal{E}
^+ $ (resp. $\mathcal{E}^-$) for the subset of edges $e$ in $\mathcal{E}$
such that $s(e)\in \mathcal{V}^+$ (resp. $\mathcal{V}^-$).

Let $G$ be a subgroup of $\mathrm{GL}_{2}(\mathbb{Q}_p)$ (as the ones
already introduced in the previous section) and let $A$ be any left $G$
-module. For any set ${\mathcal{S}}$, e.g.$\,{\mathcal{S}}=\mathcal{V}$ or $%
\mathcal{E}$, write $C({\mathcal{S}},A)$ for the group of functions on ${\ 
\mathcal{S}}$ with values in $A$.

Let also $C_0(\mathcal{E},A)$ be the subgroup of functions $c$ in $C( 
\mathcal{E},A)$ such that $c(\bar{e})=-c(e)$ for all $e\in \mathcal{E}$, and 
\begin{equation*}
C_{har}( A) =\{c\in C_0(\mathcal{E},A)\,:\sum_{s(e)=v}c(e)=0\quad \forall
v\in \mathcal{V}\,\}
\end{equation*}
be the subgroup of $A$-valued \emph{harmonic cocycles}. These groups are
naturally endowed with a left action of $G$ by the rule $(^{\gamma
}c)(e):=\gamma (\,c(\gamma ^{-1}e)\,)$ and it is easy to see that they sit
in the exact sequences (cf.\,\cite[Lemma 24]{Gr} for the first one): 
\begin{equation}  \label{exact sequence edges/vertices}
\begin{array}{ccccccccc}
0 & \rightarrow & C_{har}( A) & \rightarrow & C_0(\mathcal{E},A) & \overset{
\varphi }{\rightarrow } & C(\mathcal{V},A) & \rightarrow & 0 \\ 
&  &  &  & \varphi (c)(v) & := & \sum_{s(e)=v}c(e)\text{,} &  & 
\end{array}%
\end{equation}
\begin{equation}  \label{exact sequence vertices/edges}
\begin{array}{ccccccccc}
0 & \rightarrow & A & \rightarrow & \mathcal{C}( \mathcal{V},A) & \overset{
\partial ^{\ast }}{\rightarrow } & \mathcal{C}_0( \mathcal{E},A) & 
\rightarrow & 0 \\ 
&  &  &  & ( \partial ^{\ast }c) ( e) & := & c( s( e) ) -c( t( e) ) \text{.}
&  & 
\end{array}%
\end{equation}

\subsection{Rational representations}

\label{S13}

In this section we recall a construction of a rational representation $%
\mathbb{V}_n$ of $\mathcal{B}^{\times}$ for each even integer $n\geq 0$
which already appears in \cite[§1.2]{BDIS} and \cite[§5]{IS}. Its relevance
will be apparent in the next section, as according to the Eichler-Shimura
isomorphism (cf.\,\eqref{Diagram Eichler-Shimura isomorphism} below) the
cohomology groups of $\mathbb{V}_n$ provide a natural rational structure for
both the spaces of holomorphic and $p$-adic modular forms with respect to
the arithmetic subgroups of $\mathcal{B}^{\times}$.

Let $\mathcal{B}_0=\{b\in \mathcal{B}, \func{Tr}(b)=0\} \subset \mathcal{B}$
, endowed with a right action of $\mathcal{B}^{\times}$ by the rule $b\cdot
\beta:=\beta^{-1} b \beta$ for $\beta\in \mathcal{B}^{\times}$ and $b\in 
\mathcal{B}_0$. The pairing 
\begin{equation}  \label{pair}
\left\langle b_1,b_2\right\rangle :=\frac1{2}\mathrm{Tr}(b_1\cdot \bar{b}_2)
\end{equation}
is non-degenerate and symmetric on $\mathcal{B}_0$, and allows to identify $%
\mathcal{B}_0$ with its own dual, whence to regard it as a left $\mathcal{B}
^{\times}$-module.

For any $r\geq 0$, the $r$-th symmetric power $\limfunc{Sym}^{r}( \mathcal{B}
_0)$ of $\mathcal{B}_0$ is naturally a left $\mathcal{B}^{\times}$-module
endowed with the pairing induced by \eqref{pair}, which we continue to
denote $\left\langle -,-\right\rangle $. For $r\geq 2$, the Laplace operator 
\begin{equation*}
\begin{array}{cccc}
\Delta_r : & \limfunc{Sym}\nolimits^{r}( \mathcal{B}_0) & \rightarrow & 
\limfunc{Sym}\nolimits^{r-2}( \mathcal{B}_0)%
\end{array}%
\end{equation*}
attached to $\left\langle -,-\right\rangle $ is defined by the rule 
\begin{equation*}
\Delta_r ( b_1\cdot ...\cdot b_r) := \tsum\nolimits_{i,j=1}^{r}\left\langle
b_i,b_{j}\right\rangle b_1\cdot ...\cdot \widehat{b}_i\cdot ...\cdot 
\widehat{b}_{j}\cdot ...\cdot b_r.
\end{equation*}

The Laplace operator $\Delta_r$ is a morphism of $\mathcal{B}^{\times}$
-modules because, as one checks, $\langle b_1\cdot \beta ,b_2 \beta \rangle
= \langle b_1 ,b_2\rangle$ for all $\beta \in \mathcal{B} ^{\times}$, $b_1,
b_2\in \mathcal{B}_0$.

\begin{definition}
Let $\mathbb{V}_0 = \mathbb{Q} $, $\mathbb{V}_2=\mathcal{B}_0$ and, for any 
\emph{even} integer $n\geq 4$, let $m:=n/2$ and 
\begin{equation*}
\mathbb{V}_n:=\ker \Delta _{m}.
\end{equation*}
\end{definition}

If $R$ is a commutative $\mathbb{Q}$-algebra, write $\mathbb{V}_n( R) := 
\mathbb{V}_n\otimes R$. For $n=0$, $\mathbb{V}_0=\mathbb{Q} $ is endowed
simply with the trivial action of $\mathcal{B}^{\times}$. For arbitrary $n$,
we may regard the spaces $\mathbb{V}_n$ both as right and left $\mathcal{B}
^{\times}$-modules, the pairing $\left\langle -,-\right\rangle $ identifying
one with another (cf.\,\cite[§1.2]{BDIS}). As such, the general theory
reviewed in §\ref{S11} and §\ref{S12} applies in particular to these modules.

Over a base field $K/\mathbb{Q} $ which splits $\mathcal{B}$, the modules $%
\mathbb{V}_n(K)$ admit a much simpler and classical description, which we
now review. For any even integer $n\geq 0$ let $\mathbf{P}_n$ denote the $%
\mathbb{Q}$-vector space of polynomials of degree at most $n$ with rational
coefficients, and write $\mathbf{P}_n(R):=\mathbf{P}_n\otimes R$ for any
algebra $R$ as above. It can be endowed with a right action of $\mathrm{GL}
_{2}(R)$ by the rule 
\begin{equation*}
P(x)\cdot \gamma \,:=\,\frac{(cx+d)^n}{\det(\gamma)^{n/2}}\cdot P(\frac{ax+b 
}{cx+d}),\quad \quad \gamma = 
\begin{pmatrix}
a & b \\ 
c & d%
\end{pmatrix}
,\,P\in \mathbf{P}_n(R).
\end{equation*}

This way, $\mathbf{V}_n(R)=\mathbf{P}_n^{\vee }( R) :=\Hom_{R}(\mathbf{P}
_n(R),R)$, the dual of$\mathbf{P}_n( R) $, inherits a left $\mathrm{GL}
_{2}(R)$-action, which actually descends to ${\mathrm{PGL}}_2(R)$.

Let $K$ be a field of characteristic $0$ such that $\mathcal{B}\otimes_{ 
\mathbb{Q}} K\simeq {\mathrm{M}}_2(K)$, and identify these two algebras by
fixing an isomorphism between them. The function 
\begin{equation*}
\begin{matrix}
\mathcal{B}_0 & {\longrightarrow } & \mathbf{P_2}(K) \\ 
b & \mapsto & \mathrm{tr}(b\cdot 
\begin{pmatrix}
x & -x^2 \\ 
1 & -x%
\end{pmatrix}
)%
\end{matrix}%
\end{equation*}
is an isomorphism of right $\mathcal{B}^{\times}$-modules. Identifying $%
\mathcal{B}_0$ with its own dual via \eqref{pair}, it induces an isomorphism
of left $\mathcal{B}^{\times}$-modules (we omit the details; cf.\,\cite[§1.2]%
{BDIS}, where the definitions of the pairings and actions are the same as
the ones taken here, and \cite[§5]{IS}, \cite[§2]{JL}): 
\begin{equation}  \label{Vn}
\mathbb{V}_n(K) \simeq \mathbf{V}_n(K).
\end{equation}

Notice that we already fixed in §\ref{S11} isomorphisms $\iota_{\ell}: 
\mathcal{B} \otimes \mathbb{Q}_{\ell}\simeq {\mathrm{M}}_2(\mathbb{Q}
_{\ell}) $ for places $l\leq \infty$, $l\nmid N^-$. Accordingly, in the
sequel we shall freely identify $\mathbb{V}_n(\mathbb{Q}_{\ell})$ with $%
\mathbf{V}_n( \mathbb{Q}_{\ell})$.

\subsection{Modular forms and the Eichler-Shimura isomorphism}

\label{S14}

For any even integer $n\geq 0$ set $k=n+2 = 2(m+1)$. Let $G$ either $%
\Gamma_0(pN^+)$, $\Gamma_0(N^+)$, $\hat{\Gamma }_0(N^+)$ or $\Gamma$.

\begin{definition}
A $\mathcal{H}(G)$-module $M$ admits an \emph{Eisenstein/Cuspidal
decomposition} (of weight $k$) whenever there exists a Hecke operator $%
T_{\ell}$ for some $\ell\nmid N_G$ such that $M=M^{Eis}\oplus M^{c}$ and $%
t_{\ell}:=T_{\ell}- \ell^{k-1}-1$ vanishes on $M^{Eis} $ and is invertible
on $M^{c}$.
\end{definition}

\begin{remark}
\label{RmkEC} If such decomposition exists, it is easy to check that it is
unique and both factors $M^{Eis}$ and $M^{c}$ are naturally $\mathcal{H}(G)$
-modules.

Furthermore, let $M_i$, $i=1,2$, be $\mathcal{H}(G_i)$-modules, where $G_1$, $G_2$ is any choice of groups in
either the set $\{\Gamma_0(pN^+), \Gamma\}$ or the set $\{\Gamma_0(N^+), \hat{\Gamma }
_0(N^+) \}$. Let $f: M_1{\rightarrow } M_2$ be a morphism which is
equivariant for the actions of $\mathcal{H}(G_1)$ and $\mathcal{H}(G_2)$ in
the obvious sense. If both $M_1$ and $M_2$ admit an Eisenstein/Cuspidal
decomposition, then $f$ decomposes accordingly as $f=f^{Eis}\oplus f^{c}$.
In particular, $\ker(f)$ and $\mathrm{coker}(f)$ also admit an
Eisenstein/Cuspidal decomposition.

Finally, if 
\begin{equation*}
0 {\rightarrow } M_1 {\rightarrow } M_2 {\rightarrow } M_3 {\rightarrow }
M_4 {\rightarrow } M_5 {\rightarrow } 0
\end{equation*}
is a Hecke equivariant exact sequence of Hecke modules such that $M_1$, $M_2$
, $M_4$ and $M_5$ admit an Eisenstein/Cuspidal decomposition, then so does $%
M_3$.
\end{remark}

Let now $G$ denote either $\Gamma_0(pN^+)$ or $\Gamma_0(N^+)$. Let 
\begin{equation*}
S_k(G)\subseteq M_{k}(G)
\end{equation*}
denote the $\mathbb{C}$-vector space of weight $k$ holomorphic (cuspidal)
modular forms with respect to $G$. Let $\mathbb{T}_G$ (resp. $\tilde{\mathbb{%
\ T}}_G$) be the maximal quotient of the Hecke algebra $\mathcal{H}
(G)\otimes \mathbb{Q}$ acting faithfully on $S_k(G)$ (resp. $M_{k}(G)$).

As a basic example, $M=M_{k}(G)$ admits an Eisenstein/Cuspidal decomposition
with $M^{c}=S_{k}(G)$ and $M^{Eis}=E_{k}( G)$, the space of modular forms
generated by the Eisenstein series. These series are only defined for $N^-=1$
; in order to have uniform notations, we set this space to be $\{0\}$ when $%
N^->1$.

By \cite[Theorem 3.51]{Sh1} and the Jacquet-Langlands correspondence, 
\begin{equation}  \label{Remark on the rank 1}
\dim _{\mathbb{C}}S_{k}(G) =\dim _{\mathbb{Q}}\mathbb{T}_G
\end{equation}
and in fact $S_{k}(G)$ is a module of rank one over $\mathbb{T}_G\otimes 
\mathbb{C}$.

The Eichler-Shimura isomorphism yields an identification of exact sequences
(see \cite[Ch. 6]{Hi} and \cite[Ch. III]{Fr}) 
\begin{equation}  \label{Diagram Eichler-Shimura isomorphism}
\begin{array}{ccccccccc}
0 & \rightarrow & S_{k}( G) \otimes _{{\mathbb{R}}}\mathbb{C} & \rightarrow
& (S_{k}( G ) \otimes _{{{\mathbb{R}}}}\mathbb{C})\oplus E_{k}( G ) & 
\rightarrow & E_{k}( G) & \rightarrow & 0 \\ 
&  & \parallel &  & \parallel &  & \parallel &  &  \\ 
0 & \rightarrow & H_{par}^1(G,\mathbf{V}_n( \mathbb{C}) ) & \rightarrow & 
H^1( G,\mathbf{V}_n( \mathbb{C}) ) & \overset{\func{res}}{\rightarrow } & 
H^1_{Eis}(G,\mathbf{V}_n(\mathbb{C})) & \rightarrow & 0,%
\end{array}%
\end{equation}
where $H^1_{Eis}(G,\mathbf{V}_n(\mathbb{C}))$ is the image of the
restriction map 
\begin{equation}  \label{restr}
H^1( G,\mathbf{V}_n( \mathbb{C}) ) {\longrightarrow }
\bigoplus_{i=1}^{t}H^1(G_{s_i},\mathbf{V}_n( \mathbb{C}) ).
\end{equation}
Here, $C_{G}=\{s_1,...,s_{t}\}$ denotes a set of representatives for the
cusps of $G$ and for any $s\in C_{G}$, $G_s$ denotes the stabilizer of $s$
in $G$.

Thanks to \eqref{Vn} and to the theorem of Universal Coefficients, there
is an isomorphism of Hecke modules 
\begin{equation}
H^{1}(G,\mathbf{V}_{n}(\mathbb{C}))\simeq H^{1}(G,\mathbb{V}_{n})\otimes 
\mathbb{C}.  \label{H1Vn}
\end{equation}

Note that the map \eqref{restr} can in fact be viewed as the base change to $%
\mathbb{C}$ of the natural restriction map $H^{1}(G,\mathbf{V}_{n}){\
\longrightarrow }\bigoplus_{i=1}^{t}H^{1}(G_{s_{i}},\mathbb{V}_{n})$. As a
point of caution, the reader may notice that when $N^{-}>1$ these cohomology
groups make no sense if we replace the module of coefficients $\mathbb{V}%
_{n} $ by $\mathbf{V}_{n}$. The kernel and image of this map can thus be
taken as the definition of $H_{par}^{1}(G,\mathbb{V}_{n})$ and $%
H_{Eis}^{1}(G,\mathbb{\ V}_{n})$, respectively. Since the Eichler-Shimura
isomorphism is Hecke equivariant, it follows from \eqref{H1Vn} that $H^{1}(G,%
\mathbb{V} _{n})=H_{par}^{1}(G,\mathbb{V}_{n})\oplus H_{Eis}^{1}(G,\mathbb{V}%
_{n})$ is already an Eisenstein/Cuspidal decomposition.

\begin{remark}
\label{rankone} It thus follows from \eqref{Remark on the rank 1}, %
\eqref{Diagram Eichler-Shimura isomorphism} and \eqref{H1Vn} that $H^1(G, 
\mathbb{V}_n)^{c}$ is a module of rank two over $\mathbb{T}_{G}.$ More
precisely, since $\mathrm{n}(\omega_{\infty})=-1$, it follows from the work 
\cite{Sh3} that the Atkin-Lehner involution $W_{\infty}$ acts on $H^1(G, 
\mathbb{V}_n(\mathbb{C} ))^{c}$ as complex conjugation. Hence, for each
choice of sign $\epsilon\in \{\pm 1\}$, $H^1(G,\mathbb{V}_n)^{c,\epsilon}$
is a module of rank one over $\mathbb{T}_{G}.$
\end{remark}

\begin{remark}
\label{Remark on the rank 3} If $f\in S_{k}(G)$ is a primitive normalized
eigenform for the action of $\mathbb{T}_{G}$, it corresponds via the
Eichler-Shimura isomorphism to an element $c_{f}\in H^{1}(G,\mathbb{V}
_{n}(L_{f}))^{c}$ , where $L_{f}$ is the number field generated over $%
\mathbb{Q}$ by the eigenvalues of $f$. This is a consequence of multiplicity
one and the fact that $H^{1}(G,\mathbb{V}_{n})^{c}$ is a rational structure
for $S_{k}(G)$ which is preserved by the action of the Hecke algebra. Hence,
if $K$ is a field which contains all the eigenvalues for the action of the
Hecke operators, then $H^{1}(G,\mathbb{V}_{n}(K))^{c}$ admits a basis of
eigenvectors for this action.
\end{remark}

\begin{remark}
\label{i} For all $i\geq 0$, the spaces $H^i(G,\mathbb{V}_n)$ also admit an
Eisenstein/Cuspidal decomposition. For $i=1$ this is the content of the
above discussion. For $i>2$, these groups vanish because the cohomological
dimension of $G$ is $2$.

For $i=0$: if $n>0$, $H^0(G,\mathbb{V}_0) = \{ 0\}$ by \cite[p.\,162,
Prop.\,1; p.\,165, Lemma 2]{Hi} and there is nothing to prove; if $n=0$, the
action of the Hecke operators $T_{\ell}$ for $\ell \nmid N_G$ is given by
multiplication by $\ell+1$ and therefore 
\begin{equation}  \label{0}
H^0(G,\mathbb{V}_0)^{c} = \{ 0\}.
\end{equation}

For $i=2$: $H^2(G,\mathbb{V}_n)\simeq H_{c}^0(G,\mathbb{V}_n)^{\vee}$ by
Poincar\'e duality and the paragraph above applies. Here, the latter group
stands for the cohomology group with compact support of $G$ with
coefficients on $\mathbb{V}_n$. See e.g.\cite[Ch.\,III]{Fr} and \cite{MS}
for more details.
\end{remark}

\begin{definition}
\label{pnew} Let 
\begin{equation*}
\begin{matrix}
\mathrm{cor}: & H^1( \Gamma _0(pN^+),\mathbb{V}_n) & {\rightarrow } & 
H^1(\Gamma_0(N^+),\mathbb{V}_n) \\ 
\hat{\mathrm{cor}}: & H^1( \Gamma _0(pN^+),\mathbb{V}_n) & {\rightarrow } & 
H^1(\hat{\Gamma }_0(N^+),\mathbb{V}_n)%
\end{matrix}%
\end{equation*}
denote the corestriction maps induced by the inclusions $\Gamma_0(pN^+)
\subset \Gamma_0(N^+), \hat{\Gamma }_0(N^+)$ and let 
\begin{equation*}
H^1( \Gamma_0(pN^+),\mathbb{V}_n)^{p-new}:=\mathrm{Ker}(\mathrm{cor} \oplus 
\hat{\mathrm{cor}}).
\end{equation*}
\end{definition}

Similarly, we may define $H^1(\Gamma_0(pN^+),\mathbb{V}_n)^{p-old, c}:= 
\mathrm{Im}(\func{res} + \hat{\func{res}})^{c}$, where 
\begin{equation*}
H^1(\Gamma_0(N^+),\mathbb{V}_n)\oplus H^1(\Gamma_0(N^+),\mathbb{V}_n) \,\, 
\overset{\func{res} +\hat{\func{res}}}{{\longrightarrow }} \,\,
H^1(\Gamma_0(pN^+),\mathbb{V}_n)
\end{equation*}
is the sum of the natural restriction maps.

Obviously, over $\mathbb{C}$ the above corestriction maps admit a parallel
description purely in terms of modular forms and degeneracy maps, via %
\eqref{Diagram Eichler-Shimura isomorphism}. Via the above identifications,
the Petersson inner product induces on $H^1(\Gamma_0(pN^+),\mathbf{V}_n( 
\mathbb{C}))^{c}$ a perfect pairing with respect to which $%
H^1(\Gamma_0(pN^+),\mathbf{V}_n(\mathbb{C}))^{p-old, c}$ is the orthogonal
complement of $H^1( \Gamma_0(pN^+),\mathbf{V}_n(\mathbb{C}))^{p-new, c}$.

Let $K_p$ be a complete field extension of $\mathbb{Q}_p$. We call $H^1(
\Gamma_0(pN^+),\mathbb{V}_n)^{p-new}\otimes K_p = H^1( \Gamma_0(pN^+), 
\mathbf{V}_n(K_p))^{p-new}$ the \emph{space of $p$-new $p$-adic modular
forms }. It follows from the above discussion that 
\begin{equation}  \label{decomp}
H^1(\Gamma_0(pN^+),\mathbf{V}_n(K_p))^{c} = H^1(\Gamma_0(pN^+),\mathbf{V}
_n(K_p))^{p-old,c} \oplus H^1(\Gamma_0(pN^+),\mathbf{V}_n(K_p))^{p-new,c}.
\end{equation}

\subsection{The cohomology of $\Gamma$}

\label{S15}

Besides the relationship between $H^1(\Gamma_0(pN^+),\mathbb{V}_n)$ and
modular forms provided by the Eichler-Shimura isomorphism, these groups can
also be related to the cohomology of the group $\Gamma$ introduced in §\ref%
{S11} with values on the modules of functions on Bruhat-Tits's tree $%
\mathcal{T}_p$, as we now review.

The long exact sequence in cohomology arising from 
\eqref{exact sequence
edges/vertices} with $A=\mathbb{V}_n$ gives rise to an exact sequence of $%
\mathcal{H}(\Gamma)$-modules (cf.\,§\ref{S11}) 
\begin{equation}  \label{les}
...{\rightarrow } H^0(\Gamma,C(\mathcal{V},\mathbb{V}_n)) \,{\rightarrow }
\, H^1(\Gamma, C_{har}(\mathbb{V}_n)) \, \overset{s}{{\rightarrow }} \,
H^1(\Gamma,C_0(\mathcal{E},\mathbb{V}_n)) \, \rightarrow \, H^1(\Gamma,C( 
\mathcal{V},\mathbb{V}_n)).
\end{equation}

By Shapiro's lemma, for all $i\geq 0$ there are isomorphisms 
\begin{equation}  \label{shp}
\begin{matrix}
\mathscr{S}_{\mathcal{V}}: & H^i(\Gamma,C(\mathcal{V},\mathbb{V}_n)) & \simeq
& H^i(\Gamma_0(N^+),\mathbb{V}_n)^2, \\ 
\mathscr{S}_{\mathcal{E}}: & H^i(\Gamma,C_0(\mathcal{E},\mathbb{V}_n)) & 
\simeq & H^i(\Gamma_0(pN^+),\mathbb{V}_n),%
\end{matrix}%
\end{equation}
where throughout, by a slight abuse of notation, by $H^i(\Gamma_0(N^+), 
\mathbb{V}_n)^2$ we actually mean $H^i(\Gamma_0(N^+),\mathbb{V}_n)\oplus
H^i( \hat{\Gamma }_0(N^+),\mathbb{V}_n)$. Note that conjugation by ${\omega }
_p$ induces a canonical isomorphism 
\begin{equation}  \label{GahGa}
H^i(\Gamma_0(N^+),\mathbb{V}_n)\simeq H^i(\hat{\Gamma }_0(N^+),\mathbb{V}_n)
\end{equation}
which we will sometimes use in order to identify these two spaces without
further comment.

These isomorphisms are Hecke-equivariant in the following sense: for every
prime $\ell \nmid p N^+ N^-$, $T_{\ell}\cdot \mathscr{S}_{\mathcal{V}} = %
\mathscr{S}_{\mathcal{V}} \cdot T_{\ell }$ and $T_{\ell}\cdot \mathscr{S}_{ 
\mathcal{E}} = \mathscr{S}_{\mathcal{E}} \cdot T_{\ell }$. Although we are
using the same symbol for the Hecke operator at $\ell $ acting on the two
cohomology groups, note that they lie in the two different Hecke algebras $%
\mathcal{H}(\Gamma)$ and $\mathcal{H}(\Gamma_0(N^+))$ (resp., $\mathcal{H}
(\Gamma_0(pN^+))$). The compatibility with the isomorphism $\mathscr{S}$
follows from the key fact that $T_{\ell}$ can be defined in $\mathcal{H} (G)$
for $G=\Gamma_0(pN^+), \Gamma_0(N^+), \Gamma$ as a double-coset operator by
means of the \emph{same} choices of local representatives. Cf.\, \cite[Prop.
A.1]{Das}, \cite[§2.3]{LRV} for more details.

Remark \ref{RmkEC} and the isomorphisms of \eqref{shp} can be used to define
an Eisenstein/Cuspidal decomposition on $H^1(\Gamma,C_0(\mathcal{E},\mathbb{V%
}_n))$, $H^1(\Gamma,C(\mathcal{V},\mathbb{V}_n))$ and $H^1( \Gamma ,C_{har}( 
\mathbb{V}_n))$ by transporting it from $H^1(G,\mathbb{V}_n)$, where $%
G=\Gamma_0(pN^+)$, $\Gamma_0(N^+)$ or $\hat{\Gamma }_0(N^+)$.

\begin{lemma}
\label{Lemma harmonic cocycles/cusp forms} There is a Hecke equivariant
isomorphism 
\begin{equation*}
H^1( \Gamma ,C_{har}( \mathbb{V}_n))^{c}\overset{\simeq }{\rightarrow }H^1(
\Gamma _0(pN^+),\mathbb{V}_n) ^{p-new, c}\text{.}
\end{equation*}
\end{lemma}

\begin{proof}
Composing the map $s$ in \eqref{les} with Shapiro's isomorphism $\mathscr{S}
_{\mathcal{E}}$ in \eqref{shp}, we obtain a map 
\begin{equation}
H^1(\Gamma,C_{har}(\mathbb{V}_n)) \, \,\overset{\mathscr{S}_{\mathcal{E}}
\cdot s}{{\longrightarrow }} \, \, H^1(\Gamma_0(pN^+),\mathbb{V}_n),
\end{equation}
which we already argued to be Hecke equivariant. By Definition \ref{pnew}
and \eqref{shp}, $\mathscr{S}_{\mathcal{E}} \cdot s$ maps surjectively onto $%
H^1( \Gamma_0(pN^+),\mathbb{V}_n)^{p-new}$. By \cite[p.\,165]{Hi} , the $%
\Gamma_0(pN^+)$-module $\mathbf{V}_n(\mathbb{C} )$ is irreducible for $n>0$.
Hence $\mathbb{V}_n^{\Gamma_0(pN^+)}=0$ by \eqref{Vn}; since $\Gamma_0(pN^+)
\subset \Gamma_0(N^+), \hat{\Gamma }_0(N^+)$, the proposition now follows
for $n>0$ from the exactness of \eqref{les}, and \eqref{shp}.

When $n=0$, the action of $\mathcal{B}^{\times}$ on $\mathbb{V}_n$ is
trivial, whence $\mathbb{V}_n^{\Gamma_0(pN^+)}=\mathbb{V}_n$. Since $H^0(G, 
\mathbb{V}_n)^{c}=\{0\}$ both for $G=\Gamma_0(N^+)$ and $\hat{\Gamma }
_0(N^+) $ by \eqref{0}, the proposition follows as before.
\end{proof}

\vspace{0.3cm}

Let $K_p$ be a complete field extension of $\mathbb{Q}_p$. The long exact
sequence in cohomology arising from $( \text{\ref{exact sequence
vertices/edges}}) $ with $A=\mathbf{V}_n( K_p)$ is 
\begin{equation}  \label{long exact sequence vertices/edges}
\begin{matrix}
...{\rightarrow } \, H^1(\Gamma_0(N^+),\mathbf{V}_n(K_p))^2 \rightarrow H^1(
\Gamma _0( pN^+) ,\mathbf{V}_n(K_p)) \, \overset{\delta }{\rightarrow } ...
\\ 
\\ 
\overset{\delta }{\rightarrow } H^{2}( \Gamma ,\mathbf{V}_n(K_p)) \overset{
\varepsilon }{\rightarrow } H^{2}( \Gamma_0(N^+),\mathbf{V}_n(K_p))^2
\rightarrow H^{2}( \Gamma _0( pN^+) ,\mathbf{V}_n(K_p)),%
\end{matrix}%
\end{equation}
once we apply the isomorphisms of \eqref{shp}. Exactly as in \eqref{les},
all maps in \eqref{long exact sequence vertices/edges} are Hecke equivariant
and admit an Eisentein/Cuspidal decomposition.

\begin{lemma}
\label{Lemma p-new and p-old subspaces}The boundary map $\delta ^{c}$
restricts to an isomorphism 
\begin{equation*}
\delta ^{c}:\,H^1( \Gamma _0(pN^+),\mathbf{V}_n(K_p)) ^{p-new,c}\,\overset{
\simeq }{\rightarrow }\,H^{2}( \Gamma ,\mathbf{V}_n( K_p) )^{c}\text{.}
\end{equation*}
\end{lemma}

\begin{proof}
We have $H^{2}( \Gamma_0(N^+),\mathbf{V}_n(K_p))^{c}=0$ by \eqref{0}. Remark %
\ref{RmkEC} implies that \emph{taking cuspidal parts} in an exact functor.
It thus follows from \eqref{long exact sequence vertices/edges} that there
is an exact sequence 
\begin{equation}  \label{long exact sequence vertices/edges cuspidal parts}
(H^1(\Gamma_0(N^+),\mathbf{V}_n(K_p))^{c})^2\rightarrow H^1( \Gamma _0(
pN^+) ,\mathbf{V}_n(K_p))^{c}\,\,\overset{\delta ^{c}}{\rightarrow }
\,\,H^{2}( \Gamma ,\mathbf{V}_n(K_p)) ^{c}\rightarrow 0\text{.}
\end{equation}

The lemma now follows from \eqref{decomp}.
\end{proof}

\subsection{Morita-Teitelbaum's integral representations}

\label{S16}

Recall the $\mathrm{GL}_{2}$-module $\mathbf{P}_n$ and note that it admits
as a natural $\mathbb{Z}$-structure the free $\mathbb{Z}$-module $\mathbf{P}
_{n,\mathbb{Z}}$ of polynomials of degree at most $n$ with integers
coefficients. For any $\mathbb{Z}$-algebra $R$, $\mathbf{P}_n(R):=\mathbf{P}
_{n,\mathbb{Z}}\otimes R$ is endowed with a right $\mathrm{GL}_{2}( R) $
-action by the same formula. This way, $\mathbf{V}_n(R)=\mathbf{P}_n^{\vee
}( R) :=\Hom_{R}(\mathbf{P}_n(R),R)$, the dual of $\mathbf{P}_n( R) $,
inherits a left $\mathrm{GL}_{2}( R) $-action.

For any vertex $v\in \mathcal{V}$, choose any element $\gamma _{v}\in {\ 
\mathrm{GL}}_{2}(\mathbb{Q}_p)$ such that $\gamma _{v}(v)=v_{\ast }$ and set 
\begin{equation*}
\mathbf{V}_{n,v}(\mathbb{Z}_p):=\gamma _{v}^{-1}\cdot \mathbf{V}_n(\mathbb{Z}
_p) \subset \mathbf{V}_n(\mathbb{Q}_p)\text{.}
\end{equation*}

Notice that this definition does not depend on the choice of $\gamma _{v}$,
because the stabilizer of $v_{\ast }$ in ${\mathrm{PGL}}_{2}(\mathbb{Q}_p)$
is ${\mathrm{PGL}}_{2}(\mathbb{Z}_p)$, which leaves $\mathbf{V} _n(\mathbb{Z}
_p)$ invariant. Define 
\begin{equation*}
C^{int}(\mathcal{V},\mathbf{V}_n(\mathbb{Q}_p)):=\{c\in C(\mathcal{V}, 
\mathbf{V}_n(\mathbb{Q}_p)):c(v)\in \mathbf{V}_{n,v}\,\forall v\in \mathcal{%
V }\}.
\end{equation*}

Similarly, for any oriented edge $e=(v,v^{\prime })\in \mathcal{E}$, set $%
\mathbf{V}_{n,e}(\mathbb{Z}_p):=\mathbf{V}_{n,v}(\mathbb{Z}_p)\cap \mathbf{V}
_{n,v^{\prime }}(\mathbb{Z}_p)\subset \mathbf{V}_n(\mathbb{Q}_p)$ and define 
\begin{equation*}
C^{int}(\mathcal{E},\mathbf{V}_n(\mathbb{Q}_p)):=\{c\in C(\mathcal{E}, 
\mathbf{V}_n(\mathbb{Q}_p)):c(e)\in \mathbf{V}_{n,e}(\mathbb{Z}_p)\,\forall
e\in \mathcal{E}\},
\end{equation*}
which is naturally a $\mathbb{Z}_p$-module.

Introduce also the $\mathbb{Z}_p$-modules 
\begin{equation*}
C_0^{int}(\mathcal{E},\mathbf{V}_n(\mathbb{Q}_p)):=C^{int}(\mathcal{E}, 
\mathbf{V}_n(\mathbb{Q}_p))\cap C_0(\mathcal{E},\mathbf{V}_n(\mathbb{Q}_p))
\end{equation*}
and 
\begin{equation*}
C_{har}^{int}(\mathcal{E},\mathbf{V}_n(\mathbb{Q}_p)):=C^{int}(\mathcal{E}, 
\mathbf{V}_n(\mathbb{Q}_p))\cap C_{har}(\mathbf{V}_n(\mathbb{Q}_p)).
\end{equation*}

The next result of Teitelbaum should be regarded as a refinement of the
exact sequence $( \text{\ref{exact sequence edges/vertices}}) $.

\begin{proposition}
\cite[p.\,564-566]{Te2} \label{t93} For every \emph{even} integer $n\geq 0$
the natural sequence 
\begin{equation}  \label{esint}
0\,{\rightarrow }\,C_{har}^{int}( \mathbf{V}_n(\mathbb{Q}_p)) \,{\rightarrow 
}\,C_0^{int}(\mathcal{E},\mathbf{V}_n(\mathbb{Q}_p))\,{\rightarrow }
\,C^{int}(\mathcal{V},\mathbf{V}_n(\mathbb{Q}_p))\,{\rightarrow }\,0
\end{equation}
is an exact sequence of $\mathrm{PGL}_{2}( \mathbb{Q}_p) $-modules.
\end{proposition}

In particular we may regard the above sequence as an exact sequence of $%
\Gamma $-modules by means of the identification $\iota _p:\,\mathcal{B}
_p\,\simeq \,\mathrm{M}_{2}( \mathbb{Q}_p) $.

\bigskip

As a piece of notation, by \emph{extended norm} on a space $A$ we mean a
function $\Vert \cdot \Vert: A{\rightarrow } {\mathbb{R}}_{\geq 0} \cup
\{+\infty \}$ satisfying the usual properties of a norm, extended in a
natural way to the semigroup of values ${\mathbb{R}}_{\geq 0} \cup \{+\infty
\}$.

Let $K_p/\mathbb{Q}_p$ be a complete field extension, with ring of integers $%
R_p$, which we fix throughout the article.

Let $\left \vert\cdot \right\vert$ denote the absolute value of $K_p$. Let $%
\left\vert \cdot \right \vert _{L_{\ast }}$ and $\vert \cdot \vert _{ 
\widehat{L}_{\ast }}$ be two norms on $\mathbf{P}_n(K_p)$. We require the
first one to be $\mathrm{GL}_{2}(L_{\ast })=\mathrm{GL}_{2}(\mathbb{Z}_p)$
-invariant and the second one to be $\mathrm{GL}_{2}(\widehat{L}_{\ast })$
-invariant (cf.\,§\ref{S12} for notations). Choose also a $\mathrm{GL}
_{2}(L_{\ast })\cap \mathrm{GL}_{2}( \widehat{L}_{\ast }) $-invariant norm $%
\left\vert \cdot \right\vert _{L_{\ast },\widehat{L}_{\ast }}$ on $\mathbf{P}
_n(K_p)$. We may choose for example $\left\vert \cdot \right\vert _{L_{\ast
},\widehat{L}_{\ast }} = \left\vert \cdot\right\vert _{L_{\ast
}}:=\left\vert \cdot \right\vert $ to be the supremum of the absolute values
of the coefficients of a polynomial and set $\left\vert \cdot\right\vert _{ 
\widehat{L}_{\ast }}:=\left\vert \omega _p^{-1}\cdot \right\vert _{L_{\ast
}} $. By duality we can consider the corresponding norms on $\mathbf{V}
_n(K_p)$.

Define extended norms on $C(\mathcal{V}^+,\mathbf{V}_n(K_p))$, $C(\mathcal{V}
^-,\mathbf{V}_n(K_p))$ and $C_0(\mathcal{E},\mathbf{V}_n (K_p))$ by the
rules 
\begin{equation*}
\left\Vert c\right\Vert _+:=\mathrm{sup}_{v\in \mathcal{V}^+}|\gamma
_{v}\cdot c(v)|_{L_{\ast }},
\end{equation*}
\begin{equation*}
\left\Vert c\right\Vert _-:=\mathrm{sup}_{v\in \mathcal{V}^-}|\gamma
_{v}\cdot c(v)|_{L_{\ast }}.
\end{equation*}
\begin{equation*}
\left\Vert c\right\Vert := \mathrm{sup}_{e\in \mathcal{E}^+}|\gamma _{e}
\cdot c(e)|.
\end{equation*}

Here, $\gamma_v$ (resp. $\gamma _{e}$) is any element in $\mathrm{GL}_{2}( 
\mathbb{Q}_p)$ such that $\gamma _{v}(v)=v_{\ast }$ (resp. $\gamma
_{e}(e)=e_{\ast }$). The invariance properties of the above norms imply that
the above definitions do not depend on the choice of the sets $\left\{
\gamma _{v}\right\} $ and $\left\{ \gamma _{e}\right\} $.

\vspace{0.4cm}

Let $C(K_p)$ denote either $C(\mathcal{V},\mathbf{V}_n(K_p))):=C(\mathcal{V}
^+,\mathbf{V}_n(K_p))\oplus C(\mathcal{V}^-,\mathbf{V}_n(K_p))$, $C_0( 
\mathcal{E},\mathbf{V}_n(K_p))$ or $C_{har}(\mathbf{V}_n(K_p))$ and write $%
C^{b}(K_p):=\{c\in C:\left\Vert c\right\Vert <\infty \}$. In all these cases
the restriction of $\left\Vert \cdot \right\Vert $ to $C^{b}(K_p)$ is a norm
with respect to which $C^{b}(K_p)$ is a Banach space over $K_p$.

\begin{lemma}
\label{bint}

\begin{enumerate}
\item[(i)] $C^{b}(K_p)=C^{b}(\mathbb{Q}_p)\hat{\otimes} K_p$.

\item[(ii)] $C^{b}(\mathbb{Q}_p)=C^{int}(\mathbb{Q}_p)\otimes_{\mathbb{Z}_p} 
\mathbb{Q}_p$.
\end{enumerate}
\end{lemma}

\begin{proof}
We sketch a proof only for $\mathcal{V}$, as the remaining cases work
similarly. For (i), there is in fact a natural inclusion $C^{b}(\mathbb{Q}
_p) \hat{\otimes} K_p{\hookrightarrow } C^{b}(K_p)$, which we wish to show
that is an isomorphism. Enumerate the set of vertices: $\mathcal{V}=\{ v_1,
..., v_r,...\}$. Given $c\in C^{b}(K_p)$, define $c_r$ by $c_r(v)=c(v)$ if $%
v\in \{ v_1, ..., v_r\}$; $c_r(v)=0$ otherwise. Note that $c_r\in C^{b}( 
\mathbb{Q} _p)\otimes K_p\subset C^{b}(\mathbb{Q}_p)\hat{\otimes} K_p$ and
that the sequence $\{ c_r\}$ converges to $c$. This shows (i).

As for (ii), the inclusion $C^{int}(\mathbb{Q}_p)\otimes_{\mathbb{Z}_p} 
\mathbb{Q}_p\subset C^{b}(\mathbb{Q}_p)$ follows from the definitions: given 
$c\in C^{int}(\mathbb{Q}_p)$, $\left\Vert c\right\Vert = \sup_v \left\vert
\gamma_v\cdot c(v)\right \vert$ is bounded because $\gamma_v\cdot c(v)\in 
\mathbf{V}_n(\mathbb{Z}_p)$. As for the opposite inclusion, let $c\in C( 
\mathbb{Q}_p)$ be such that $\left\Vert c\right\Vert = B<\infty $. Then $c$
can be replaced by a scalar multiple of it such in a way that $\left\Vert
c\right\Vert = \sup_v \left\vert \gamma_v\cdot c(v)\right \vert\leq 1$. This
implies that $c\in C^{int}(\mathbb{Q}_p)$.
\end{proof}

Next corollary now follows from Proposition \ref{t93} and Lemma \ref{bint}.

\begin{corollary}
\label{Corollary fundamental exact sequence} For every \emph{even} integer $%
n\geq 0$ the natural sequence 
\begin{equation}
0\,{\rightarrow }\,C_{har}^{b}( \mathbf{V}_n(K_p)) \,{\rightarrow }
\,C_0^{b}( \mathcal{E},\mathbf{V}_n(K_p))\,{\rightarrow }\,C^{b}(\mathcal{V}
, \mathbf{V} _n(K_p))\,{\rightarrow }\,0
\end{equation}
is an exact sequence of $\mathrm{PGL}_{2}( \mathbb{Q}_p) $-modules.
\end{corollary}

Again we may regard the above sequence as an exact sequence of $\Gamma $
-modules by means of the identification $\iota _p:\,\mathcal{B}_p\,\simeq \, 
\mathrm{M}_{2}( \mathbb{Q}_p) $

\section{$p$-adic integration and a ${\mathcal{L}}$-invariant}

\label{S2}

\subsection{The cohomology of distributions and harmonic cocyles}

\label{S21}

Let $\mathcal{H}_p$ denote the $p$-adic upper half-plane over $\mathbb{Q}_p$
. It is a rigid analytic variety over $\mathbb{Q}_p$ such that $\mathcal{H}
_p(K_p)=K_p\setminus \mathbb{Q}_p$ for every complete field extension $K_p/ 
\mathbb{Q}_p$. Let $\mathcal{O}_{\mathcal{H}_p}$ denote the ring of entire
functions of $\mathcal{H}_p$, which is a Fr\'echet space over $\mathbb{Q}_p$
(cf.\,\cite[Prop.\,1.2.6]{DT}). For any even integer $k\geq 2$, write ${\ 
\mathcal{O}}_{\mathcal{H}_p}(k)$ for the ring ${\mathcal{O}}_{\mathcal{H}_p}$
equipped with the action of $\mathrm{GL}_2(\mathbb{Q}_p)$ given by 
\begin{equation*}
f|\gamma = \frac{\det(\gamma)^{k/2}}{(cz+d)^k}\cdot f(\gamma z), \text{ for }
\gamma= 
\begin{pmatrix}
a & b \\ 
c & d%
\end{pmatrix}
, f\in {\mathcal{O}}_{\mathcal{H}_p}.
\end{equation*}

Let $\mathcal{H}^{int}_p$ denote the formal scheme over $\mathbb{Z}_p$
introduced by Mumford in \cite{Mu} (cf.\,also \cite[p.\,567]{Te2}). The
rigid analytic space associated with its generic fiber is $\mathcal{H}_p$.
The dual graph of its special fiber is the tree $\mathcal{T} _p $ (cf.\,\cite%
{DT} and \cite{Te2} for a detailed discussion).

Let $\omega_{\mathbb{Z}_p}$ denote the sheaf on $\mathcal{H}^{int}_p$
introduced in \cite[Def.\,10]{Te2}, such that $\omega_{\mathbb{Z}_p}\otimes 
\mathbb{Q}_p$ is the sheaf $\omega $ of rigid analytic differential forms on 
$\mathcal{H}_p$. The map $f(z)\mapsto f(z) dz^{k/2}$ induces an isomorphism
of ${\mathrm{PGL}}_2(\mathbb{Q}_p)$-modules between $\mathcal{O}_{\mathcal{H}
_p}(k)$ and $H^0(\mathcal{H}_p, \omega^{k/2})$. Set 
\begin{equation*}
\mathcal{O}^b_{\mathcal{H}_p}(k):=H^0(\mathcal{H}^{int}_p, \omega_{\mathbb{Z}
_p}^{k/2})\cdot dz^{-k/2}\otimes_{\mathbb{Z}_p}\mathbb{Q}_p \,\subset \, 
\mathcal{O}_{\mathcal{H}_p}(k).
\end{equation*}

Recall that $n:=k-2\geq 0$. As it follows from e.g. \cite[§2.2.4]{DT} or 
\cite[Theorem 15]{Te2}, the residue map on $\mathcal{O}_{\mathcal{H}_p}(k)$
yields an epimorphism of ${\mathrm{PGL}}_{2}(\mathbb{Q}_p)$-modules 
\begin{equation}  \label{res}
\mathrm{Res}:\mathcal{O}_{\mathcal{H}_p}(k)\,\twoheadrightarrow \,C_{har}( 
\mathbf{V}_n(\mathbb{Q}_p) ) \text{.}
\end{equation}

The following deep result is proved in \cite[p.\,569-574]{Te2}, and will be
crucial for our purposes.

\begin{proposition}
\label{tei} The map \emph{Res} restricts to an isomorphism of ${\mathrm{PGL}}
_2(\mathbb{Q}_p)$-modules 
\begin{equation}
\mathrm{Res}: \mathcal{O}^b_{\mathcal{H}_p}(k) \overset{\sim}{{\
\longrightarrow } }C^b_{har}(\mathbf{V}_n(\mathbb{Q}_p)).
\end{equation}
\end{proposition}

As in the previous section, fix a complete field extension $K_p/\mathbb{Q}_p$
.

\begin{definition}
Let $\mathcal{A}_n( \mathbb{P}^1( \mathbb{Q}_p), K_p)$ be the space of $K_p$
-valued locally analytic functions on $\mathbb{P}^1( \mathbb{Q}_p) $ with a
pole of order at most $n$ at $\infty $. More precisely, an element $f\in 
\mathcal{A}_n$ is a locally analytic function $f: \mathbb{Q}_p {\rightarrow }
K_p$ for which there exists an integer $N$ such that $f$ is locally analytic
on $\{ z\in \mathbb{Q}_p: {\func{ord}}_p(x)\geq N\}$ and admits a convergent
expansion 
\begin{equation*}
f(z)=a_n z^n + a_{n-1}z^{n-1} + ... +a_0+\sum_{r\geq 1} a_{-r}z^{-r}
\end{equation*}
on $\{ z\in \mathbb{Q}_p: {\func{ord}}_p(z)<N\}$.
\end{definition}

The space $\mathcal{A}_{n}(\mathbb{P}^{1}(\mathbb{Q}_{p}),K_{p})$ carries a
right action of ${\mathrm{GL}}_{2}(\mathbb{Q}_{p})$ defined by the rule $%
f\cdot \gamma =\frac{(cx+d)^{n}}{\det (\gamma )^{n/2}}\cdot f(\frac{ax+b}{
cx+d})$, for any $f\in \mathcal{A}_{n}(\mathbb{P}^{1}(\mathbb{Q}_{p}),K_{p})$
and $\gamma =( 
\begin{smallmatrix}
a & b \\ 
c & d%
\end{smallmatrix}
)\in {\mathrm{GL}}_{2}(\mathbb{Q}_{p})$. Note that $\mathbf{P}_{n}(K_{p})$
is a natural ${\mathrm{GL}}_{2}(\mathbb{Q}_{p})$-submodule of it.

\begin{definition}
Write $\mathcal{D}_n( \mathbb{P}^1(\mathbb{Q}_p),K_p) $ and $\mathcal{D}
_n^0( \mathbb{P}^1(\mathbb{Q}_p),K_p)$ for the strong continuous dual of $%
\mathcal{\ A}_n( \mathbb{P}^1( \mathbb{Q}_p), K_p )$ and of its quotient by $%
\mathbf{P} _n(K_p)$, respectively.
\end{definition}

These modules of distributions inherit from $\mathcal{A}_n( \mathbb{P}^1( 
\mathbb{Q}_p), K_p )$ a left action of ${\mathrm{GL}}_{2}(\mathbb{Q}_p)$. As
explained in \cite[§2.1.1]{DT}, $\mathcal{D}_n^0(\mathbb{P}^1(\mathbb{Q}
_p),K_p)$ is a Fr\'echet space over $K_p$ and $\mathcal{D}_n^0(\mathbb{P}^1( 
\mathbb{Q}_p),K_p)=\mathcal{D}_n^0(\mathbb{P}^1(\mathbb{Q}_p),\mathbb{Q}_p) 
\hat{\otimes}K_p$.

Morita's (or sometimes also called Schneider-Teitelbaum) duality, yields an
isomorphism $\mathcal{D}^0_n(\mathbb{P}^1(\mathbb{Q}_p),\mathbb{Q}_p) 
\overset{\sim}{{\longrightarrow } } \mathcal{O}_{\mathcal{H}_p}(k)$ which
induces an isomorphism 
\begin{equation}  \label{I}
I: \mathcal{D}^0_n(\mathbb{P}^1(\mathbb{Q}_p),K_p) \overset{\sim}{{\
\longrightarrow } } \mathcal{O}_{\mathcal{H}_p}(k)\hat{\otimes }K_p.
\end{equation}

We refer the reader to \cite[§2.2]{DT} for more details; in the sequel, we
shall freely identify these two spaces. Define 
\begin{equation*}
\mathcal{D}^0_n(\mathbb{P}^1(\mathbb{Q}_p),K_p)^b:=I^{-1}(\mathcal{O}^b_{ 
\mathcal{H}_p}(k)\hat{\otimes}K_p)\subset \mathcal{D}^0_n(\mathbb{P}^1( 
\mathbb{Q}_p),K_p).
\end{equation*}

\begin{remark}
As a consequence of a variant of the theorem of Amice-Velu-Vishik
(cf.\,e.g.\,\cite[Theorem 2.3.2]{DT}), the space $\mathcal{D}^0_n(\mathbb{P}
^1(\mathbb{Q}_p),K_p)^b$ can alternatively be described as the subspace of
distributions $\mu \in \mathcal{D}^0_n(\mathbb{P}^1(\mathbb{Q}_p),K_p)$ for
which there is a constant $A$ such that, for all $i\geq 0$, $j\geq 0$, and
all $a\in \mathbb{Z}_p$, 
\begin{equation*}
|\mu((x-a)^i | a+p^j\mathbb{Z}_p)|\leq p^{ A - j(i-1-k/2)}.
\end{equation*}

A distribution $\mu $ satisfying the above condition is then completely
determined.
\end{remark}

Lemma \ref{bint} and \eqref{I} yields an epimorphism of $\mathrm{GL}_{2}( 
\mathbb{Q}_p)$-modules 
\begin{equation*}
r:\mathcal{D}_n^0(\mathbb{P}^1(\mathbb{Q}_p),K_p)\,\twoheadrightarrow
\,C_{har}( \mathbf{V}_n(K_p)) \text{,}
\end{equation*}
which can be described purely in terms of distributions by the rule 
\begin{equation*}
r( \mu ) ( e) ( P) =\int\nolimits_{U_{e}}P( t) d\mu ( t) :=\mu ( P\cdot \chi
_{U_{e}}) \text{.}
\end{equation*}

Here $U_{e}\subset \mathbb{P}^1( \mathbb{Q}_p) $ is the open compact subset
of $\mathbb{P}^1( \mathbb{Q}_p)$ corresponding to the ends leaving from the
oriented edge $e$, and $\chi _{U_{e}}$ stands for its characteristic
function.

By Proposition \ref{tei} and Lemma \ref{bint}, it restricts to an
isomorphism 
\begin{equation}  \label{Morita}
r:\mathcal{D}_n^0(\mathbb{P}^1(\mathbb{Q}_p),K_p)^{b}\overset{\sim }{
\longrightarrow }C_{har}^{b}(\mathbf{V}_n(K_p)),
\end{equation}
which by abuse of notation we denote with the same symbol $r$. The same
abuse will be made for the several maps that $r$ induces in cohomology, as
below.

The following theorem is the basic piece which shall allow us to introduce a 
$p$-adic integration theory on indefinite quaternion algebras.

\begin{theorem}
\label{2parts} There is a commutative diagram of morphisms of Hecke-modules 
\begin{equation*}
\begin{matrix}
H^1(\Gamma,\mathcal{D}_n^0(\mathbb{P}^1(\mathbb{Q}_p))^b) & {\longrightarrow 
} & H^1(\Gamma,\mathcal{D}_n^0(\mathbb{P}^1(\mathbb{Q}_p))) \\ 
\downarrow &  & \downarrow \\ 
H^1(\Gamma,C_{har}^{b}(\mathbf{V}_n(K_p))) & {\longrightarrow } & H^1(
\Gamma ,C_{har}( \mathbf{V}_n(K_p))) \\ 
&  & 
\end{matrix}%
\end{equation*}
such that the composition $r: H^1(\Gamma,\mathcal{D}_n^0(\mathbb{P}^1( 
\mathbb{Q}_p))^b) \overset{\sim}{{\longrightarrow }} H^1(\Gamma,C_{har}( 
\mathbf{V}_n(K_p)))$ is an isomorphism.
\end{theorem}

In the statement, by Hecke-modules we mean modules over the Hecke algebra $%
\mathcal{H}(\Gamma)$ introduced in §\ref{S11}. Since $r$ and the natural
inclusions $\mathcal{D}_n^0(\mathbb{P}^1(\mathbb{Q}_p))^b)\subset \mathcal{D}
_n^0(\mathbb{P}^1(\mathbb{Q}_p)))$ and $C_{har}^{b}(\mathbf{V}_n(K_p))
\subset C_{har}( \mathbf{V}_n(K_p))$ are morphisms of $\mathrm{GL}_2(\mathbb{%
Q}_p)$-modules, it follows from the discussion around \eqref{Hecke} that
there indeed exists a commutative diagram as above, where the maps are
morphisms of $\mathcal{H}(\Gamma)$-modules. Notice that, by \eqref{Morita},
it suffices to show that the inclusion $C_{har}^{b}( \mathbf{V}
_n(K_p))\subset C_{har}(\mathbf{V}_n(K_p))$ induces an isomorphism 
\begin{equation}  \label{iso}
H^1(\Gamma,C_{har}^{b}(\mathbf{V}_n(K_p))) \simeq H^1(\Gamma,C_{har}(\mathbf{%
\ V}_n(K_p))).
\end{equation}

We devote the rest of the section to prove this statement. In order to prove %
\eqref{iso} we need a further preliminary discussion. Quite generally, let ${%
\mathcal{S}}$ be a set on which $\Gamma $ acts transitively. Fix an element $%
s_{\ast}\in{\mathcal{S}}$ and let $\Gamma_0\subset \Gamma $ denote its
stabilizer in $\Gamma $, which we assume to be finitely generated. Let $\{
\gamma _{s}\}_{s\in \mathcal{S}}$ be a set of representatives for the coset
space $\Gamma_0\backslash \Gamma $ such that $\gamma_{s_{\ast}}=1$ and $%
\gamma_{s}s=s_{\ast }$ for all $s\in {\mathcal{S}}$. Let $A$ be $\Gamma $
-module endowed with a $\Gamma_0$-invariant non-archimedean norm $\left\vert
\cdot\right\vert $ with values in $K_p$. On the group of functions $C( 
\mathcal{S},A)$, define an extended norm by the rule 
\begin{equation*}
\left\Vert c\right\Vert :=\sup_{s\in \mathcal{S}}\left\vert \gamma _{s}
(c(s)) \right\vert = \sup_{s\in \mathcal{S}}\left\vert ({}^{\gamma _{s}}
c)(s_{\ast}) \right\vert\text{.}
\end{equation*}

As before, let $C^{b}(\mathcal{S},A)=\{c\in C({\mathcal{S}},A): \Vert c\Vert
<\infty \}$. It is a $\Gamma$-submodule of $C({\mathcal{S}},A)$ and the
restriction of norm $\Vert \cdot \Vert$ to $C^{b}(\mathcal{S},A)$ turns out
to be $\Gamma$-invariant, as can be easily checked. Note that the subspace $%
C^{b}(\mathcal{S},A)$ does not depend on the choice of the set $\{\gamma_s\}$
.

\begin{lemma}
\label{Lemma boundness}Let $G$ be a finitely generated group and let $M$ be
a $G$-module endowed with a $G$-invariant non-archimedean norm $[ \cdot ] $
with values in $K_p$. Then 
\begin{equation*}
[ [ c] ] :=\sup_{g\in G}[ c( g) ]
\end{equation*}
defines a norm on $Z^1(G,M)$.
\end{lemma}

\begin{proof}
Let $\left\{ g_i:i\in I\right\} $ be a set of generators of $G$ with $%
\#I<\infty $. Given an element $c\in Z^1( G,B) $ define $K_{c}:=\sup
\limits_i\{[ c( g_i) ],[c( g_i^{-1})]\in {\mathbb{R}} _{\geq 0}$. Let $g\in
G $ be any element. Then $g=g_{i_1}^{\varepsilon
_1}...g_{i_{k}}^{\varepsilon _{k}}$ for some $i_j\in I$ and $%
\varepsilon_j\in \left\{ \pm 1\right\} $. Let us show by induction on $k$
that $[ c( g) ] \leq K_{c}$. When $k=1$ this is clear. When $k>1$ , the
cocyle relation $c(g) =c(g_{i_1}^{\varepsilon _1}) +g_{i_1}^{\varepsilon
_1}c( g_{i_{2}}^{\varepsilon_{2}}...g_{i_{k}}^{\varepsilon _{k}})$, together
with the $G$-invariance of $[ \cdot] $, imply that 
\begin{eqnarray*}
[ c( g) ] &\leq &\max \left\{ [ c( g_{i_1}^{\varepsilon _1}) ] ,[
g_{i_1}^{\varepsilon _1}c( g_{i_{2}}^{\varepsilon
_{2}}...g_{i_{k}}^{\varepsilon _{k}}) ] \right\} = \\
&=&\max \left\{ [ c( g_{i_1}^{\varepsilon _1}) ] ,[ c(
g_{i_{2}}^{\varepsilon _{2}}...g_{i_{k}}^{\varepsilon _{k}}) ] \right\} \leq
K_{c}\text{.}
\end{eqnarray*}
\end{proof}

\begin{proposition}
\label{Proposition boundness} For $i=0,1$ the inclusion $\iota :C^{b}( 
\mathcal{S},A) \subset C( \mathcal{S},A) $ induces an isomorphism 
\begin{equation*}
\iota :H^i( \Gamma ,C^{b}( \mathcal{S},A) ) \overset{\simeq }{\rightarrow }
H^i( \Gamma,C( \mathcal{S},A))\text{.}
\end{equation*}
\end{proposition}

\begin{proof}
Let $i=0$. We wish to show that every $\Gamma $-invariant element of $C( 
\mathcal{S},A) $ has bounded norm. By Shapiro's lemma there is an
isomorphism $\mathscr{S}: C({\mathcal{S}},A)^{\Gamma}\overset{\simeq}{{\
\longrightarrow } } A^{\Gamma_0}$, whose inverse is given explicitly by the
map $a\mapsto c_{a}$, where $c_a( s)=\gamma _{s}^{-1}a$. One checks from
this description that $\mathscr{S}$ is an isometry. Since $\left\vert
\cdot\right\vert $ is a norm on $A^{\Gamma _0}\subset A$, the proposition
follows.

Assume now $i=1$. Again by Shapiro's lemma, the natural map 
\begin{equation*}
\pi:Z^1(\Gamma ,C({\mathcal{S}},A)) \rightarrow Z^1(\Gamma_0,A)
\end{equation*}
induces an isomorphism $[\pi]: H^1(\Gamma ,C({\mathcal{S}},A))\simeq
H^1(\Gamma_0,A)$.

Let us first construct an explicit section 
\begin{equation*}
\tau:Z^1(\Gamma_0,A)\rightarrow Z^1(\Gamma,C(\mathcal{S},A))
\end{equation*}
of $\pi$ with values in the submodule $Z^1(\Gamma,C^{b}({\mathcal{S}},A))$
of $Z^1(\Gamma,C(\mathcal{S},A))$.

Given a cocyle $c\in Z^1(\Gamma_0,A) $ define a chain $\tau_c(\gamma,s)
:=\gamma _{s}^{-1}c( g_{\gamma ,s})$, where $\gamma _{s}\gamma
=g_{\gamma,s}\gamma _{s^{\prime }}$, with $g_{\gamma ,s}\in \Gamma_0$ and $%
s^{\prime}\in {\mathcal{S}}$. An elementary verification shows that $%
\tau_c\in Z^1(\Gamma,C({\mathcal{S}},A))$ is well-defined and that $\pi \tau
= \func{Id}$. Moreover, the morphism $[\tau]: H^1(\Gamma_0,A)\rightarrow
H^1(\Gamma,C(\mathcal{S},A))$ that $\tau$ induces in cohomology is an
explicit inverse of the isomorphism $[\pi ]$.

Let us prove that $\tau_c\in Z^1(\Gamma ,C^{b}(\mathcal{S},A))$. Since $%
\Gamma_0$ is finitely generated and the norm $\left\vert \cdot\right\vert$
is $\Gamma_0$-invariant, it follows from Lemma \ref{Lemma boundness} applied
to $(G,M,[ \cdot] ) =(\Gamma_0,A,\left\vert \cdot\right\vert )$ that there
exists a constant $K\geq 0$ such that 
\begin{equation*}
\left\vert c(g)\right\vert \leq K\text{ for all }g\in \Gamma_0\text{.}
\end{equation*}

It then follows that for all $\gamma \in \Gamma $:

\begin{eqnarray*}
\left\Vert \tau_c(\gamma,\cdot) \right\Vert & = & \sup_{s\in \mathcal{S}
}\left\vert \gamma _{s}\cdot \tau_c (\gamma ,s) \right\vert =\sup_{\substack{
s\in \mathcal{S}  \\ \gamma _{s}\gamma =g_{\gamma ,s}\gamma _{s^{\prime }}}}
\left\vert \gamma _{s}\gamma _{s}^{-1}c( g_{\gamma ,s}) \right\vert = \\
&=&\sup_{\substack{ s\in \mathcal{S}  \\ \gamma _{s}\gamma =g_{\gamma
,s}\gamma _{s^{\prime }}}}\left\vert c( g_{\gamma ,s}) \right\vert \leq
\sup_{g\in \Gamma_0}\left\vert c(g)\right\vert \leq K\text{.}
\end{eqnarray*}

We can now easily prove that $\iota $ is surjective. Indeed, let $[\tilde{c}
] $ denote the class of a cocycle $\tilde{c}\in Z^1( \Gamma ,C( \mathcal{S}
,A)) $. Set $c:=\pi (\tilde{c})\in Z^1(\Gamma_0,A)$. By the above
discussion, $\tau (c)\in Z^1(\Gamma,C^{b}({\mathcal{S}},A))$ and $[\tilde{c}
]=[ \iota(\tau (c))]$.

In order to prove that $\iota $ is injective, let $[\tilde{c}]\in H^1(
\Gamma ,C^{b}( \mathcal{S},A) ) $. Note that Lemma \ref{Lemma boundness}
applied to $(G,M,[ \cdot] ) =(\Gamma,C^{b}( \mathcal{S},A) ,\left\Vert
\cdot\right\Vert ) $ yields the existence of a constant $K\geq 0$ such that 
\begin{equation}
\sup_{\gamma \in \Gamma }\left\Vert \tilde{c}( \gamma ) \right\Vert
=\sup_{\gamma \in \Gamma ,s\in \mathcal{S}}\left\vert \gamma _{s}\tilde{c}(
\gamma ,s) \right\vert \leq K\text{.}  \label{Equation boundness}
\end{equation}

Suppose that the class of $\tilde{c}$ vanishes in $H^1( \Gamma ,C( \mathcal{%
\ \ \ S },A))$, that is, there exists $C\in C(\mathcal{S},A) $ such that $%
c(\gamma )=C-\gamma C$ for all $\gamma \in \Gamma $. Equivalently, for all $%
s\in \mathcal{S}$ we have 
\begin{equation*}
c(\gamma ,s)=C( s) -\gamma C( \gamma ^{-1}s) \text{.}
\end{equation*}

If $C$ were not bounded, there would exist a sequence $\left\{ s_n\right\}
\subset \mathcal{S}$ such that $|\gamma _{s_n}C(s_n)|$ $\rightarrow $ $%
\infty $. Thus for $n>>0$ we would have $|C(s_{\ast })|<|\gamma
_{s_n}C(s_n)| $ and by the non-archimedean property of $\left\vert
\cdot\right\vert $ we would conclude that 
\begin{equation*}
|\gamma _{s_n}c(\gamma _{s_n}^{-1},s_n)|=|\gamma _{s_n}C(s_n)-C(s_{\ast
})|=|\gamma _{s_n}C(s_n)|\rightarrow \infty \text{.}
\end{equation*}
Now $( \text{\ref{Equation boundness}}) $ yields a contradiction.
\end{proof}

We are now ready to prove Theorem \ref{2parts}, which was already reduced to
proving \eqref{iso}.

\begin{proof}
By Corollary \ref{Corollary fundamental exact sequence} we can consider the
following commutative diagram, with exact rows: 
\begin{equation*}
\begin{array}{ccccccccc}
0 & \rightarrow & C_{har}^{b}(\mathbf{V}_n(K_p)) & \rightarrow & C_0^{b}( 
\mathcal{E},\mathbf{V}_n(K_p)) & \rightarrow & C^{b}(\mathcal{V},\mathbf{V}
_n(K_p)) & \rightarrow & 0 \\ 
&  & \cap &  & \cap &  & \cap &  &  \\ 
0 & \rightarrow & C_{har}(\mathbf{V}_n(K_p)) & \rightarrow & C_0(\mathcal{E}
, \mathbf{V}_n(K_p)) & \rightarrow & C(\mathcal{V},\mathbf{V}_n(K_p)) & 
\rightarrow & 0%
\end{array}%
\end{equation*}

The respective long exact sequences in cohomology induce the following
commutative diagram, with exact rows. 
\begin{equation*}
\begin{array}{ccccccc}
...{\rightarrow } & C^{b}(\mathcal{V},\mathbf{V}_n(K_p))^{\Gamma} & 
\rightarrow & H^1( \Gamma ,C_{har}^{b}(\mathbf{V}_n(K_p))) & \rightarrow & 
H^1( \Gamma,C_0^{b}(\mathcal{E},\mathbf{V}_n(K_p))) & \rightarrow... \\ 
& \downarrow &  & \downarrow &  & \downarrow &  \\ 
...{\rightarrow } & C(\mathcal{V},\mathbf{V}_n(K_p))^{\Gamma} & \rightarrow
& H^1(\Gamma,C_{har}(\mathbf{V}_n(K_p))) & \rightarrow & H^1(\Gamma,C_0( 
\mathcal{E},\mathbf{V}_n(K_p))) & \rightarrow...%
\end{array}%
\end{equation*}

Proposition \ref{Proposition boundness}, applied to ${\mathcal{S}}=\mathcal{V%
}$ and $\mathcal{E}^+$, shows that the first and third vertical arrows are
isomorphisms. The same applies to the two vertical maps arising just before
and after in the long exact sequence, which we do not draw. By the five
lemma the middle vertical arrow is an isomorphism too, which is what we
needed to prove.
\end{proof}

\subsection{Higher $p$-adic Abel-Jacobi maps}

\label{S22}

As in the previous sections, let $K_p/\mathbb{Q}_p$ be a complete field
extension. Since this choice is fixed throughout this subsection, we shall
drop it from the notation and simply write $\mathbf{P}_n=\mathbf{P}_n(K_p)$, 
$\mathbf{V}_n=\mathbf{V}_n(K_p)$, $\mathcal{A}_n( \mathbb{Q}_p) =\mathcal{A}
_n(\mathbb{Q}_p,K_p) $, $\mathcal{D}_n( \mathbb{Q}_p) =\mathcal{D}_n( 
\mathbb{Q}_p,K_p) $ and $\mathcal{D}_n^0( \mathbb{Q}_p) =\mathcal{D}_n^0( 
\mathbb{Q}_p,K_p)$. Let $k_p/\mathbb{Q}_p$ denote the maximal unramified
sub-extension of $K_p/\mathbb{Q}_p$.

\begin{definition}
\label{defpairings} Define pairings 
\begin{equation*}
\begin{matrix}
\tint\nolimits_-^-\,-\,\omega _-^{\log }: & \func{Div}^0( \mathcal{H}
_p)(k_p) \otimes \mathbf{P}_n\otimes \mathcal{D}_n^0(\mathbb{P}^1(\mathbb{Q}
_p)) & \rightarrow & K_p \\ 
& (\tau _{2}-\tau _1)\otimes P\otimes \mu & \mapsto & \tint\nolimits_{\tau
_1}^{\tau _{2}}P\,\omega _{\mu }^{\log } \\ 
&  &  &  \\ 
\tint\nolimits_-^-\,-\,\omega _-^{\func{ord}}: & \func{Div}^0( \mathcal{H}
_p)(k_p) \otimes \mathbf{P}_n\otimes \mathcal{D}_n^0(\mathbb{P}^1(\mathbb{Q}
_p)) & \rightarrow & K_p \\ 
& (\tau _{2}-\tau _1)\otimes P\otimes \mu & \mapsto & \tint\nolimits_{\tau
_1}^{\tau _{2}}P\omega _{\mu }^{\func{ord}}%
\end{matrix}%
\end{equation*}
where for any $\tau_1$, $\tau_2\in \mathcal{H}_p$, $P\in \mathbf{P}_n$ and $%
\mu\in \mathcal{D}^0_n(\mathbb{P}^1(\mathbb{Q}_p))$: 
\begin{equation*}
\tint\nolimits_{\tau _1}^{\tau _{2}}P\,\omega _{\mu }^{\log }:=\dint_{ 
\mathbb{P}^1( \mathbb{Q}_p) }\log _p( \frac{t-\tau _{2}}{t-\tau _1}) P( t)
d\mu ( t)
\end{equation*}
and 
\begin{equation*}
\tint\nolimits_{\tau _1}^{\tau _{2}}P\omega _{\mu }^{\func{ord}
}:=\dsum\limits_{e:red( \tau _1) \rightarrow red( \tau _{2}) }\dint_{U_e}P(
t) d\mu ( t) .
\end{equation*}
\end{definition}

Several comments are in order concerning the definition above. Recall that $%
\func{Div}^{0}(\mathcal{H}_{p})(k_{p})$ stands for the module of degree zero
divisors of $\mathcal{H}_{p}(\mathbb{Q}_{p}^{ur})=\mathbb{Q}
_{p}^{ur}\setminus \mathbb{Q}_{p}$ which are fixed by the action of the
Galois group $\func{Gal}(\mathbb{Q}_{p}^{ur}/k_{p})$. We shall regard $\func{
Div}^{0}(\mathcal{H}_{p})(k_{p})\otimes \mathbf{P}_{n}$ as a right $\mathrm{%
\ GL}_{2}(\mathbb{Q}_{p})$-module by the rule 
\begin{equation*}
(\tau _{2}-\tau _{1})\otimes P\cdot \gamma :=(\gamma ^{-1}\tau _{2}-\gamma
^{-1}\tau _{1})\otimes P\cdot \gamma \text{.}
\end{equation*}

Note that the definition of the first pairing depends on the choice of a
branch of a $p$-adic logarithm $\mathrm{log}_{p}:\,K_{p}^{\times
}\,\rightarrow \,K_{p}$; we do not specify a priori any such choice.

Finally, note that in the definition of the second pairing, the fact that $%
k_p/\mathbb{Q}_p$ is unramified implies that the reduction of any $\tau \in
k_p\setminus \mathbb{Q}_p$ is a \textit{vertex} (and not an edge) of the
tree $\mathcal{T}$. The sum is taken over the edges of the path joining the
two vertices $\mathrm{red}(\tau _1)$ and $\mathrm{red}(\tau _{2})$.

\begin{lemma}
\label{Cohomology of distrubutions lemma}The map 
\begin{eqnarray*}
\func{Div}^0( \mathcal{H}_p)(k_p) \otimes \mathbf{P}_n &\rightarrow & 
\mathcal{A} _n( \mathbb{Q}_p) /\mathbf{P}_n \\
(\tau _{2}-\tau _1)\otimes P &\mapsto &\log _p( \frac{t-\tau _{2}}{t-\tau _1}
) P( t)
\end{eqnarray*}
is $\mathrm{GL}_{2}( \mathbb{Q}_p) $-equivariant.
\end{lemma}

\begin{proof}
Write $\theta _{\tau _{2}-\tau _{1}}(t):=\frac{t-\tau _{2}}{t-\tau _{1}}$. A
direct computation shows that $(t-\tau )\cdot \gamma =\det \left( \gamma
\right) ^{-1/2}\left( a-c\tau \right) \cdot (t-\bar{\gamma}\tau )$, for all $%
\gamma =%
\begin{pmatrix}
a & b \\ 
c & d%
\end{pmatrix}%
\in \mathrm{GL}_{2}(\mathbb{Q}_{p})$. Here we write $\bar{\gamma}:=\det
(\gamma )\gamma ^{-1}$. It follows that $\theta _{\tau _{2}-\tau
_{1}}(t)\cdot \gamma =\frac{a-c\tau _{2}}{a-c\tau _{1}}\cdot \theta _{\bar{%
\gamma}\tau _{2}-\bar{\gamma}\tau _{1}}(t)$ and hence 
\begin{equation*}
\lbrack \log _{p}\theta _{\tau _{2}-\tau _{1}}(t)\cdot P(t)]\cdot \gamma
=\log _{p}(\frac{a-c\tau _{2}}{a-c\tau _{1}})\cdot (P\gamma )(t)+\log
_{p}\theta _{\bar{\gamma}\tau _{2}-\bar{\gamma}\tau _{1}}(t)\cdot (P\gamma
)(t).
\end{equation*}%
The claim follows, as $\gamma ^{-1}\tau =\bar{\gamma}\tau $ for every $%
\gamma \in \mathrm{GL}_{2}(\mathbb{Q}_{p})$ and $\log _{p}(\frac{a-c\tau _{2}%
}{a-c\tau _{1}})(P\gamma )(t)\in \mathbf{P}_{n}$.
\end{proof}

From now on, thanks to Theorem \ref{2parts}, we shall make the
identification 
\begin{equation}  \label{H}
\mathbf{H}(K_p):=H^1(\Gamma ,\mathcal{D}_n^0(\mathbb{P}^1( \mathbb{Q}
_p))^{b})=H^1(\Gamma,C_{har}(\mathbf{V}_n))\text{.}
\end{equation}

The natural inclusion $\mathcal{D}_n^0(\mathbb{P}^1(\mathbb{Q}
_p))^{b}\subseteq \mathcal{D}_n^0(\mathbb{P}^1(\mathbb{Q}_p))$ induces a map 
\begin{equation*}
H^1(\Gamma ,\mathcal{D}_n^0(\mathbb{P}^1(\mathbb{Q}_p))^{b})\hookrightarrow
H^1(\Gamma ,\mathcal{D}_n^0(\mathbb{P}^1(\mathbb{Q}_p))),
\end{equation*}
which is a monomorphism thanks to Theorem \ref{2parts}. Together with Lemma %
\ref{Cohomology of distrubutions lemma} and the cap product, the above
pairings induce maps 
\begin{equation}  \label{Psi}
\begin{matrix}
\Psi ^{\log },\Psi ^{\func{ord}}: & H_1( \Gamma ,\func{Div}^0(\mathcal{H}
_p)(k_p)\otimes \mathbf{P}_n) & \longrightarrow & \mathbf{H}(K_p)^{\vee }.%
\end{matrix}%
\end{equation}

\begin{lemma}
\label{Lemma H^1(gamma,V_n)=0} $H^1( \Gamma ,\mathbf{V}_n) =0$.
\end{lemma}

\begin{proof}
The exact sequence \eqref{exact sequence vertices/edges} and the
identifications \eqref{shp} provided by Shapiro's lemma induce the long
exact sequence 
\begin{equation}
0\rightarrow H^0( \Gamma,\mathbf{V}_n)\rightarrow H^0( \Gamma _0( N^+) , 
\mathbf{V}_n)^{2}\rightarrow H^0( \Gamma_0(pN^+),\mathbf{V}_n)\overset{
\delta }{\rightarrow }...  \label{e1}
\end{equation}
\begin{equation*}
...\overset{\delta }{\rightarrow }H^1( \Gamma,\mathbf{V}_n)\overset{
\varepsilon }{\rightarrow }H^1( \Gamma_0( N^+),\mathbf{V}_n) ^{2}\overset{
\partial^{\ast }}{\rightarrow }H^1( \Gamma _0( pN^+),\mathbf{V}_n){\
\rightarrow }...
\end{equation*}

Notice first that $\varepsilon $ is a monomorphism. Indeed, if $n=0$, $%
\mathbf{V}_n=K_p$ is trivial as $\Gamma $-module. Thus $H^0(\Gamma,\mathbf{V}
_n)=H^0(\Gamma_0(N^+),\mathbf{V}_n)=H^0(\Gamma_0(pN^+),\mathbf{V}_n)=K_p$.
The exactness of \eqref{e1} implies our claim. If $n>0$, $%
H^0(\Gamma_0(pN^+), \mathbf{V}_n)=0$ by \cite[pag. 165 Lemma 2]{Hi} and we
again deduce that $\mathrm{ker}(\varepsilon )=0$.

It thus remains left to show that $\mathrm{Ker}(\partial^{\ast })=0$. Let us
first show that $\mathrm{Ker}(\partial ^{\ast})^c=0$. The map $%
\partial^{\ast }$ can be composed with the degeneracy map $%
\partial_{\ast}:H^1( \Gamma_0(pN^+),\mathbf{V}_n)\rightarrow
H^1(\Gamma_0(N^+),\mathbf{V}_n)^{2}$ given by $\partial_{\ast }=(\func{cor}
_{\Gamma_0(pN^+)}^{\Gamma_0(N^+)},\func{cor}_{\Gamma _0(pN^+)}^{\hat{\Gamma}
_0(N^+)})$. The reader may wish to recall the natural identifications
already made in \eqref{GahGa}.

A computation now shows that the endomorphism $\partial _{\ast }\cdot
\partial^{\ast }$ of $H^1( \Gamma _0( N^+) ,\mathbf{V}_n) ^{2}$ is 
\begin{equation*}
\begin{pmatrix}
p+1 & p^{-m}T_p \\ 
p^{-m}T_p & p+1%
\end{pmatrix}
.
\end{equation*}

Fix an embedding of $K_p$ into the field $\mathbb{C}$ of complex numbers. By
Deligne's bound, the complex absolute value of the eigenvalues of the Hecke
operator $T_p$ acting on $H^1(\Gamma_0(N^+),\mathbf{V}_n(\mathbb{C}))^{c}$
are bounded above by $2\sqrt{p^{k-1}}=2p^{m}\sqrt{p}$. It thus follows that $%
\partial _{\ast }\cdot \partial ^{\ast }$ restricts to a linear automorphism
of the cuspidal part of $H^1(\Gamma_0(N^+),\mathbf{V}_n(\mathbb{C}))$, and
thus $(\partial ^{\ast})^c$ is injective.

In order to conclude, let us now show that $\mathrm{Ker}(\partial
^{\ast})^{Eis}=0$ when $N^-=1$. Let $C_{\Gamma_0(N^+)}=\{s_1,...,s_{t}\}$ be
a set of representatives for the cusps of $\Gamma_0(N^+)$. One then can
check that $C_{\hat{\Gamma }_0(N^+)}=\left\{ \widehat{s}_i:=\omega _p
s_i\right\}$ and $C_{\Gamma_0(pN^+)}=\left\{ s_{i},\hat{s}_{i}\right\} $ are
systems of representatives for the cusps of $\hat{\Gamma }_0(N^+)$ and $%
\Gamma_0(pN^+)$, respectively.

It follows from \eqref{Diagram Eichler-Shimura isomorphism} and the
discussion around it that for our purposes it suffices to show that for each 
$i=1,...,t$ the natural map 
\begin{equation*}
H^1(\Gamma_0(N^+)_{s_i},\mathbf{V}_n(\mathbb{C}))\overset{\partial^{\ast }}{{%
\longrightarrow }} H^1(\Gamma_0(pN^+)_{s_i},\mathbf{V}_n(\mathbb{C}))
\end{equation*}
induced by $\partial^{\ast}$ by restriction is a monomorphism (and
analogously for $\hat{\Gamma }_0(N^+)$ and $\hat{s}_i$). But this is clear
because $\partial ^{\ast }$ is the restriction map $\func{res}
_{\Gamma_0(pN^+)_{s_i}}^{\Gamma_0(N^+)_{s_i}}$, which is injective by a
similar reason as before: the composition with the corresponding
corestriction map is multiplication by the index $[
\Gamma_0(N^+)_{s_i}:\Gamma_0(pN^+)_{s_i}] $, which is finite as it divides $%
[\Gamma_0(N^+):\Gamma_0(pN^+)] = p+1$.
\end{proof}

Consider the exact sequence of $\Gamma $-modules 
\begin{equation}
0\rightarrow \func{Div}^0( \mathcal{H}_p)(k_p) \rightarrow \func{Div}( 
\mathcal{H} _p)(k_p) \rightarrow \mathbb{Z}\rightarrow 0\text{,}
\label{Exact sequence degree zero divisors p-adic upper halfplane}
\end{equation}
take the tensor product with $\mathbf{P}_n$ and then form the long exact
sequence in homology, which yields a connecting map 
\begin{equation}
H_{2}( \Gamma ,\mathbf{P}_n) \overset{\partial _{2}}{\rightarrow }H_1(
\Gamma ,\func{Div}^0( \mathcal{H}_p)(k_p) \otimes \mathbf{P}_n) \text{.}
\label{partialp}
\end{equation}

Recall from §\ref{S14} and §\ref{S15} that the cuspidal part $\mathbf{H}
(K_p)^{c}$ of $\mathbf{H}(K_p)$ is naturally identified with the $p$-new
space of cuspidal modular forms of level $pN^+$ on the quaternion algebra $%
\mathcal{B}$. Let $\mathrm{pr}_c: \mathbf{H}(K_p)^{\vee } {\longrightarrow }
(\mathbf{H}(K_p)^{c})^{\vee}$ denote the natural projection. By an abuse of
notation, we shall use the same symbol $\mathrm{pr}_c$ for the map $\mathrm{%
\ \ pr}_c\oplus \mathrm{pr}_c$.

\begin{theorem}
\label{Theorem L-invariant} For every $n\geq 0$ the morphism 
\begin{equation*}
\mathrm{pr}_c\circ \Psi^{\func{ord}}\circ \partial_2 :H_{2}( \Gamma ,\mathbf{%
P}_n) \rightarrow (\mathbf{H}(K_p)^{c})^{\vee }
\end{equation*}
is surjective and induces an isomorphism 
\begin{equation*}
(\Psi^{\func{ord}}\circ \partial_2)^c: H_{2}( \Gamma ,\mathbf{P}_n) ^{c} 
\overset{\simeq }{\rightarrow }(\mathbf{H}(K_p)^{c})^{\vee }\text{.}
\end{equation*}
\end{theorem}

\begin{proof}
Let us rewrite the homomorphism $\Psi ^{\func{ord}}$ as a composition of
several natural maps. First, consider the following commutative diagram with
exact rows: 
\begin{equation}
\begin{array}{ccccccccc}
0 & \rightarrow & \func{Div}^0( \mathcal{H}_p)(k_p) \otimes \mathbf{P}_n & 
\rightarrow & \func{Div}( \mathcal{H}_p)(k_p) \otimes \mathbf{P}_n & 
\rightarrow & \mathbf{P}_n & \rightarrow & 0 \\ 
&  & red\otimes 1\downarrow \text{ \ \ \ \ \ \ } &  & red\otimes 1\downarrow 
\text{ \ \ \ } &  & \parallel &  &  \\ 
0 & \rightarrow & \func{Div}^0( \mathcal{V}) \otimes \mathbf{P}_n & 
\rightarrow & \func{Div}( \mathcal{V}) \otimes \mathbf{P}_n & \rightarrow & 
\mathbf{P}_n & \rightarrow & 0.%
\end{array}
\label{Formula third exact sequence}
\end{equation}

The long exact sequence in homology yields a commutative diagram 
\begin{equation}
\begin{array}{ccc}
H_{2}( \Gamma ,\mathbf{P}_n) & \overset{\partial _{2}}{\rightarrow } & H_1(
\Gamma ,\func{Div}^0( \mathcal{H}_p)(k_p) \otimes \mathbf{P}_n) \\ 
& \overset{\partial _{\mathcal{V}}}{\searrow } & \downarrow \Psi _1 \\ 
&  & H_1( \Gamma ,\func{Div}^0( \mathcal{V}) \otimes \mathbf{P}_n) .%
\end{array}
\label{dia1}
\end{equation}

Second, let $\overline{\func{Div}}( \mathcal{E}) $ be the quotient of $\func{
Div}( \mathcal{E}) $ obtained by imposing the relation $\overline{e}+e=0$
for all $e\in \mathcal{E}$. Note that the morphisms 
\begin{equation*}
\begin{matrix}
path\text{ }: & \func{Div}^0( \mathcal{V}) & \rightarrow & \overline{\func{
Div}}( \mathcal{E}) \\ 
& v_1-v_{2} & \mapsto & \sum\nolimits_{e:v_1\rightarrow v_{2}}e \\ 
&  &  &  \\ 
\partial : & \overline{\func{Div}}( \mathcal{E}) \, & \rightarrow & \,\func{
Div}^0(\mathcal{V}) \\ 
& e & \mapsto & \,s( e) -t( e)%
\end{matrix}%
\end{equation*}
are one the inverse of the other and identify the two $\Gamma $-modules. In
particular we obtain from $path$ the commutative diagram 
\begin{equation}
\begin{array}{ccc}
H_{2}( \Gamma ,\mathbf{P}_n) & \overset{\partial _{\mathcal{V}}}{\rightarrow 
} & H_1( \Gamma ,\func{Div}^0( \mathcal{V}) \otimes \mathbf{P}_n) \\ 
& \overset{\partial _{\mathcal{E}}}{\searrow } & \downarrow \Psi _{2} \\ 
&  & H_1( \Gamma ,\overline{\func{Div}}( \mathcal{E}) \otimes \mathbf{P}_n) .%
\end{array}
\label{dia2}
\end{equation}
where the morphism $\partial _{\mathcal{E}}$ is obtained from the second row
of $( \text{\ref{Formula third exact sequence}}) $ and the identification $%
\func{Div}^0( \mathcal{V}) =\overline{\func{Div}}( \mathcal{E}) $.

Third, consider the exact sequence obtained from 
\eqref{exact
sequence vertices/edges} with $A=\mathbf{V}_n$: 
\begin{equation}  \label{Formula exact sequence}
0\rightarrow \mathbf{V}_n\rightarrow C( \mathcal{V},\mathbf{V}_n) \overset{
\partial ^{\ast }}{\rightarrow }C_0( \mathcal{E},\mathbf{V}_n) \rightarrow 0 
\text{.}
\end{equation}

The dual exact sequence of \eqref{Formula exact sequence} is canonically
identified with the exact sequence obtained from the second row of $( \text{ %
\ref{Formula third exact sequence}}) $ and the identification $\func{Div}^0( 
\mathcal{V}) =\overline{\func{Div}}( \mathcal{E}) $: 
\begin{equation}
0\rightarrow \overline{\func{Div}}( \mathcal{E}) \otimes \mathbf{P}_n\, 
\overset{\partial \otimes \func{Id}}{\longrightarrow }\,\func{Div}( \mathcal{%
V}) \otimes \mathbf{P}_n\rightarrow \mathbf{P}_n\rightarrow 0.
\label{Formula second exact sequence}
\end{equation}

More precisely the duality between \eqref{Formula exact sequence} and %
\eqref{Formula second exact sequence} is induced by the evaluation pairings: 
\begin{equation*}
\begin{matrix}
\left\langle -,-\right\rangle _{\mathcal{V}}: & \func{Div}( \mathcal{V})
\otimes \mathbf{P}_n\otimes C( \mathcal{V},\mathbf{V}_n) & \rightarrow & K_p
\\ 
& v\otimes P\otimes c & \mapsto & c( v,P) \\ 
&  &  &  \\ 
\left\langle -,-\right\rangle _{\mathcal{E}}: & \overline{\func{Div}}( 
\mathcal{E}) \otimes \mathbf{P}_n\otimes C_0( \mathcal{E},\mathbf{V}_n) & 
\rightarrow & K_p \\ 
& e\otimes P\otimes c & \mapsto & c( e,P) .%
\end{matrix}%
\end{equation*}

By cap product, these pairings yield the following commutative diagram: 
\begin{equation}
\begin{array}{cccc}
H_{2}( \Gamma ,\mathbf{P}_n) & \overset{\partial _{\mathcal{E}}}{\rightarrow 
} & H_1( \Gamma ,\overline{\func{Div}}( \mathcal{E}) \otimes \mathbf{P}_n) & 
\\ 
\downarrow &  & \downarrow \Psi _{3} &  \\ 
H^{2}( \Gamma ,\mathbf{V}_n) ^{\vee } & \overset{\delta ^{\vee }}{
\rightarrow } & H^1( \Gamma ,C_0( \mathcal{E},\mathbf{V}_n) ) ^{\vee }. & 
\end{array}
\label{dia3}
\end{equation}

The Universal Coefficients theorem guarantees that the above vertical arrows
are isomorphisms. With these notations the morphism $\Psi ^{\func{ord}}$ is
obtained as follows. Let $\Psi _{4}$ be the dual of the morphism 
\begin{equation*}
\mathbf{H}(K_p)=H^1(\Gamma ,C_{har}(\mathbf{V}_n))\rightarrow H^1(\Gamma
,C_0( \mathcal{E},\mathbf{V}_n)).
\end{equation*}

Then we have 
\begin{equation*}
\Psi ^{\func{ord}}=\Psi _{4}\circ \Psi _{3}\circ \Psi _{2}\circ \Psi _1\text{
.}
\end{equation*}

Hence the morphism $\mathrm{pr}_c\circ \Psi ^{\func{ord}}$ is obtained by
further composition with the morphism $\mathrm{pr}_c$ dual to the inclusion $%
i^{c}:\mathbf{H}(K_p)^{c}\subset \mathbf{H}(K_p)$: 
\begin{equation*}
\mathrm{pr}_c\circ \Psi^{\func{ord}} \circ \partial_2 =\mathrm{pr}_c\circ
\Psi _{4}\circ \Psi _{3}\circ \Psi _{2}\circ \Psi _1\text{.}
\end{equation*}

From the commutativity of diagrams (\ref{dia1}), (\ref{dia2}) and (\ref{dia3}
), we obtain the commutative diagram 
\begin{equation*}
\begin{array}{ccccc}
H_{2}( \Gamma ,\mathbf{P}_n) & \overset{\partial _{2}}{\rightarrow } & H_1(
\Gamma ,\func{Div}^0( \mathcal{H}_p) \otimes \mathbf{P}_n) &  &  \\ 
\downarrow &  & \downarrow \Psi _{3}\circ \Psi _{2}\circ \Psi _1 &  &  \\ 
H^{2}( \Gamma ,\mathbf{V}_n) ^{\vee } & \overset{\delta ^{\vee }}{
\rightarrow } & H^1( \Gamma ,C_0( \mathcal{E},\mathbf{V}_n) ) ^{\vee } & 
\overset{\mathrm{pr}_c\circ \Psi _{4}}{\rightarrow } & (\mathbf{H}(K_p)
^{c})^{\vee}.%
\end{array}%
\end{equation*}

As already mentioned, the left vertical arrow is an isomorphism. It remains
to prove that $\mathrm{pr}_c\circ \Psi _{4}\circ \delta ^{\vee }$ restricts
to an isomorphism on the cuspidal parts, since then the first statement
about surjectivity will follow then from Remark \ref{RmkEC}.

Equivalently, since the composition $\mathrm{pr}_c\circ \Psi _{4}$ is dual
to the morphism 
\begin{equation*}
\mathbf{H}(K_p)^{c}\subset \mathbf{H}(K_p)\rightarrow H^1(\Gamma ,C_0( 
\mathcal{E}, \mathbf{V}_n))\text{,}
\end{equation*}
we need to show that the morphism 
\begin{equation*}
\mathbf{H}(K_p)=H^1(\Gamma ,C_{har}(\mathbf{V}_n))\rightarrow H^1(\Gamma
,C_0( \mathcal{E},\mathbf{V}_n))\rightarrow H^{2}( \Gamma ,\mathbf{V}_n)
=H_{2}( \Gamma ,\mathbf{P}_n) ^{\vee }
\end{equation*}
induces an isomorphism when restricted to the cuspidal parts. This is the
content of Lemmas \ref{Lemma harmonic cocycles/cusp forms} and \ref{Lemma
p-new and p-old subspaces}.
\end{proof}

\begin{remark}
\label{Remark L-invariant} When $n>0$ and $N^->1$ we have $\mathrm{pr}_c= 
\mathrm{Id}$. Furthermore $H_{2}( \Gamma ,\mathbf{P}_n) ^{c}=H_{2}( \Gamma , 
\mathbf{P}_n) $, $\mathbf{H}(K_p)^{c}=\mathbf{H}(K_p)$ and the morphism $%
\Phi ^{\func{ ord}}$ is an isomorphism.
\end{remark}

Let $\mathbb{T}:=\mathbb{T}_{\Gamma_0(pN^+)}^{p-new}$ denote the maximal
quotient of the Hecke algebra $\mathcal{H}(\Gamma_0(pN^+))\otimes \mathbb{Q}$
acting on $S_k(\Gamma_0(pN^+))^{p-new}$ and put $\mathbb{T}_p=\mathbb{T}
\otimes_{\mathbb{Q}}\mathbb{Q}_p$, $\mathbb{T}_{K_p}=\mathbb{T }\otimes_{ 
\mathbb{Q}} K_p$.

\begin{corollary}
\label{Corollary existence and uniqueness of the L-invariant} There exists a
unique endomorphism $\mathcal{L}\in \func{End}_{\mathbb{T}_{K_p}}((\mathbf{H}
(K_p)^{c})^{\vee})$ such that 
\begin{equation}  \label{L}
\mathrm{pr}_c\circ \Psi ^{\log }\circ \partial_2 =\mathcal{L}\circ \mathrm{\
pr }_c \circ \Psi ^{\func{ord}}\circ \partial_2 : H_{2}( \Gamma ,\mathbf{P}
_n) \rightarrow (\mathbf{H}(K_p)^{c})^{\vee}\text{.}
\end{equation}
\end{corollary}

\begin{proof}
Let $i: (\mathbf{H}(K_p)^c)^{\vee} \overset{\sim}{{\rightarrow }}
H_2(\Gamma, \mathbf{P}_n)^c$ be the inverse of the isomorphism $(\Psi^{\func{
ord}}\circ \partial_2)^c$ of Theorem \ref{Theorem L-invariant} and define $%
\mathcal{L}= (\Psi^{log}\circ \partial_2)^c \circ i$. Since $i\circ \Psi ^{ 
\func{ord} }\circ \partial_2$ is the natural projection $H_{2}( \Gamma , 
\mathbf{P}_n) \rightarrow H_{2}( \Gamma ,\mathbf{P}_n)^c$, it is clear that %
\eqref{L} holds true with this choice of ${\mathcal{L}}$. As for the
uniqueness, let $\tilde{{\mathcal{L}}}\in \func{End}_{\mathbb{T}_{\mathbb{Q}
_p}}( (\mathbf{H}(K_p) ^{c})^{\vee})$ be any endomorphism satisfying %
\eqref{L} and let $\tilde{{\mathcal{L}}}=\tilde{{\mathcal{L}}}^{Eis}\oplus 
\tilde{{\mathcal{L}}}^c$ denote its Eisenstein/cuspidal decomposition
(cf.\,Remark \ref{RmkEC}). Since the Eisenstein subspace of $(\mathbf{H}
(K_p)^{c})^{\vee}$ is trivial, it follows that $\tilde{{\mathcal{L}}}
^{Eis}=0 $. Hence $\tilde{{\mathcal{L}}}= \tilde{{\mathcal{L}}}^c$ and, by
Theorem \ref{Theorem L-invariant}, $\tilde{{\mathcal{L}}}^c={\mathcal{L}} $
is necessarily the endomorphism defined above.
\end{proof}

\begin{definition}
\label{Linv} The ${\mathcal{L}}$-invariant of the space $S_{k}(\Gamma
_{0}(pN^{+}))^{p-new} $ of $p$-new modular forms is the endomorphism 
\begin{equation*}
\mathcal{L}\in \func{End}_{\mathbb{T}_{K_p}}((\mathbf{H}(K_p)^{c})^{\vee })
\end{equation*}
appearing in the above corollary.
\end{definition}

By Remark \ref{rankone}, $(\mathbf{H}(K_p)^{c})^{\vee }$ is a free rank one $%
\mathbb{T}_{K_p}$-module. Hence $\mathcal{L}\in \mathbb{T}_{K_p}$. But we
can even claim that ${\mathcal{L}}\in \mathbb{T}_p$, because our
construction of the ${\mathcal{L}}$-invariant is valid for any complete
field extension $K_p/\mathbb{Q}_p$ and it is clear from Corollary \ref%
{Corollary existence and uniqueness of the L-invariant} that it is invariant
under base change.

\section{Monodromy modules}

\label{SMonodromy}

\subsection{Fontaine-Mazur theory}

\label{S41}

As in §\ref{S22}, fix a complete field extension $K_p/\mathbb{Q}_p$ and let $%
k_p/\mathbb{Q}_p$ denote the maximal unramified sub-extension of $K_p/ 
\mathbb{Q}_p$. Write $\sigma \in \func{Aut}(k_p)$ for the absolute Frobenius
of $k_p$. Throughout in this section, $k\geq 2$ is a fixed positive even
integer.

Let $\mathbb{T}_p$ be a finite dimensional commutative $\mathbb{Q}_p$
-algebra and write $\mathbb{T}_{k_p}=\mathbb{T}_p\otimes k_p$ and $\mathbb{T}
_{K_p}=\mathbb{T}_p\otimes K_p$. Set $\sigma _{\mathbb{T}_{k_p}}:=\func{Id}
\otimes \sigma $ on $\mathbb{T}_{k_p}$.

\begin{definition}
\label{2dimMon} A \emph{two dimensional monodromy $\mathbb{T}_{p}$-module
over $K_{p}$} is a $4$-ple $(D,\varphi ,N,F^{\cdot })$ where $D$ is a $
\mathbb{T}_{k_{p}}$-module, $\varphi :D{\rightarrow }D$ is $\sigma $-linear
endomorphism (i.e.\thinspace $\varphi (ax)=\sigma (a)x$ for all $a\in
k_{p},x\in D$) and $N:D{\rightarrow }D$ is a $k_{p}$-linear endomorphism
such that

\begin{itemize}
\item[(a)] $F^{\cdot }$ is a filtration on the $K_p$-vector space $%
D\otimes_{k_p} K_p$ of the form 
\begin{equation*}
D\otimes K_p=F^{0}\supset F^{1}=...=F^{k-1}\supset F^{k}=0
\end{equation*}
where $F^{k-1}$ is a free $\mathbb{T}_{K_p}$-module of rank one;

\item[(b)] $D\otimes K_{p}=F^{k-1}\oplus N_{K_{p}}(D\otimes K_{p})$ as a $%
\mathbb{T}_{K_{p}}$-module, with $N_{K_{p}}:F^{k-1}\rightarrow
N_{K_{p}}(D\otimes K_{p})$ a $\mathbb{T}_{K_{p}}$-module isomorphism.

\item[(c)] $N\cdot \varphi =p\varphi \cdot N$ and, for any $T\in \mathbb{T}
_{k_{p}}$, $\varphi T=\sigma _{\mathbb{T}_{k_{p}}}(T)\varphi $ and $TN=NT$.
\end{itemize}
\end{definition}

See \cite{CI}, \cite[§2]{IS} and \cite[§9, p.\,12]{Ma} for related but
slightly different notions, and for proofs of some of the claims.

\vspace{0.3cm}

Let $D=(D,\varphi ,N,F^{\cdot })$ be two dimensional monodromy $\mathbb{T}_p$
-module over $K_p$, which by an abuse of notation sometimes will be referred
simply as $D$. When we forget the $\mathbb{T}_p$-structure, $D$ is what it
is customary to call a filtered Frobenius monodromy module, or simply a $%
(\varphi,N)$-module over $K_p$. Write $\mathrm{MF} _{K_p}(\varphi,N) $ for
the category of such objects, in which a morphism is a homomorphism of $k_p$
-modules preserving the filtrations and commuting with $\varphi$ and $N$. As
an illustrative example, multiplication by a scalar $a\in k_p$ on $D$ is an
endomorphism of vector spaces over $k_p$ which is a morphism in $\mathrm{MF}
_{K_p}(\varphi,N)$ if and only if $a\in \mathbb{Q}_p$. The category $\mathrm{%
\ MF}_{K_p}(\varphi,N)$ is an additive tensor category admitting kernels and
cokernels.

\begin{remark}
\label{basechange} If $K_{p}^{+}\supseteq K_{p}$ is a complete field
extension of $\mathbb{Q}_{p}$ containing $K_{p}$, then the maximal
unramified sub-extension $k_{p}^{+}$ of $K_{p}^{+}/\mathbb{Q}_{p}$ contains $%
k_{p}$, and there is a natural obvious notion of base change of monodromy
modules: $D_{K_{p}^{+}}:=(D\otimes _{k_{p}}k_{p}^{+},\varphi _{k_{p}}\otimes
\sigma _{k_{p}^{+}/k_{p}},N_{k_{p}}\otimes k_{p}^{+},F^{\cdot }\otimes
K_{p}^{+})$ is a two-dimensional monodromy $\mathbb{T}_{p}$-module over $%
K_{p}^{+}$.
\end{remark}

In our applications in §\ref{S42}, we shall be working with monodromy
modules over the quadratic unramified extension $\mathbb{Q}_{p^2}$ of $%
\mathbb{Q}_p$ which in fact can be obtained as the base change of a
monodromy module over $\mathbb{Q}_p$.

Consider the slope decomposition 
\begin{equation*}
D=\tbigoplus\nolimits_{\alpha\in \mathbb{Q}}D^{\alpha}
\end{equation*}
where for $\alpha = r/s$, $r,s \in \mathbb{Z}$, $s>0$, $D^{\alpha}\subset D$
is the largest subspace of $D$ which has an ${\mathcal{O}}_{k_p}$-stable
lattice $D_0$ with $\varphi^s(D_0)$ = $p^r D_0$.

Since $N\neq 0$ by (b) and $N(D^{\alpha +1})\subset D^{\alpha }$ by (c),
there exists $\lambda \in \mathbb{Q}$ such that $D^{\lambda },D^{\lambda
+1}\neq 0$ are free $\mathbb{T}_{k_{p}}$-modules of rank $1$ and the map $%
N:D^{\lambda +1}\rightarrow D^{\lambda }$ is non-zero. It can be shown that $%
D$ is free of rank two over $\mathbb{T}_{k_{p}}$ and we deduce that such a $%
\lambda $ is unique, and we call it the \emph{slope} of $D$. It is easy to
check that 
\begin{eqnarray}
&&D=D^{\lambda }\oplus D^{\lambda +1},\quad D^{\lambda }=\ker N=N(D)\text{,}
\label{Lemma slope decomposition} \\
&&\text{with }D^{i}\simeq \mathbb{T}_{k_{p}}\text{ for }i=\lambda ,\lambda
+1 \text{.}
\end{eqnarray}

\begin{definition}
\label{L-FM} The \emph{$\mathcal{L}$-invariant $\mathcal{L}_{D}$ of $%
D=(D,\varphi ,N,F^{\cdot })$} is defined to be the unique element $\mathcal{%
\ L } _{D}\in \mathbb{T}_{K_p}$ such that 
\begin{equation*}
x-\mathcal{L}_{D}N_{K_p}(x)\in F^{k-1}\text{ for every }x\in D^{\lambda
+1}\otimes K_p\text{.}
\end{equation*}
\end{definition}

The existence and uniqueness of $\mathcal{L}_{D}\in \mathbb{T}_{K_p}$ are
again easy to check.

\begin{lemma}
\label{Lemma endomorphisms of a monodromy module} $\mathbb{T}_p\simeq \func{
End}_{\func{MF}_{K_p}^{ad}(\varphi ,N)}(D)\text{.}$
\end{lemma}

\begin{proof}
It follows from (a)-(c) that there is a natural map $\mathbb{T}_p\overset{
\eta}{{\rightarrow }} \func{End}_{\func{MF}_{K_p}^{ad}(\varphi ,N)}(D)$. The
algebra $\mathbb{T}_{K_p}$ preserves $F^{k-1}$ and it follows from (a) that
the above map induces an isomorphism $\func{End} _{K_p}(F^{k-1}) =\mathbb{T}
_{K_p}$. In particular $\eta$ is injective. As for surjectivity, let $f\in 
\func{End}_{\func{MF} _{K_p}(\varphi,N)}(D)$. Since $f$ commutes with $%
\varphi$, it preserves the slope decomposition 
\eqref{Lemma slope
decomposition}. For $i\in \{\lambda, \lambda+1\}$, let $t_{f}^{i}\in \mathbb{%
\ T}_{k_p}$ be such that $f_{|D^{i}}=t_f^{i}\in \func{End}_{k_p}(D^i)$.

Since $N: D^{\lambda +1} {\rightarrow } D^{\lambda }$ is an isomorphism of $%
\mathbb{T}_{k_p}$-modules by \eqref{Lemma slope
decomposition}, we may write $D^{\lambda+1}=\mathbb{T}_{k_p}\cdot e$, $%
D^{\lambda}=\mathbb{T}_{k_p} N(e)$ for some $e\in D^{\lambda+1}$. Since $%
t_{f}^{\lambda +1}N(e)=Nt_{f}^{\lambda +1}(e) = N f (e)=f N
(e)=t_{f}^{\lambda } N (e)$ we deduce that $t:=t_{f}^{\lambda
+1}=t_{f}^{\lambda }$.

Finally, since $t$ must commute with the $\sigma$-linear endomorphism $%
\varphi$, it follows that $t\in \mathbb{T}_p$.
\end{proof}

Along with ${\mathcal{L}}_D$, one may also attach to $D$ the following
invariant. The notation is as in the previous proof.

\begin{definition}
Let $U=U_D\in \mathbb{T}_{k_p}$ be the element such that $\varphi N(e) =U
N(e)$.
\end{definition}

Notice that $U$ exists and is well-defined, because $D^{\lambda}$ is
preserved by $\varphi $ and $D^{\lambda}=\mathbb{T}_{k_p}\cdot N(e)$. The
reader may check that $U$ does not depend on the choice of the generator $e$
of $D^{\lambda+1}$.

As a final remark in this short review of monodromy modules, we quote the
following simple but clarifying fact: the invariants $U_{D}$ and ${\mathcal{%
L }}_{D}$ of a two dimensional monodromy $\mathbb{T}_{p}$-module $D$ over $%
K_{p}$ are not only invariants of $D$, but actually \emph{completely
determine} $D$ up to isomorphism. More precisely, we can prove the following
statement:

\begin{proposition}
\label{UL} For any pair of elements $U\in \mathbb{T}_{k_{p}}$ and ${\mathcal{%
\ L}}\in \mathbb{T}_{K_{p}}$ there exists a two dimensional monodromy $\mathbb{\ T}_{p}$-module $D_{U,{\mathcal{L}}}$ over $K_{p}$ such that $U_{D_{U,{\mathcal{L}}}}=U$ and $\mathcal L_{D_{U,{\mathcal{L}}}}=\mathcal L$. It is admissible if and only if its slope is $\left( k-2\right) /2$.
Moreover, for any two-dimensional monodromy $\mathbb{\ T}_{p}$-module $D$,
\begin{equation}
D\simeq D_{U,{\mathcal{L}}}\,\text{ if and only if }\,U_{D}=U,\,{\mathcal{L}}
_{D}={\mathcal{L}}\text{.}
\end{equation}

\end{proposition}

This will be useful for our purposes in §\ref{S42}. As we were not able to
find an explicit proof of this fact in the literature, let us sketch the
details.

\begin{proof}
Fix $U\in \mathbb{T}_{k_{p}}$ and ${\mathcal{L}}\in \mathbb{T}_{K_{p}}$ as
in the statement and define 
\begin{equation*}
D_{U,{\mathcal{L}}}:=\mathbb{T}_{k_{p}}\oplus \mathbb{T}_{k_{p}}
\end{equation*}
endowed with:

\begin{itemize}
\item a filtration $D_{U,{\mathcal{L}}}\otimes K_{p}=F^{0}\supsetneq
F^{1}=...=F^{k-1}\supsetneq F^{k}=0$, where for all $1\leq j\leq k-1$, 
\begin{equation*}
F^{j}=\left\{ (-\mathcal{L}x,x):x\in \mathbb{T}_{K_{p}}\right\} ;
\end{equation*}

\item a Frobenius operator $\varphi _{U,{\mathcal{L}}}$ given by the rule 
\begin{equation*}
\varphi _{U,{\mathcal{L}}}(x,y):=(\sigma _{\mathbb{T}_{k_{p}}}(x)U,p\,\sigma
_{\mathbb{T}_{k_{p}}}(y)U);
\end{equation*}

\item a monodromy operator $N_{U,{\mathcal{L}}}$ defined by the rule 
\begin{equation*}
N_{U,{\mathcal{L}}}(x,y)=(y,0)\text{.}
\end{equation*}
\end{itemize}

One immediately checks that $D_{U,{\mathcal{L}}}$ is a two dimensional
monodromy $\mathbb{T}_{p}$-module over $K_{p}$, satisfying conditions (a),
(b), (c) as required. It also follows from the definitions that ${\mathcal{L}
}_{D_{U,{\mathcal{L}}}}={\mathcal{L}}$ and $U_{D_{U,{\mathcal{L}}}}=U$.

In order to prove the converse, let now $D=(D,\varphi ,N,F^{\cdot })$ be any
two dimensional monodromy $\mathbb{T}_p$-module over $K_p$, say of slope $%
\lambda$, such that $U_D=U$ and ${\mathcal{L}}_D={\mathcal{L}}$. As in the
proof of Lemma \ref{Lemma endomorphisms of a monodromy module}, we can write 
$D=D^{\lambda}\oplus D^{\lambda+1}=\mathbb{T}_{k_p} N(e) \oplus \mathbb{T}
_{k_p} e$ and this allows us to fix the isomorphism of $\mathbb{T}_{k_p}$
-modules $\mu: D\simeq D_{U,{\mathcal{L}}}=\mathbb{T}_{k_p}\oplus \mathbb{T}
_{k_p}$ given by $\mu(e) = (0,1)$, $\mu(N(e))=(1,0)$.

Let us show that $\mu$ is also an isomorphism of monodromy $\mathbb{T}_p$
-modules over $K_p$. It is obvious from the construction that $\mu $
intertwines the action of $N$. It also follows immediately from Definition %
\ref{L-FM} and the equality ${\mathcal{L}}_D={\mathcal{L}}$ that $\mu$
preserves the filtration. Finally, $\mu$ commutes with $\varphi$ thanks to
the defining property of $U$, condition (c) of Definition \ref{2dimMon} and
the fact that $N_{|D^{\lambda+1}}: D^{\lambda+1} {\rightarrow } D^{\lambda}$
is an isomorphism.
\end{proof}

Fix an algebraic closure $\qbar$ of $\mathbb{Q}$. Assume $K_p$ is algebraic
over $\mathbb{Q}_p$ and choose an algebraic closure $\qbar_p$ of $\mathbb{Q}
_p$ containing $K_p$. Choose also a prime ideal $\bar \wp$ of $\qbar$ over $%
p $, which we may use to fix an embedding of $G_{\mathbb{Q}_p}:= \func{ Gal\,%
}(\qbar_p/\mathbb{Q}_p)$ into $G_{\mathbb{Q}}:=\func{Gal\,}(\qbar
/ \mathbb{Q})$.

For a $p$-adic representation $V$ of $G_{K_p}$ over $\mathbb{Q}_p$ one
defines $D_{st}(V):= (V\otimes B_{st})^{G_{K_p}}$, where $B_{st}$ is
Fontaine's ring defined in \cite{Fo} and from which $D_{st}(V)$ inherits the
structure of a filtered $(\varphi,N)$-module over $K_p$. A $p$-adic
representation $V$ of $G_{K_p}$ is called \emph{semistable} if the canonical
monomorphism $D_{st}(V)\otimes_{k_p} B_{st}{\rightarrow } V\otimes_{\mathbb{%
\ Q }_p} B_{st}$ is an isomorphism. A filtered $(\varphi,N)$-module $D$ over 
$K_p $ is called \emph{admissible} if $D\simeq D_{st}(V)$ for some
semistable representation $V$. The functor $D_{st}$ establishes an
equivalence of categories between that of semistable continuous
representations of $G_{K_p}$ over $\mathbb{Q}_p$ and admissible filtered $%
(\varphi,N)$-modules $D$ over $K_p$.

Let $\mathbb{T}=\mathbb{T}^{p-new}_{\Gamma_0(pN^+)}\otimes \mathbb{Q}$ and
put $\mathbb{T}_p:= \mathbb{T}\otimes \mathbb{Q}_p$. As recalled in the
introduction, let $V_p:=H_p( \mathcal{M}_n)^{p-new}$ denote the $p$-new
quotient of the $p$ -adic ètale realization of the motive $\mathcal{M}_{n}$
attached to the space of $p$-new cusp forms of weight $k$ with respect to $%
\Gamma_0(pN^+)$. Let us regard $V_p$ as a representation of $\func{Gal\,}( %
\qbar_p/ \mathbb{Q}_p)$, by restricting the action of $G_{\mathbb{Q}}$ to
the decomposition subgroup of the fixed prime $\bar \wp$ above. As it is
well-known to the experts, $V_p$ is semistable (cf.\,\cite{C} and \cite{CI}
). Crucial for this is the fact that the Shimura curve $X_0^{N^-}(pN^+)$ has
semistable reduction at $p$.

The admissible filtered $(\varphi,N)$-module $\mathbf{D}^{FM}:=D_{st}(V_p)$
attached by Fontaine and Mazur to $V_p$ is in a natural way a
two-dimensional monodromy $\mathbb{T}_p$-module over $\mathbb{Q}_p$ in the
sense of Definition \ref{2dimMon}, for which $\varphi=U_p$ is the usual
Hecke operator at $p$ and the slope is $m$; cf.\,again \cite{C} and \cite{CI}%
. Let $\mathcal{L} ^{FM} := \mathcal{L}_{\mathbf{D}^{FM}}\in\mathbb{T}_p$.

Recall from the previous section the complete field extension $K_p/\mathbb{Q}
_p$ and its maximal unramified subfield $k_p$. Write $\mathbf{D}
_{K_p}^{FM}=D_{st}({V_p}_{|G_{K_p}})$, which can also be obtained from $%
\mathbf{D}^{FM}$ by base change to $K_p$ in the sense of Remark \ref%
{basechange}. As it follows from the definitions, ${\ \mathcal{L}}_{\mathbf{D%
}_{K_p}^{FM}}={\mathcal{L}}_{\mathbf{D}^{FM}}={\ \mathcal{L}}^{FM}\in 
\mathbb{T}_p$.

\subsection{A monodromy module arising from $p$-adic integration}

\label{S42}

The aim of this section is to explain how the theory developed above allows
us to construct a monodromy module attached to the space of $p$-new modular
forms $S_{k}( \Gamma _0(pN^+))^{p-new}$ and the invariant ${\mathcal{L}}$
introduced in Definition \ref{Linv}.

As before, let $\mathbb{T}=\mathbb{T}^{p-new}_{\Gamma_0(pN^+)}\otimes 
\mathbb{Q}$ and put $\mathbb{T}_p:=\mathbb{T}\otimes \mathbb{Q}_p$. For any
field extension $L/\mathbb{Q}$, set 
\begin{equation}  \label{defH}
\mathbf{H}(L):=H^1(\Gamma,C_{har}(\mathbb{V}_n(L)));
\end{equation}
recalling the identification $\mathbb{V}_n(L)=\mathbf{V}_n(L)$ of \eqref{Vn}
whenever $L$ splits $\mathcal{B}$, this notation is in consonance with %
\eqref{H}.

Fix a choice of a sign $w_{\infty}\in \{\pm 1\}$. Since $\mathbf{H}( \mathbb{%
\ Q}_p)^{c,\vee ,w_{\infty}}$ is a free module of rank $1$ over $\mathbb{T}%
_p $ by Remark \ref{rankone}, we can fix a generator and identify this way $%
\mathbf{H}(\mathbb{Q}_p)^{c,\vee ,w_{\infty}}\simeq \mathbb{T}_p$. Exactly
as in the proof of Proposition \ref{UL}, we can attach now to $U_p, {\ 
\mathcal{L}} \in \mathbb{T}_p$ a monodromy module 
\begin{equation}  \label{monmod}
\mathbf{D}:=D_{U_p,{\mathcal{L}}} = \mathbf{H}(\mathbb{Q} _p)^{c,\vee
,w_{\infty }}\oplus \mathbf{H}(\mathbb{Q}_p)^{c,\vee ,w_{\infty }}
\end{equation}
endowed with the filtration, Frobenius and monodromy operators described in
loc. cit.

\begin{conjecture}
\label{Conjecture equality of the L-invariants 1} $\mathcal{L}_{\mathbf{D}}={%
\mathcal{L}}_{\mathbf{D}^{FM}}$.
\end{conjecture}

In the above conjecture, note that the definition of both monodromy modules
depends on the choice of a branch of the $p$-adic logarithm. We assume that
the same choice has been done for both $\mathbf{D}$ and $\mathbf{D}^{FM}$.

In view of Proposition \ref{UL}, Conjecture \ref{Conjecture equality of the
L-invariants 1} is equivalent to saying that there is an isomorphism $%
\mathbf{D}\simeq \mathbf{D}^{FM}$ of two-dimensional monodromy $\mathbb{T}
_p $-modules over $\mathbb{Q}_p$ (as $U_{\mathbf{D}} = U_{\mathbf{D}^{FM}}$).

Let $\mathbf{D}_{K_p} = (\mathbf{H} (k_p)^{c,\vee ,w_{\infty}}\oplus \mathbf{%
\ H}(k_p)^{c,\vee ,w_{\infty }},\varphi \otimes \sigma_{k_p/\mathbb{Q}
_p},N\otimes k_p,F^{\cdot }\otimes K_p )$ denote the base change to $K_p$ of 
$\mathbf{D}$ in the sense of Remark \ref{basechange}.

Let 
\begin{equation}
\Psi :=-\Psi ^{\log }\oplus \Psi ^{\func{ord}}:H_1(\Gamma ,\func{Div}^0( 
\mathcal{H}_p)(k_p)\otimes \mathbf{P}_n(K_p)) \longrightarrow \mathbf{H}
(K_p)^{\vee}\oplus \mathbf{H}(K_p)^{\vee}
\end{equation}
and set 
\begin{equation}  \label{PhiOnto}
\Phi :=-\Phi ^{\log }\oplus \Phi ^{\limfunc{ord}}:H_1(\Gamma ,\limfunc{Div}
\nolimits^0( \mathcal{H}_p)(k_p) \otimes \mathbf{P}_n(K_p)) \rightarrow 
\mathbf{D}\otimes K_p
\end{equation}
for the natural composition of the above map(s) onto $\mathbf{H}
(K_p)^{c,\vee,w_{\infty }}$.

By definition of $\Phi$, the free $\mathbb{T}_{K_p}$-submodule of rank one 
\begin{equation*}
F^1=F^{m}=F^{k-1}:=\left\{ ( -\mathcal{L}x,x) :x\in \mathbf{H}(K_p)^{c,\vee
,w_{\infty }}\right\}
\end{equation*}
of $\mathbf{D}\otimes K_p$ is $F^m=\func{Im}( \Phi \circ \partial_2)$.

As it will be useful for our purposes later in the construction of Darmon
cycles, let us recall at this point that, thanks to Lemma \ref{Lemma
H^1(gamma,V_n)=0}, there is a natural isomorphism 
\begin{equation}  \label{id}
H_1( \Gamma ,\func{Div}(\mathcal{H}_p)\otimes \mathbf{P}_n(K_p))\simeq\frac{
H_1( \Gamma ,\func{Div}^0( \mathcal{H}_p) \otimes \mathbf{P}_n(K_p)) }{\func{
Im}\partial _{2}}
\end{equation}
which allows us to identify both spaces.

\begin{definition}
\label{AJ1} The \emph{$p$-adic Abel-Jacobi maps} are the morphisms 
\begin{equation*}
\Psi ^{AJ}:H_1( \Gamma ,\func{Div}(\mathcal{H}_p)\otimes \mathbf{P}_n(K_p))
\longrightarrow \frac{\mathbf{H}(K_p)^{\vee}\oplus \mathbf{H}(K_p)^{\vee}}{ 
\func{Im}\Psi \circ \partial _{2}}
\end{equation*}
and 
\begin{equation}  \label{AJ2}
\Phi ^{AJ}=\mathrm{pr}_c\circ \Psi ^{AJ}:H_1( \Gamma ,\func{Div}(\mathcal{H}
_p)\otimes \mathbf{P}_n(K_p)) \longrightarrow \mathbf{D}\otimes K_p/F^{m} 
\text{.}
\end{equation}
induced respectively by $\Psi$ and $\Phi$, respectively, together with the
isomorphism \eqref{id}.
\end{definition}

\subsection{An Eichler-Shimura construction}

\label{Subsubsection ES-construction}

Let $\mathbb{T}$ be a finite dimensional semisimple commutative algebra over 
$\mathbb{Q}$. For any algebraic extension $L/\mathbb{Q}$, set $\mathcal{X}_{ 
\mathbb{T}}(L):=\Hom_{\mathbb{Q}\text{-}\func{alg}}(\mathbb{T},L)$. By an $L$
-valued system of eigenvalues we shall mean an element $\lambda \in \mathcal{%
\ X}_{\mathbb{T}}(L)$.

Let $H$ be a $\mathbb{Q}$-vector space endowed with a linear action of $%
\mathbb{T}$. Given $\lambda \in \mathcal{X}_{\mathbb{T}}(L)$, a $\lambda $
-eigenvector in $H$ is a non-zero element $f\in H\otimes _{\mathbb{Q}}L$
such that $T\cdot f=\lambda (T)f$ for all $T\in \mathbb{T}$; write $%
H_{\lambda}(L)$ for the subspace of $H\otimes _{\mathbb{Q}}L$ they span.
When $H_{\lambda}(L)\ne 0$, we say that $\lambda $ \emph{occurs} in $%
H\otimes _{ \mathbb{Q}}L$. If $\mathbb{T}\subset \func{End}_{\mathbb{Q}}(H)$
, all $\lambda \in\mathcal{X}_{\mathbb{T}}(L)$ occur.

The Galois group $G_{\mathbb{Q}}$ acts on $\mathcal{X}_{\mathbb{T}}( 
\overline{\mathbb{Q}})$ by composition. Given $\lambda \in \mathcal{X}_{ 
\mathbb{T}}(\overline{\mathbb{Q}})$, we write $[\lambda ]$ for the orbit of $%
\lambda $ under this action. Note that $\ker(\lambda _{1})=\ker(\lambda
_{2}) $ if and only if $\left[ \lambda _{1}\right] =\left[ \lambda _{2}%
\right]$; put $I_{[\lambda]}:=\ker(\lambda)\subset \mathbb{T}$.

Set $L_{\lambda }:=\lambda ( \mathbb{T})$ so that $\lambda\in\mathcal{X}_{ 
\mathbb{T}}(L_{\lambda })$, and $L_{[\lambda
]}=\tprod\nolimits_{\lambda^{\prime}\in \left[ \lambda \right]
}L_{\lambda^{\prime}}\subset \overline{\mathbb{Q}}$; $L_{[\lambda ]}/ 
\mathbb{Q}$ is a Galois extension. Set $H_{[\lambda ]}(L_{[\lambda
]}):=\oplus_{\lambda^{\prime}\in [\lambda]}
H_{\lambda^{\prime}}(L_{[\lambda]})$; an easy descent argument shows that
there exists a $\mathbb{T}$-submodule $H_{\left[ \lambda \right]}\subset H$
over $\mathbb{Q}$ such that $H_{\left[ \lambda \right] }(L_{\left[ \lambda %
\right] })=H_{\left[ \lambda \right]} \otimes_{\mathbb{Q}} L_{\left[ \lambda %
\right]}$.

Given $\lambda \in \mathcal{X}_{\mathbb{T}}(\qbar)$, let $\iota :H_{[\lambda
]}\subset H$ be the natural inclusion and let $(H^{\vee })^{\lambda
}=H^{\vee }/I_{[\lambda ]}\cdot H^{\vee }$ denote the maximal quotient of $%
H^{\vee }$ on which $\mathbb{T}$ acts through $\lambda $. Then there is a
canonical commutative diagram of $\mathbb{T}$ -modules with exact rows 
\begin{equation}
\begin{array}{ccccccccc}
0 & \rightarrow & I_{[\lambda ]}(H^{\vee }) & \rightarrow & H^{\vee } & 
\rightarrow & (H^{\vee })^{\lambda } & \rightarrow & 0 \\ 
&  &  &  & \parallel &  & \downarrow \simeq &  &  \\ 
&  &  &  & H^{\vee } & \overset{\iota ^{\vee }}{\rightarrow } & H_{[\lambda
]}^{\vee } & \rightarrow & 0\text{.}%
\end{array}
\label{diagram}
\end{equation}

Let now $\mathbb{T}=\mathbb{T}_{\Gamma _{0}(pN^{+})}^{p-new}\otimes \mathbb{%
Q }$ and let $\mathbf{H}^{c,w_{\infty }}=H^{1}(\Gamma ,\mathcal{C}_{har}( 
\mathbb{V}_{n}(\mathbb{Q})))^{c,w_{\infty }}$ be the module introduced in %
\eqref{defH}; note that $\func{End}(\mathbf{H}^{c,w_{\infty }})=\mathbb{T}$
by Remark \ref{rankone}. By Remark \ref{Remark on the rank 3}, Lemma \ref%
{Lemma harmonic cocycles/cusp forms}, the Jacquet-Langlands correspondence
and the $q$-expansion principle, $\dim _{L_{\lambda }}(\mathbf{H}_{\lambda
}^{c,w_{\infty }}\left( L_{\lambda }\right) )=1$ for all $\lambda \in 
\mathcal{X}_{\mathbb{T}}(\qbar)$.

Given a non-zero eigenvector $f$, write $\lambda _{f}$ for the corresponding
system of eigenvalues and put $L_{f}:=L_{\lambda _{f}}$, $I_{[f]}=I_{\left[
\lambda_f\right] }$ and $\mathbf{H}_{\left[ f\right] }^{c,w_{\infty }}= 
\mathbf{H}_{\left[ \lambda_f\right] }^{c,w_{\infty}}$.

Since the category of admissible filtered Frobenius modules over $\mathbb{Q}
_p$ is an abelian category and the elements of $I_{[f]}$ act on $\mathbf{D}$
, we can introduce the module $\mathbf{D}_{[f]}\in \mathrm{MF}_{\mathbb{Q}
_p}(\varphi ,N)$ as the one sitting in the exact sequence 
\begin{equation*}
0\rightarrow I_{[f]}\mathbf{D}\rightarrow \mathbf{D}\overset{\lambda _{f}}{
\rightarrow }\mathbf{D}_{[f]}\rightarrow 0\text{.}
\end{equation*}

Tensoring with $\mathbb{Q}_p$ over $\mathbb{Q}$ yields an exact sequence 
\begin{equation*}
0\rightarrow I_{[\lambda],p}\rightarrow \mathbb{T}_p\overset{\lambda_p}{
\rightarrow }L_{\lambda,p}\rightarrow 0.
\end{equation*}

Since $\mathbb{T}_{p}\subset \func{End}_{\func{MF}_{\mathbb{Q}
_{p}}^{ad}(\varphi ,N)}(\mathbf{D})$, we have $\mathbf{D}_{[f]}=\mathbf{D}
/I_{[\lambda ],p}\mathbf{D}$ and it follows that $\mathbf{D}_{[f]}$ is
canonically a two-dimensional monodromy $L_{f,p}$-module over $K_{p}$. Its ${%
\ \mathcal{L}}$-invariant is 
\begin{equation}
\mathcal{L}_{\left[ f\right] }:=\lambda _{f,p}(\mathcal{L})\in L_{f,p}
\end{equation}
and its $U$-invariant is $\lambda _{f,p}(U_{p})=a_{p}(f)=\pm p^{m}$. In the
notation of Proposition \ref{UL}, 
\begin{equation}
\mathbf{D}_{\left[ f\right] }=\mathbf{D}_{a_{p}(f),{\mathcal{L}}_{\left[ f %
\right] }}\text{.}
\end{equation}

Explicitly, $\mathbf{D}_{[f]}$ can be described as the filtered Frobenius
monodromy module over $\mathbb{Q}_p$ whose underlying vector space is 
\begin{equation}
\mathbf{D}_{[f]}=(\mathbf{H}^{c,w_{\infty },\vee }(\mathbb{Q}
_p))^{\lambda_f} \oplus (\mathbf{H}^{c,w_{\infty },\vee }(\mathbb{Q}
_p))^{\lambda_f} \simeq \mathbf{H}_{\left[ f\right] }^{c,w_{\infty },\vee }( 
\mathbb{Q}_p) \oplus \mathbf{H}_{\left[ f\right] }^{c,w_{\infty },\vee }( 
\mathbb{Q}_p),
\end{equation}
where the latter isomorphism arises from \eqref{diagram}. The filtration $%
F_{[f]}^{\cdot}$ is given as in Definition \ref{2dimMon}, where 
\begin{equation}
F_{[f]}^{m}=\left\{ (-\mathcal{L}_{\left[ f\right] }x,x):x\in \mathbf{H}_{ %
\left[ f\right] }^{c,w_{\infty },\vee }(\mathbb{Q}_p)\right\}.
\end{equation}

Fix a complete field extension $K_p/\mathbb{Q}_p$ as in the previous
sections and let $\mathbf{D}_{[f],K_p}$ denote the base change to $K_p$ of $%
\mathbf{D}_{[f]}$ in the sense of Remark \ref{basechange}. As in %
\eqref{PhiOnto} and \eqref{AJ2}, we can introduce the map 
\begin{equation}
\Phi _{\left[ f\right] }:H_{1}(\Gamma ,\func{Div}^{0}(\mathcal{H}_p)\otimes 
\mathbf{P}_{n})\overset{\Phi }{\longrightarrow }\mathbf{D}\otimes K_p\overset%
{\lambda _{f}}{\twoheadrightarrow }\mathbf{D}_{\left[ f\right] }\otimes
K_p\quad
\end{equation}
and the Abel-Jacobi map 
\begin{equation}
\Phi _{\left[ f\right] }^{AJ}:H_{1}(\Gamma ,\func{Div}(\mathcal{H}_p)\otimes 
\mathbf{P}_{n})\overset{\Phi ^{AJ}}{\rightarrow }\mathbf{D}\otimes K_p/F^{m} 
\overset{\lambda _{f}}{\twoheadrightarrow }\mathbf{D}_{\left[ f\right]
,K_p}/F_{\left[ f\right] }^{m}.
\end{equation}

Of course the monodromy module $\mathbf{D}_{\left[ f\right] }$ canonically
decomposes according to $L_{f,p}=\tbigoplus\nolimits_{\mathfrak{p}\mid
p}L_{f,\mathfrak{p}}$, where $L_{f,\mathfrak{p}}$ denotes the completion of $%
L_{f}$ at the prime $\mathfrak{p}$ above $p$: 
\begin{equation*}
\mathbf{D}_{\left[ f\right] }=\tbigoplus\nolimits_{\mathfrak{p}\mid p} 
\mathbf{D}_{\left[ f\right] ,\mathfrak{p}}\text{.}
\end{equation*}
In the notation of Proposition \ref{UL}, $\mathbf{D}_{\left[ f\right] , 
\mathfrak{p}}=\mathbf{D}_{a_{p}(f),{\mathcal{L}}_{\left[ f\right] ,\mathfrak{%
\ p}}}$, where ${\mathcal{L}}_{\left[ f\right] ,\mathfrak{p}}$ denotes the $%
\mathfrak{p}$-component of $\mathcal{L}_{\left[ f\right] }$. We can further
consider $\Phi _{\left[ f\right] ,\mathfrak{p}}$ as well as $\Phi _{\left[ f %
\right] ,\mathfrak{p}}^{AJ}$.

\section{Darmon cycles}

\label{S3}

\subsection{Construction of Darmon homology classes}

\label{S31}

The aim of this section is to introduce what we call \textit{Darmon cycles}, which should
be regarded as analogues of the classical Heegner cycles attached to
imaginary quadratic fields and weight $k\geq 4$ modular forms by Nekov\'a\v{r} (c.f.\,\cite{Ne}, \cite{IS}) and of Stark-Heegner points (also called Darmon points in \cite{LRV2}) attached to real quadratic
fields and weight $2$ modular forms (cf.\,\cite{Dar}, \cite{Gr}, \cite{LRV}, 
\cite{LRV2}).

Let $p$ be a prime and let $N$ be a positive integer such that $p\mid N$, $%
p^2\nmid N$. Let $K/\mathbb{Q} $ be a real quadratic field in which $p$
remains inert. Assume for simplicity that the discriminant $D_{K}$ of $K$ is
prime to $N$. This induces a factorization of $N$ as $N=pN^+N^-$, where $%
(N^+,N^-)=1$ and all prime factors of $N^+$ (respectively $N^-$) split
(resp. remain inert) in $K$.

Crucial for our construction is the following \emph{Heegner hypothesis} (see
also our general discussion in the introduction), which we assume for the
rest of this section.

\vspace{0.3cm}

\textbf{Assumption.} \emph{$N^-$ is the square-free product of an even
number of primes.}

\bigskip

In consonance with the notations introduced in section \ref{S1}, let $K_p$
denote the completion of $K$ at $p$, a quadratic unramified extension of $%
\mathbb{Q}_p$. Since this field shall be fixed during the whole section and
the maximal unramified subextension of $K_p$ is $k_p=K_p$ itself, we shall
simply write $\mathcal{H}_p$, $\func{Div}(\mathcal{H}_p)$ and $\mathbf{P}_n$
instead of $\mathcal{H}_p(K_p)$, $\func{Div}(\mathcal{H}_p)(K_p)$ and $%
\mathbf{P}_n(K_p)$, respectively.

Let $\mathcal{B}$ be the indefinite quaternion algebra of discriminant $N^-$
over $\mathbb{Q}$, $\mathcal{R}$ be a $\mathbb{Z}[1/p]$-Eichler order of
level $N^+$ in $\mathcal{B}$ and $\Gamma$ be the subgroup of $\mathcal{R}
^{\times }$ of elements of reduced norm $1$.

As in §\ref{S0}, fix an embedding $\mathcal{B}^{\times }\hookrightarrow 
\mathrm{GL}_{2}(\mathbb{Q}_{p})$, which allows us to regard $\Gamma $ as a
subgroup of ${\mathrm{SL}}_{2}(\mathbb{Q}_{p})$. Choose also embeddings 
\begin{equation*}
\sigma _{\infty }:K\rightarrow \mathbb{R}\text{ and }\sigma
_{p}:K\rightarrow K_{p}
\end{equation*}%
that we use to regard $K$ as a subfield of both $\mathbb{R}$ and of $K_{p}$.
In particular we have $D_{K}^{-\frac{k-2}{4}}\in K_{p}$ via $\sigma _{p}$.

Let us denote by $\mathrm{Emb}( K,\mathcal{B}) $ the set of all the $\mathbb{%
Q} $-algebra embeddings of $K$ into $\mathcal{B}$. Let $\mathcal{O}\subset K$
be a $\mathbb{Z}[1/p] $-order of conductor $c\geq 1$, $(c, N)=1$, and let $%
\mathrm{Emb}(\mathcal{O}, \mathcal{R})$ be the set of $\mathbb{Z} [ 1/p] $%
-optimal embeddings of $\mathcal{O}$ into $\mathcal{R}$. Attached to an
embedding $\Psi \in \mathrm{Emb}( \mathcal{O}, \mathcal{R})$ there is the
following data:

\begin{itemize}
\item the two fixed points $\tau _{\Psi },\overline{\tau }_{\Psi }\in 
\mathcal{H}_p\cap K$ for the action of $\Psi(K_p^{\times }) $ on $\mathcal{H}
_p\cap K$, labelled in such a way that the action of $K^{\times }$ on the
tangent space at $\tau _{\Psi }$ is given by the character $z\mapsto z/ 
\overline{z}$;

\item the unique vertex $v_{\Psi }\in \mathcal{V}$ which is fixed for the
action of $\Psi(K_p^{\times }) $ on $\mathcal{V}$; we have $v_{\Psi}= 
\mathrm{red}(\tau _{\Psi }) = \mathrm{red}(\overline{\tau }_{\Psi})$;

\item the unique polynomial up to sign $P_{\Psi }$\ in $\mathbf{P}_{2}$
which is fixed by the action of $\Psi \left( K_{p}^{\times }\right) $ on $%
\mathbf{P}_{2}$ and satisfies $\left\langle P_{\Psi },P_{\Psi }\right\rangle
_{\mathbf{P}_{2}}=-D_{K}/4$. We single out one by%
\begin{equation*}
P_{\Psi }:=\limfunc{Tr}(\Psi (\sqrt{D_{K}}/2)\cdot 
\begin{pmatrix}
X & -X^{2} \\ 
1 & -X%
\end{pmatrix}%
)\in \mathbf{P}_{2}\text{;}
\end{equation*}

\item the stabilizer $\Gamma _{\Psi }$ of $\Psi $ in $\Gamma $, that is, 
\begin{equation*}
\Gamma _{\Psi }=\Psi (K^{\times })\cap \Gamma =\Psi (\mathcal{O}_{1}^{\times
})
\end{equation*}%
where $\mathcal{O}_{1}^{\times }:=\{\gamma \in {\mathcal{O}}^{\times },%
\mathrm{n}(\gamma )=1\}$;

\item the generator $\gamma _{\Psi }:=\Psi (u)$ of $\Gamma _{\Psi }/\{\pm
1\}\simeq \mathbb{Z}$, where $u\in \mathcal{O}_{1}^{\times }$ is the unique
generator of $\mathcal{O}_{1}^{\times }/\{\pm 1\}$ such that $\sigma (u)>1$
if $\sigma (\tau _{\Psi })>\sigma (\bar{\tau}_{\Psi })$ and $\sigma (u)<1$
if $\sigma (\tau _{\Psi })<\sigma (\bar{\tau}_{\Psi }).$
\end{itemize}

For each $\tau \in \mathcal{H}_p$, we say that $\tau $ has positive
orientation at $p$ if $\mathrm{red}( \tau ) \in \mathcal{V}^+$. We write $%
\mathcal{H} _p^+$ to denote the set of positive oriented elements in $%
\mathcal{H} _p$. We say that $\Psi \in \mathrm{Emb}({\mathcal{O}}, \mathcal{%
R })$ has positive orientation whenever $v_{\Psi }\in \mathcal{V}^+$, i.e. $%
\tau _{\Psi },\overline{\tau }_{\Psi }\in \mathcal{H}_p^+\cap K$. Put 
\begin{equation*}
\mathrm{Emb}({\mathcal{O}}, \mathcal{R})=\mathrm{Emb}_+({\mathcal{O}}, 
\mathcal{R})\sqcup \mathrm{Emb}_-({\mathcal{O}}, \mathcal{R})
\end{equation*}
with the obvious meaning. The group $\Gamma $ acts on $\mathrm{Emb}({\ 
\mathcal{O}}, \mathcal{R})$ by conjugation, preserving orientations.

As in the previous sections, fix an even integer $k\geq 2$ and let $n=k-2$, $%
m=n/2$. The $\Gamma _{\Psi }$-module $K_{p}\cdot \tau _{\Psi }\otimes
D_{K}^{-\frac{k-2}{4}}P_{\Psi }^{m}\subset \func{Div}(\mathcal{H}%
_{p})\otimes \mathbf{P}_{n}$ is endowed with the trivial $\Gamma _{\Psi }$%
-action (see the computation \eqref{Formula
action of gamma on datas} below). Hence, the choice of the generator $\gamma
_{\Psi }$ for the cyclic group $\Gamma _{\Psi }$ allow us to fix an
identification $K_{p}=H_{1}(\Gamma _{\Psi },K_{p}\cdot \tau _{\Psi }\otimes
D_{K}^{-\frac{k-2}{4}}P_{\Psi }^{m})$. The inclusion $K_{p}\cdot \tau _{\Psi
}\otimes D_{K}^{-\frac{k-2}{4}}P_{\Psi }^{m}\subset \func{Div}(\mathcal{H}%
_{p})\otimes \mathbf{P}_{n}$ then induces the \emph{cycle class map} 
\begin{equation*}
cl_{\Psi }:K_{p}=H_{1}(\Gamma _{\Psi },K_{p}\cdot \tau _{\Psi }\otimes
D_{K}^{-\frac{k-2}{4}}P_{\Psi }^{m})\rightarrow H_{1}(\Gamma ,\func{Div}(%
\mathcal{H}_{p})\otimes \mathbf{P}_{n})\text{.}
\end{equation*}

The group $H_1( \Gamma ,\func{Div}(\mathcal{H}_p)\otimes \mathbf{P}_n) $
should be regarded as a substitute of the local Chow group in our real
quadratic setting. See §\ref{S0} for more on this analogy. With this in mind
we make the following definition.

\begin{definition}
The \emph{Darmon cycle} attached to an embedding $\Psi \in \mathrm{Emb}({\ 
\mathcal{O}}, \mathcal{R})$ is 
\begin{equation*}
y_{\Psi }:=cl_{\Psi }(1) \in H_1( \Gamma ,\func{Div}(\mathcal{H}_p)\otimes 
\mathbf{P}_n) \text{.}
\end{equation*}
\end{definition}

\begin{lemma}
\label{Lemma Darmon cycles} The homology class $y_{\Psi }\in H_1( \Gamma , 
\func{Div}(\mathcal{H}_p)\otimes \mathbf{P}_n)$ does not depend on the
choice of $\Psi$ in its conjugacy class of optimal embeddings for the action
of $\Gamma$.
\end{lemma}

\begin{proof}
Let $\gamma \in \Gamma $. The assignation $\Psi \mapsto (\tau _{\Psi
},P_{\Psi },\gamma _{\Psi })$ behaves under conjugation by $\gamma $ as 
\begin{equation}
(\tau _{\gamma \Psi \gamma ^{-1}},P_{\gamma \Psi \gamma ^{-1}},\gamma
_{\gamma \Psi \gamma ^{-1}})=(\gamma \tau _{\Psi },\gamma P_{\Psi }:=P_{\Psi
}\gamma ^{-1},\gamma \gamma _{\Psi }\gamma ^{-1})\text{.}
\label{Formula action of gamma on datas}
\end{equation}%
from what it follows that 
\begin{equation*}
cl_{\Psi }(1)=\tau _{\Psi }\otimes D_{K}^{-\frac{k-2}{4}}P_{\Psi
}^{m}\otimes \lbrack \gamma _{\Psi }]=\gamma \cdot \tau _{\Psi }\otimes
\gamma \cdot D_{K}^{-\frac{k-2}{4}}P_{\Psi }^{m}\otimes \lbrack \gamma
\gamma _{\Psi }\gamma ^{-1}]=cl_{\gamma \Psi \gamma ^{-1}}(1)\in
H_{1}(\Gamma ,\func{Div}(\mathcal{H}_{p})\otimes \mathbf{\ P}_{n}).
\end{equation*}
\end{proof}

As a consequence of Lemma \ref{Lemma Darmon cycles}, there is a well-defined
morphism 
\begin{equation*}
y: \Gamma \backslash \mathrm{Emb}({\mathcal{O}},\mathcal{R})\rightarrow
H_1(\Gamma ,\func{Div}(\mathcal{H}_p)\otimes \mathbf{P}_n)
\end{equation*}
attaching a Darmon cycle $y_{ [\Psi] }:=y_{\Psi }$ to any conjugacy class $[
\Psi ]$ of optimal embeddings. Invoke now the Abel-Jacobi map 
\begin{equation*}
\begin{array}{lll}
H_1( \Gamma ,\func{Div}( \mathcal{H}_p)\otimes \mathbf{P}_n) & \overset{\Phi
^{AJ}}{\rightarrow } & \mathbf{D}\otimes K_p/F^{m}%
\end{array}%
\end{equation*}
introduced in \eqref{AJ2}.

\begin{definition}
\label{defSH} The \emph{Darmon cohomology class} attached to $[\Psi]\in
\Gamma \backslash \mathrm{Emb}({\mathcal{O}},\mathcal{R})$ is $s_{[\Psi]}:=
\Phi ^{AJ}(y_{[\Psi]})\in \mathbf{D}\otimes K_p/F^{m}$.
\end{definition}

\begin{remark}
Sometimes, when one is only interested in the conductor $c=c({\mathcal{O}})$
of the quadratic order ${\mathcal{O}}$ and not on the specific optimal
embedding that is being used in the construction of the cycle, one may
simply write $y_{c}$ and $s_{c}$ instead of $y_{[\Psi ]}$ or $s_{[\Psi ]}$.
\end{remark}

\begin{remark}
This construction can also be formulated from a different (but equivalent)
point of view, which reinforces the analogy with the classical case of
imaginary quadratic fields. Namely, let $\mathcal{H}_p^{{\mathcal{O}}}=\{
\tau \in \mathcal{H}_p: \tau=\tau_{\Psi} \text{ for some } \Psi\in \mathrm{\
Emb}({\mathcal{O}},\mathcal{R})\}$. Note that there is a well-defined action
of $\Gamma$ on $\mathcal{H}_p^{{\mathcal{O}}}$. With this notation, the
above formalism yields a map 
\begin{equation}  \label{y}
d: \Gamma\backslash \mathcal{H}_p^{{\mathcal{O}} } \overset{y}{{\
\longrightarrow }} H_1( \Gamma ,\func{Div}(\mathcal{H}_p)\otimes \mathbf{P}
_n) \overset{\Phi^{AJ}}{{\longrightarrow }} \mathbf{D}\otimes K_p/F^{m}.
\end{equation}
\end{remark}

\begin{remark}
For every prime $\ell\mid pN^+N^-$, let $\omega _{\ell}\in \mathcal{R}%
_0(N^+) $ be an element of reduced norm $\ell$ lying in the normalizator of $%
\Gamma $ . Conjugation by $\omega_{\ell }$ induces an involution $W_{\ell }$
on $\Gamma \backslash \mathrm{Emb}({\mathcal{O}},\mathcal{R})$ given by $%
W_{\ell}(\Psi) = \omega_{\ell }\Psi \omega_{\ell }^{-1}$.

Besides, conjugation by $\omega_{\ell }$ also induces an involution $%
W_{\ell} $ both on $H_1( \Gamma ,\func{Div}(\mathcal{H}_p)\otimes \mathbf{P}%
_n)$ and on $\mathbf{D}\otimes K_p/F^{m}$, as already mentioned in §\ref{S11}%
. It follows as in the proof of Lemma \ref{Lemma Darmon cycles} and the
Hecke equivariance of $\Phi ^{AJ}$ that there are commutative diagrams 
\begin{equation*}
\begin{array}{ccccccc}
d: \Gamma \backslash \mathrm{Emb}({\mathcal{O}},\mathcal{R}) & {\rightarrow}
& H_1( \Gamma ,\func{Div}(\mathcal{H}_p)\otimes \mathbf{P}_n) & \overset{
\Phi^{AJ}}{\rightarrow} & \mathbf{D}\otimes K_p/F^{m} &  &  \\ 
\text{ \ }\downarrow W_{\ell} &  & \text{ \ \ \ \ \ \ \ \ \ \ \ \ }
\downarrow W_{\ell} &  & \text{ \ \ } \downarrow W_{\ell} &  &  \\ 
d: \Gamma \backslash \mathrm{Emb}({\mathcal{O}},\mathcal{R}) & {\rightarrow}
& H_1( \Gamma ,\func{Div}(\mathcal{H}_p)\otimes \mathbf{P}_n) & \overset{
\Phi^{AJ}}{\rightarrow} & \mathbf{D}\otimes K_p/F^{m}\text{.} &  & 
\end{array}%
\end{equation*}
\end{remark}

Recall that an orientation on the Eichler order $\mathcal{R}$ (resp. on the
quadratic order ${\mathcal{O}}$) is the choice, for each $\ell \mid N^+N^-$,
of a ring homomorphism $\mathcal{R}\,{\rightarrow }\,k_{\ell }$ (resp. $%
\mathcal{O}\,{\rightarrow }\,k_{\ell }$), where $k_{\ell }= \mathbb{F}_{\ell
^{2}}$ (resp. $k_{\ell}=\mathbb{F}_{\ell}$) for $\ell \mid N^-$ (resp. $\ell
\mid N^+$).

Fix orientations both on ${\mathcal{O}}$ and on $\mathcal{R}$. An optimal
embedding $\Psi: {\mathcal{O}} {\rightarrow } \mathcal{R}$ is \emph{oriented}
if, for all $\ell \mid N^+ N^-$, $\Psi\otimes k_{\ell}$ commutes with the
chosen local orientations on ${\mathcal{O}}\otimes k_{\ell}$ and $\mathcal{R}
\otimes k_{\ell}$, respectively. Write $\overrightarrow{\mathrm{Emb}}_{+}({\ 
\mathcal{O}},\mathcal{R}) \subset \mathrm{Emb}_+({\mathcal{O}},\mathcal{R})$
for the set of oriented positive optimal embeddings. The action of $\Gamma$
on $\mathrm{Emb}_+({\mathcal{O}},\mathcal{R})$ leaves $\overrightarrow{ 
\mathrm{Emb}}_{+}({\mathcal{O}},\mathcal{R})$ stable and thus induces a
well-defined action on it.

By Eichler's theory of optimal embeddings, $\overrightarrow{\mathrm{Emb}}_+({%
\ \mathcal{O}},\mathcal{R})$ is not empty and the quotient $\Gamma
\backslash \overrightarrow{\mathrm{Emb}}_+({\mathcal{O}},\mathcal{R})$ is
endowed with a free transitive action of the narrow class group ${\mathrm{Pic%
}}({\mathcal{\ O}})$ of the $\mathbb{Z}[\frac{1}p]$-order ${\mathcal{O}}$
(see for example \cite[Lemma 2.5]{BD} for related result in the setting of
imaginary quadratic fields). Denote this action by 
\begin{equation*}
([ \mathfrak{a}] ,[ \Psi ] )\,\mapsto \,[ \mathfrak{a} \star \Psi ], \quad 
\text{ for } [\mathfrak{a}]\in {\mathrm{Pic}}({\mathcal{O}}), \Psi\in 
\overrightarrow{\mathrm{Emb}}_+({\mathcal{O}}, \mathcal{R}).
\end{equation*}

Artin's reciprocity map of global class field theory provides an isomorphism 
\begin{equation*}
\mathrm{rec}:\mathrm{Pic}( \mathcal{O}) \overset{\simeq }{\rightarrow } 
\func{Gal\,}(H_{\mathcal{O}}/K),
\end{equation*}
where $H_{ \mathcal{O}}$ stands for the narrow ring class field attached to
the $\mathbb{Z}$-order ${\mathcal{O}}_0$ of $K$ which is locally maximal at $%
p$ and ${\mathcal{O}}_0[\frac{1}p]=\mathcal{O}$.

In order to state our conjectures it is convenient to introduce the
following linear combinations of Darmon cycles.

\begin{definition}
\label{Definition twisted Abel-Jacobi image of Darmon cycles} Let $\chi: 
\func{Gal\,}(H_{\mathcal{O}}/K)\rightarrow \mathbb{C} ^{\times }$ be a
character. The \emph{Darmon cycle} attached to the character $\chi $ is 
\begin{equation*}
y_{\chi }:=\tsum\nolimits_{\sigma \in \func{Gal\,}(H_{\mathcal{O}}/K)}\chi
^{-1}( \sigma ) y_{[ \mathrm{rec}^{-1}(\sigma) \star \Psi ] }\in H_1( \Gamma
, \func{Div}(\mathcal{H}_p)\otimes \mathbf{P}_n) \otimes K_p(\chi),
\end{equation*}
where $[ \Psi ]$ is any choice of a class of optimal embeddings in $\Gamma
\backslash \overrightarrow{\mathrm{Emb}}_+({\mathcal{O}},\mathcal{R})$ and $%
K_p(\chi)$ is the field generated by the (algebraic) values of $\chi$ over $%
K_p$.

Write 
\begin{equation*}
s_{\chi }:=\Phi ^{AJ}(s_{\chi })\in \mathbf{D}\otimes K_{p}(\chi
)/F^{m}\otimes K_{p}(\chi ).
\end{equation*}
\end{definition}

\subsection{A conjecture on the global rationality of Darmon cycles}

\label{S34}

Keep the notations and hypothesis of §\ref{S31}. As in §\ref{S0} and §\ref%
{S41}, let $V_p:=H_p( \mathcal{M}_n)^{p-new}$, which we regard this time as
a semistable continuous representation of $G_{K_p}$, by restricting the
action of $G_K\subset G_{\mathbb{Q}}$ to the decomposition subgroup of a
prime $\bar \wp$ of $\qbar$ over $p$.

In this section we show how Conjecture \ref{Conjecture equality of the
L-invariants 1} allows us to attach to each Darmon cycle $y_{\Psi}$ a class $%
s_{\Psi}$ in the group $H_{st}^1(K_p,V_p)$ of local semistable cohomology
classes. Cf.\,\eqref{Hst}, or rather \cite{Ne2}, for the definition of this
group.

In \cite{BK}, Bloch and Kato introduced an exponential map which, in the
case which concerns us here, induces an isomorphism 
\begin{equation}  \label{Conjectures Lemma 2}
\exp :\frac{\mathbf{D}^{FM}\otimes K_p}{\limfunc{Fil}^{m}(\mathbf{D}
^{FM}\otimes K_p)}\overset{\simeq }{\rightarrow }H_{st}^1(K_p,V_p),
\end{equation}
as it follows from \cite[Lemma 2.1]{IS}.

Keeping the notation of §\ref{SMonodromy}, assume Conjecture \ref{Conjecture
equality of the L-invariants 1} and fix an isomorphism $\mathbf{D}
^{FM}\simeq \mathbf{D}$ of two-dimensional monodromy $\mathbb{T}_p$-modules
over $\mathbb{Q}_p$. The choice of this isomorphism induces an
identification 
\begin{equation}  \label{ident}
\frac{\mathbf{D}\otimes K_p}{\limfunc{Fil}^{m}(\mathbf{D}^{FM}\otimes K_p)}
= \frac{\mathbf{D}^{FM}\otimes K_p}{\limfunc{Fil}^{m}(\mathbf{D}\otimes K_p)}
\text{.}
\end{equation}

In view of \eqref{Conjectures Lemma 2} and \eqref{ident}, we may regard the
Darmon cohomology classes introduced above as cocycles 
\begin{equation}  \label{lcc}
s_{\Psi}\in H_{st}^1(K_p,V_p), \quad s_{\chi}\in H_{st}^1(K_p(\chi),V_p),
\end{equation}
for any optimal embedding $\Psi \in \mathrm{Emb}({\mathcal{O}},\mathcal{R})$
and any character $\chi: \func{Gal\,}(H_{\mathcal{O}}/K)\rightarrow \mathbb{%
C }^{\times }$, respectively.

Since $p$ does not divide the conductor of the $\mathbb{Z}$-order ${\mathcal{%
\ O}}_0$ and the ideal $p {\mathcal{O}}_K$ is principal, it splits
completely in the narrow ring class field $H_{{\mathcal{O}}}$. Choose and
fix once for all an embedding 
\begin{equation*}
\iota_p:H_{{\mathcal{O}}}\hookrightarrow K_p.
\end{equation*}

The choice of $\iota_p$ induces a a restriction morphism 
\begin{equation*}
\func{res}_p: \mathrm{MW}_{st}(H_{{\mathcal{O}}},V_p) {\longrightarrow }
H^1_{st}(K_p,V_p) \simeq \frac{\mathbf{D}^{FM}\otimes K_p}{\limfunc{Fil}%
^{m}( \mathbf{D}^{FM}\otimes K_p)} \simeq \frac{\mathbf{D}\otimes K_p}{%
\limfunc{Fil }^{m}(\mathbf{D}\otimes K_p)}
\end{equation*}
as in \eqref{Diagram Abel-Jacobi maps}.

The image of the global Mordell-Weil group $\mathrm{MW}_{\ast }(H_{{\mathcal{%
\ O}}},V_p)$ is a $\mathbb{T}_p$-submodule of $\frac{\mathbf{D}\otimes K_p}{ 
\limfunc{Fil}^{m}(\mathbf{D}\otimes K_p)}$. By Lemma \ref{Lemma
endomorphisms of a monodromy module} every automorphism of $\mathbf{D}$ acts
on $\mathbf{D/ }\limfunc{Fil}^{m}\mathbf{D}$ by multiplication by an element
in $\mathbb{T}_p$. It follows that the image of $\mathrm{MW}_{\ast }(H_{{\ 
\mathcal{O}}},V_p)$ in $\frac{\mathbf{D}\otimes K_p}{\limfunc{Fil}^{m}( 
\mathbf{D}\otimes K_p)}$ does not depend on the choice of the isomorphism $%
\mathbf{D}\simeq \mathbf{D}^{FM}$.

\begin{conjecture}
\label{Conjecture rationality}

\begin{enumerate}
\item For any optimal embedding $\Psi\in \mathrm{Emb}({\mathcal{O}},\mathcal{%
\ R})$ there is a global cohomology class $\underline{s}_{\Psi}\in \mathrm{MW%
} _{\ast }(H_{{\mathcal{O}}},V_p) $ such that 
\begin{equation*}
s_{\Psi}=\func{res}_p(\underline{s}_{\Psi})
\end{equation*}

\item For any $\Psi\in\overrightarrow{\mathrm{Emb}}_+({\mathcal{O}},\mathcal{%
\ R})$ and any ideal class $\mathfrak{a}\in {\mathrm{Pic}}({\mathcal{O}})$, 
\begin{equation*}
\func{res}_p({}^{\sigma}\underline{s}_{\Psi }) = s_{\mathfrak{a}\star \psi},
\end{equation*}
where $\sigma = \mathrm{rec}(\mathfrak{a})^{-1}\in \func{Gal\,}(H_{{\mathcal{%
\ O}}}/K)$.

\vspace{0.2cm}

\item For any character $\chi: \func{Gal\,}(H_{\mathcal{O}}/K)\rightarrow 
\mathbb{C} ^{\times }$, $s_{\chi}=\func{res}_p(\underline{s}_{\chi})$ for
some $\underline{s}_{\chi}\in \mathrm{MW}^{\chi}(H_{\chi },V_p)$, where $%
H_{\chi }/K$ is the abelian sub-extension of $H_{\mathcal{O}}/K$ cut out by $%
\chi$, and $\mathrm{MW}^{\chi}(H_{\chi },V_p)$ stands for its $\chi$
-isotypical subspace.
\end{enumerate}
\end{conjecture}

Notice that (3) is an immediate consequence of (1) and (2) above. Let $f\in
S_{k}(\Gamma _{0}(pN^{+}))^{p-new}$ be a $p$-new eigenform. By means of the $%
p$-adic Abel-Jacobi map $\Phi _{\left[ f\right] }^{AJ}$ introduced in 
\eqref{AJ on the
f-component}, we can specialize the above constructions to the $f$
-eigencomponent $V_{p}(f)$ of $V_{p}$. Conjecture \ref{Conjecture
rationality} then predicts the existence of global cohomology classes 
\begin{equation}
\underline{s}_{\Psi ,f}\in \mathrm{MW}(H_{{\mathcal{O}}},V_{p}(f))\text{ and 
}\underline{s}_{f}^{\chi }\in \mathrm{MW}(H_{\chi },V_{p}(f))^{\chi },
\label{sf}
\end{equation}
such that $\func{res}_{p}(\underline{s}_{\Psi ,f})=s_{\Psi ,f}$, $\func{res}
_{p}(\underline{s}_{f}^{\chi })=s_{\Psi ,f}^{\chi }$ and satisfying an
explicit reciprocity law as in Conjecture \ref{Conjecture rationality} (2).
The $p$-adic representation $V_{p}(f)$ canonically decomposes according to $%
L_{f,p}=\tbigoplus\nolimits_{\mathfrak{p}\mid p}L_{f,\mathfrak{p}}$, where $%
L_{f,\mathfrak{p}}$ denotes the completion of $L_{f}$ at the prime $%
\mathfrak{p}$ above $p$: 
\begin{equation*}
V_{p}(f)=\tbigoplus\nolimits_{\mathfrak{p}\mid p}V_{\mathfrak{p}}(f)\text{.}
\end{equation*}
Write $\underline{s}_{f,\mathfrak{p}}^{\chi }$ for the corresponding
conjectural cohomology class and write $s_{f,\mathfrak{p}}^{\chi }$ for the
one obtained from the Darmon cohomology classe $s_{f}^{\chi }$ (that can be
directly defined by means of $\Phi _{\left[ f\right] ,\mathfrak{p}}^{AJ}$).
In light of the results achieved in \cite{Ne} for classical Heegner cycles
in the imaginary quadratic setting, it seems reasonable to formulate the
following conjecture.

\begin{conjecture}
\label{Conjecture non-trivial cases} Assume $s_{f,\mathfrak{p}}^{\chi }\neq
0 $. Then 
\begin{equation*}
\mathrm{MW}(H_{\chi },V_{\mathfrak{p}}(f))^{\chi }=L_{f,\mathfrak{p}} 
\underline{s}_{f,\mathfrak{p}}^{\chi }\text{.}
\end{equation*}
\end{conjecture}

Note that Conjecture \ref{Conjecture non-trivial cases} predicts that,
although the map $\func{res}_{p}$ above may not be injective, the global
cohomology class $\underline{s}_{f}^{\chi }$ is well-determined whenever $%
s_{f}^{\chi }\neq 0$. In particular, $\func{res}_{p}$ would induce an
isomorphism 
\begin{equation*}
\mathrm{MW}(H_{\chi },V_{\mathfrak{p}}(f))^{\chi }=L_{f,\mathfrak{p}} 
\underline{s}_{f,\mathfrak{p}}^{\chi }\overset{\simeq }{\rightarrow }L_{f, 
\mathfrak{p}}s_{f,\mathfrak{p}}^{\chi }\text{.}
\end{equation*}

It is also possible to formulate Gross-Zagier type conjectures for these
cycles, although a proof of them seems to be a long way off, as even their
classical counterparts for classical Heegner cycles remain completely open.

\begin{conjecture}
\label{Conjecture Gross-Zagier type} 
\begin{equation*}
\underline{s}_{f,\mathfrak{p}}^{\chi }\neq 0\text{ }\Longleftrightarrow 
\text{ }L^{\prime }(f/K,\chi ,k/2)\neq 0
\end{equation*}
and in particular 
\begin{equation*}
s_{f,\mathfrak{p}}^{\chi }\neq 0\text{ }\Rightarrow \text{ }L^{\prime
}(f/K,\chi ,k/2)\neq 0\text{.}
\end{equation*}
\end{conjecture}

Note that the second statement in the above conjecture makes sense even when
it is not known that there exists a global cohomology class $\underline{s}
_{f}^{\chi }$ inducing $s_{f}^{\chi }$ as predicted by Conjecture \ref%
{Conjecture rationality}. See \cite{LRV2} for a proof of an avatar of this
formula for Darmon points, where $k=2$, $s_{f}^{\chi }$ is replaced by its
image on a suitable group of connected components and $L^{\prime }(f/K,\chi
,1)$ is replaced by the (comparatively much simpler) special value $%
L(f_{0}/K,\chi ,1)$ of the $L$-function of an eigenform $f_{0}\in
S_{2}(\Gamma _{0}(N^{+}))$.

\section{Particular cases}

\label{A}

The circle ideas of this manuscript specialize, in the particular cases of $%
k=2$ or $N^-=1$, to scenarios which can be tackled by means of finer,
simpler methods, as we now describe.

For $k=2$ and any $N^-\geq 1$, the $p$-adic integration theory of §\ref{S2}
admits a much finer \emph{multiplicative} version, which allows to introduce 
$p$-adic Darmon points on Jacobians of Shimura curves, as envisaged by first
time by Darmon in \cite{Dar}, and completed later in \cite{Das}, \cite{Gr}, \cite{DG},
\cite{LRV} and \cite{LRV2}. We briefly recall these results in §\ref{A1}
below.

For $N^-=1$ and any $k\geq 2$, the presence of cuspidal points on the
classical modular curve $X_0(pN^+)$ allows for a re-interpretation of the
whole theory in terms of modular symbols. This was again first visioned by
Darmon in \cite{Dar}, for $k=2$.

In §\ref{A2} we develop the theory for $k>2$ in the language of modular symbols, in a way that shall be employed in the forthcoming work \cite{Se}. From this point of view, the $p$
-adic integration theory may be viewed as a lift of Orton's integration
theory \cite{Or}. By means of this approach, the work
\cite{Se} of the second author offers a nontrivial partial result towards the conjectures
posed in §\ref{S34} for $N^{-}=1$.

Below, we treat separately the cases $k=2$ and $N^-=1$. For the sake of
simplicity in the exposition, we leave aside the overlapping case $N^-=1$, $%
k=2$: this is the setting considered in the original paper \cite{Dar} of
Darmon, where both methods converge.

\subsection{The case $N^->1$ and $k=2$}

\label{A1}

Assume, only for this section, that $N^->1$ and $k=2$. Thus $m=n=0$. We
proved in Theorem \ref{Theorem L-invariant} that there is a surjective
homomorphism 
\begin{equation*}
\mathrm{pr}_c\cdot\Psi^{\func{ord}}\cdot\partial_2:H_{2}( \Gamma ,K_p)
\rightarrow (\mathbf{H}^c)^{\vee}
\end{equation*}
which yields an isomorphism when restricted to $H_{2}( \Gamma ,K_p)^c$.
Notice that, since $k=2$, $\mathbf{H}^{Eis}$ is not trivial. However, since $%
N^->1$, it follows from \eqref{Diagram Eichler-Shimura isomorphism} and
Lemma \ref{Lemma harmonic cocycles/cusp forms} that $\mathbf{H}^c\simeq H^1(
\Gamma _0(pN^+),K_p)^{p-new}$.

The theory developed in \cite{Gr}, \cite{DG} and \cite{LRV} shows that there is a
multiplicative refinement of the above, as we now briefly recall. Setting 
\begin{equation*}
T^{\star}(K_p):=\Hom( H_1( \Gamma _0(pN^+),\mathbb{Z} ) ^{p-new},K_p^{\times
}),
\end{equation*}
it is shown in \cite[§5]{LRV} that there is a Hecke-equivariant
multiplicative integration map 
\begin{equation*}
\Phi^0 :H_1( \Gamma ,\func{Div}^0( \mathcal{H}_p) ) \rightarrow
T^{\star}(K_p)
\end{equation*}
such that $\Phi^{\func{ord}}=\func{ord}\cdot \Phi^0$ and $\Phi^{\log}=\log
\cdot \Phi^0$, up to extending scalars from $\mathbb{Z}$ to $K_p $.
Similarly as in \eqref{partialp}, let 
\begin{equation*}
\partial_2^0 :H_{2}( \Gamma,\mathbb{Z}) \rightarrow H_1( \Gamma ,\func{Div}
^0( \mathcal{H}_p) )
\end{equation*}
denote the boundary morphism such that $\partial_2=\partial_2^0\otimes_{ 
\mathbb{Z} }K_p$ and one may define 
\begin{equation*}
L_0:=\func{Im}( \Phi^0 \cdot \partial_2^0 ) \subset T^{\star}(K_p)
\end{equation*}

It is shown in \cite[§6]{LRV} that $L_0$ is a lattice in $T^{\star}(\mathbb{%
\ \ Q }_p)$; hence, one may define the rigid analytic torus $J:=\frac{
T^{\star}( \mathbb{Q}_p)}{L_0}$ over $\mathbb{Q}_p$. There is a natural
action of the involution $W_{\infty}$ on $J$ which allows to split the torus 
$J\sim J^+\times J^-$ up to an isogeny of $2$-power degree. Both factors $%
J^+ $ and $J^-$ are in fact isogenous and the main results of \cite{DG} and \cite{LRV}
show that $J^+$ admits a Hecke-equivariant isogeny with $\mathrm{Jac}
(X^{N^-}_0(pN^+))^{p-new}$ over the quadratic unramified extension $K_p$ of $%
\mathbb{Q}_p$. This is achieved by proving that the ${\mathcal{L}}$
-invariants of $J^+$ and of $\mathrm{Jac}(X^{N^-}_0(pN^+))^{p-new}$, in the
sense of Tate-Morikawa's uniformization theory, are equal. This is a
particular instance of Conjecture \ref{Conjecture equality of the
L-invariants 1} above.

Let now $K$ be a real quadratic field in which $p$ is inert, so that $K_p$
is isomorphic to the completion of $K$ at $p$. Using the above results, it
is possible to attach a Darmon point $y_{\Psi }\in \mathrm{Jac}
(X^{N^-}_0(pN^+))^{p-new}(K_p)$ to each optimal embedding $\Psi:{\mathcal{O}}
\hookrightarrow \mathcal{R}$ as in §\ref{S31}; see \cite[§10]{Gr}, \cite{DG} and \cite[§
3]{LRV2} for full details and for the precise statement of the conjecture
which is the analogue of Conjecture \ref{Conjecture rationality}.

\subsection{The case $N^-=1$ and $k>2$.}

\label{A2}

Let $K_p/\mathbb{Q}_p$ be a complete field extension. Let $\Delta :=\limfunc{
Div}\mathbb{P}^1(\mathbb{Q}) $ and $\Delta ^0:= \limfunc{Div}^0\mathbb{P}^1( 
\mathbb{Q})$ be respectively the space of divisors and degree zero divisors
supported ot the cusps with coefficients in our working field $K_p$, so that 
\begin{equation}
0\rightarrow \Delta ^0\rightarrow \Delta \rightarrow K_p\rightarrow 0\text{.}
\label{Exact sequence degree zero divisors cusps}
\end{equation}

For any $K_p$-vector space $A$ endowed with an action by $G\subset \mathrm{\
GL }_{2}(\mathbb{Q})$ set $\mathcal{BS}(A):=\Hom(\Delta,A)$ and $\mathcal{MS}
(A):=\Hom( \Delta^0,A)$, endowed with the natural induced actions. Then
there is a canonical exact sequence 
\begin{equation}
0\rightarrow A\rightarrow \mathcal{BS}( A) \rightarrow \mathcal{MS}( A)
\rightarrow 0\text{.}  \label{Exact sequence boundary/modular symbolds}
\end{equation}
We also write $\mathcal{BS}_{G}(A):=\mathcal{BS}(A)^{G}$ and $\mathcal{MS}
_{G}(A):=\mathcal{MS}(A)^{G} $ to denote the $G$-invariants.

When $A=\mathcal{D}_{n}^{0}(\mathbb{P}^{1}(\mathbb{Q}_{p}),K_{p})^{b}$ and $%
A=\mathcal{C}_{har}(\mathbf{V}_{n}(K_{p}))$, the corresponding exact
sequences are connected by the morphisms induced by the morphism $r$
introduced in §\ref{S21}. Taking the long exact sequences induced in $\Gamma 
$-cohomology we find:

\begin{proposition}
\label{prop} There is a commutative diagram 
\begin{equation}
\begin{array}{ccc}
\mathcal{MS}_{\Gamma }( \mathcal{D}_n^0(\mathbb{P}^1(\mathbb{Q}_p),K_p)^{b})
& \overset{\delta }{\rightarrow } & H^1( \Gamma ,\mathcal{D}_n^0(\mathbb{P}
^1( \mathbb{Q}_p))^{b},K_p) \\ 
r\downarrow \text{ \ } &  & \text{ \ }\downarrow r \\ 
\mathcal{MS}_{\Gamma }(\mathcal{C}_{har}(\mathbf{V}_n(K_p))) & \overset{
\delta }{ \rightarrow } & H^1( \Gamma ,\mathcal{C}_{har}(\mathbf{V}_n(K_p)))%
\end{array}
\label{Diagram modular symbols/cocycles}
\end{equation}
where both vertical maps $r$ and the cuspidal part $\delta^c$ of the lower
horizontal map are isomorphisms.
\end{proposition}

\begin{proof}
By Theorem \ref{2parts}, the right vertical arrow is an isomorphism.
Exploiting the isomorphisms 
\begin{equation}  \label{ShapiroMS}
\begin{matrix}
\mathcal{MS}_{\Gamma }( \mathcal{C}_{har}(\mathbf{V}_n(K_p))) & 
\hookrightarrow & \mathcal{MS}_{\Gamma }( \mathcal{C}_0(\mathcal{E},\mathbf{V%
}_n(K_p))) \\ 
\downarrow \simeq &  & \simeq \downarrow \\ 
\mathcal{MS}_{\Gamma_0(pN^+)}(\mathbf{V}_n(K_p))^{p-new} & \hookrightarrow & 
\mathcal{MS}_{\Gamma_0(pN^+)}(\mathbf{V}_n(K_p))%
\end{matrix}%
\end{equation}
provided by Shapiro's lemma as in \eqref{shp}, a similar but simpler
argument shows that the left vertical arrow is also an isomorphism (see \cite%
[Proposition 2.8]{Se} for details). As for the lower horizontal arrow, the
Eichler-Shimura isomorphism factors as the composition 
\begin{equation}
ES:S_{k}( \Gamma _0( pN^+) ) \otimes _{\mathbb{R} }\mathbb{C} \overset{
\simeq }{\rightarrow }\mathcal{MS}_{\Gamma_0(pN^+)}(\mathbf{V}_n(\mathbb{C}
))^{c}\overset{\delta ^{c}}{\rightarrow }H^1( \Gamma _0(pN^+),\mathbf{V}
_n)^{c}\text{.}  \label{Diagram factorization of ES}
\end{equation}
Here the morphism $\delta $ appearing in \eqref{Diagram factorization
of ES} is obtained from \eqref{Exact sequence boundary/modular
symbolds} with $A=\mathbf{V}_n(\mathbb{C})$ and $G=\Gamma _0( pN^+)$. It
follows from this description that the morphism $\delta ^{c}$ obtained from $%
\delta $ in \eqref{Diagram modular symbols/cocycles} is identified with the
morphism obtained from $\delta ^{c}$ in $( \text{\ref{Diagram factorization
of ES}}) $ by taking the $p$-new parts. Since $ES$ is an isomorphism, $%
\delta ^{c}$ in $( \text{\ref{Diagram factorization of ES}} ) $ is an
isomorphism and the $p$-new parts of the source and the target are
identified by the Hecke equivariance of $\delta ^{c}$; for this reason $%
\delta $ in \eqref{Diagram modular symbols/cocycles} induces an isomorphism
between the cuspidal parts.

Finally, the commutativity of the diagram follows from a rather tedious but
elementary diagram-chasing computation.
\end{proof}

Set $\mathbf{MS}(K_{p}):=\mathcal{MS}_{\Gamma }(\mathcal{C}_{har}(\mathbf{V}
_{n}(K_{p})))$. Since the boundary morphisms $\delta $ are Hecke equivariant
the induce an isomorphism $\delta ^{c}:\mathbf{MS}(K_{p})^{c}\overset{\simeq 
}{\rightarrow }\mathbf{H}(K_{p})^{c}$ between the cuspidal parts. There is a
commutative diagram 
\begin{equation}
\begin{array}{ccccc}
H_{2}(\Gamma ,\mathbf{P}_{n}(K_{p})) & \overset{\partial _{2}}{\rightarrow }
& H_{1}(\Gamma ,\func{Div}^{0}(\mathcal{H}_{p})(k_{p})\otimes \mathbf{P}
_{n}(K_{p})) & \overset{\Psi ^{\log },\Psi ^{\func{ord}}}{\longrightarrow }
& \mathbf{H}(K_{p})^{\vee } \\ 
\downarrow &  & \downarrow \text{ } &  & \text{ \ }\downarrow \delta ^{\vee }
\\ 
H_{1}(\Gamma ,\Delta ^{0}\otimes \mathbf{P}_{n}(K_{p})) & \overset{\partial
_{1}}{\rightarrow } & (\Delta ^{0}\otimes \limfunc{Div}\nolimits^{0}( 
\mathcal{H}_{p})(k_{p})\otimes \mathbf{P}_{n}(K_{p}))_{\Gamma } & \overset{
\Psi _{\mathcal{MS}}^{\log },\Psi _{\mathcal{MS}}^{\func{ord}}}{
\longrightarrow } & \mathbf{MS}(K_{p})^{\vee }\text{.}%
\end{array}
\label{Diagram comparing theories 1}
\end{equation}

Here, the morphisms $\Psi ^{\log },\Psi ^{\func{ord}}$ in the top row are
the ones already introduced in Section \ref{S22}. Similarly, the
corresponding maps $\Psi _{\mathcal{MS}}^{\log },\Psi _{\mathcal{MS}}^{\func{
}}$ in the lower row are obtained by performing the obvious formal
modifications in the definition of the pairings in Definition \ref%
{defpairings} and in \eqref{Psi}.

The connecting map $\partial_1$ arises from the long exact sequence in
homology associated to 
\eqref{Exact sequence degree zero divisors p-adic
upper halfplane}, tensored with $\Delta ^0\otimes \mathbf{P}_n(K_p)$. Quite
similarly, the first (second) vertical arrow is the connecting map arising
in the long exact sequence induced by the short exact sequence 
\eqref{Exact sequence degree zero divisors
cusps} tensored with $\mathbf{P}_n(K_p)$ (respectively, tensored with $%
\limfunc{ Div}\nolimits^0( \mathcal{H}_p)(k_p) \otimes \mathbf{P}_n(K_p)$).

By Theorem \ref{Theorem L-invariant}, $(\Psi ^{\func{ord}}\circ \partial
_{2})^{c}:H_{2}(\Gamma ,\mathbf{P}_{n}(K_{p}))^{c}\overset{\simeq }{{\
\longrightarrow }}(\mathbf{H}(K_{p})^{c})^{\vee }$ is an isomorphism. The
same circle of ideas appearing in the proof of this theorem, paying care to
the Eisenstein subspaces, shows that 
\begin{equation}
(\Psi _{\mathcal{MS}}^{\func{ord}}\circ \partial _{1})^{c}:H_{1}(\Gamma
,\Delta ^{0}\otimes \mathbf{P}_{n}(K_{p}))^{c}\overset{\simeq }{{\
\longrightarrow }}(\mathbf{MS}(K_{p})^{c})^{\vee }  \label{Analogue}
\end{equation}
is also an isomorphism.

It then follows from the isomorphism $\delta ^{c}:\mathbf{MS}(K_{p})^{c} 
\overset{\simeq }{\rightarrow }\mathbf{H}(K_{p})^{c}$ that the left vertical
arrow also induces an isomorphism $H_{2}(\Gamma ,\mathbf{P}
_{n}(K_{p}))^{c}\simeq H_{1}(\Gamma ,\Delta ^{0}\otimes \mathbf{P}
_{n}(K_{p}))^{c}$. This is helpful, because it allows to construct an $L$
-invariant ${\mathcal{L}}$, as the one already introduced in Definition \ref%
{Linv}, purely in terms of modular symbols, as we now explain.

Let as before $\mathrm{pr}_c:\mathbf{MS}(K_p)^{\vee} {\longrightarrow } ( 
\mathbf{MS}(K_p)^c)^{\vee}$ denote the natural projection and write $\Phi_{ 
\mathcal{MS} }^{\ast} = \mathrm{pr}_c\circ \Psi_{\mathcal{MS}}^{\ast}$ for
either $\ast = \log$ or ${\func{ord}}$. In light of \eqref{Analogue} and
reasoning exactly as in the proof of Corollary \ref{Corollary existence and
uniqueness of the L-invariant}, there exists a unique endomorphism ${\ 
\mathcal{L}}_{\mathcal{MS} } \in \func{End}_{\mathbb{T}_p}((\mathbf{MS}( 
\mathbb{Q}_p)^{c})^{\vee})$ such that 
\begin{equation}  \label{LMS}
\Phi_{\mathcal{MS}}^{\log }\circ \partial_1 ={\mathcal{L}}_{\mathcal{MS}}
\circ \Phi_{\mathcal{MS}}^{\func{ord}}\circ \partial_1 : H_1(\Gamma ,\Delta
^0\otimes \mathbf{P}_n(K_p)) \rightarrow (\mathbf{MS}(K_p)^{c})^{\vee}\text{%
. }
\end{equation}

The invariants ${\mathcal{L}}$ and ${\mathcal{L}}_{\mathcal{MS}}$ are
actually equal, as it follows from \eqref{Diagram comparing theories 1}, %
\eqref{Analogue} and the definition of the $L$-invariants. On the $f$
-isotypic component ${\mathcal{L}}_{\mathcal{MS}}$ specializes to the Orton $%
L$-invariant (see \cite{Or}). Hence they induce isomorphic monodromy
modules. Indeed, let $w_{\infty }\in \{\pm 1\}$ be a choice of a sign and
define a monodromy module 
\begin{equation}
\mathbf{D}_{\mathcal{MS}}=\mathbf{D}_{\mathcal{MS}}^{w_{\infty }}:=\mathbf{%
MS }(\mathbb{Q}_{p})^{c,\vee ,w_{\infty }}\oplus \mathbf{MS}(\mathbb{Q}
_{p})^{c,\vee ,w_{\infty }}  \label{DMS}
\end{equation}
over $\mathbb{Q}_{p}$ as in \eqref{monmod}, providing it with a structure of
filtered Frobenius module by formally replacing $\mathbf{H}$ by $\mathbf{MS}$
, and $\mathcal{L}$ by ${\mathcal{L}}_{\mathcal{MS}}$. It follows from the
discussion above and the explicit description of both monodromy modules that
the isomorphism $\delta ^{c}:\mathbf{MS}^{c}\overset{\simeq }{\rightarrow } 
\mathbf{H}^{c}$ induces an isomorphism 
\begin{equation}
\mathbf{D}\overset{\simeq }{{\longrightarrow }}\mathbf{D}_{\mathcal{MS}}.
\label{DD}
\end{equation}

\vspace{0.3cm}

Finally, we conclude this section by showing how the Darmon cycles that were
introduced in §\ref{S31} can also be recovered by means of the theory of
modular symbols when $N^{-}=1$; this point of view will be of fundamental
importance in \cite{Se}. As in \eqref{PhiOnto}, set 
\begin{equation*}
\Phi _{\mathcal{MS}}:=-\Phi _{\mathcal{MS}}^{\log }\oplus \Phi _{\mathcal{MS}
}^{{\func{ord}}}:(\Delta ^{0}\otimes \limfunc{Div}\nolimits^{0}(\mathcal{H}
_{p})\otimes \mathbf{P}_{n}(K_{p}))_{\Gamma }{\longrightarrow }\mathbf{D}
(\Delta ^{0}\otimes \limfunc{Div}\nolimits^{0}(\mathcal{H}
_{p})(k_{p})\otimes \mathbf{P}_{n}(K_{p}))_{\Gamma }.
\end{equation*}

As in Definition \ref{AJ1} and in \eqref{AJ2}, we would like to be able to
define a canonical morphism $\Phi^{AJ}_{\mathcal{MS}}: (\Delta^0\otimes 
\limfunc{Div}(\mathcal{H}_p)(k_p) \otimes \mathbf{P}_n(K_p)) _{\Gamma } {\
\longrightarrow } \mathbf{D}_{\mathcal{MS}}/F^{m}$ such that $\Phi^{AJ}_{ 
\mathcal{MS}}=\Phi \circ \eta$. There is however a slight complication here,
as $(\Delta^0\otimes \mathbf{P}_n(K_p)) _{\Gamma }$ is not trivial. The
reader may like to compare this situation with the one encountered in §\ref%
{S42}, where the counterpart of $(\Delta^0\otimes \mathbf{P}_n(K_p))
_{\Gamma }$ was $H^1(\Gamma,\mathbf{V}_n)$, which is trivial by Lemma \ref%
{Lemma H^1(gamma,V_n)=0}. This motivates the following

\begin{definition}
A $p$-adic Abel-Jacobi map with respect to $\Phi _{\mathcal{MS}}$ is a
morphism 
\begin{equation}
\Phi _{\mathcal{MS}}^{AJ}:(\Delta ^{0}\otimes \limfunc{Div}(\mathcal{H}
_{p})(k_{p})\otimes \mathbf{P}_{n}(K_{p}))_{\Gamma }{\longrightarrow } 
\mathbf{D}_{\mathcal{MS}}\otimes K_{p}/F^{m}
\end{equation}
such that $\Phi _{\mathcal{MS}}^{AJ}=\Phi \circ \eta $.
\end{definition}

Such morphisms exist, but they are not unique. Using now the notation
introduced in §\ref{S31}, there is a diagram 
\begin{equation}
\begin{array}{ccccc}
\Gamma \backslash \mathrm{Emb}({\mathcal{O}},\mathcal{R}) & \overset{y}{
\rightarrow } & H_{1}(\Gamma ,\limfunc{Div}(\mathcal{H}_{p})(k_{p})\otimes 
\mathbf{P}_{n}(K_{p})) & \overset{\Phi ^{AJ}}{\rightarrow } & \mathbf{D}
\otimes K_{p}/F^{m} \\ 
& y_{\mathcal{MS}}\searrow & \downarrow \text{ } &  & \downarrow \\ 
&  & (\Delta ^{0}\otimes \limfunc{Div}(\mathcal{H}_{p})\otimes \mathbf{P}
_{n}(K_{p}))_{\Gamma }. & \overset{\Phi_{\mathcal{MS}}^{AJ}}{\rightarrow }
& \mathbf{D}_{\mathcal{MS}}\otimes K_{p}/F^{m}%
\end{array}
\label{Diagram comparing theories 2}
\end{equation}

Here $y_{\mathcal{MS}}\left( \Psi \right)$ is defined to be the class of $\gamma _{\Psi
}x-x\otimes \tau _{\Psi }\otimes D_{K}^{-\frac{k-2}{4}}P_{\Psi }^{m}$, where
an arbitrary choice of $x\in \mathbb{P}^{1}\left( \mathbb{Q}\right) $ has
beeen fixed. The map $y_{\mathcal{MS}}$ is indeed well defined, as it easily
follows by arguing as in Lemma \ref{Lemma Darmon cycles}. Thus, along with the \emph{Darmon cohomology classes} $s_{[\Psi]}$ attached to $[\Psi]\in
\Gamma \backslash \mathrm{Emb}({\mathcal{O}},\mathcal{R})$ introduced in Definition \ref{defSH}, we can also define $s_{\mathcal{MS}}([\Psi]):=\Phi_{\mathcal{MS}}^{AJ}(y_{\mathcal{MS}}\left( \Psi \right))\in \mathbf{D}_{\mathcal{MS}}\otimes K_p/F^{m}$.

Although the triangle of in \eqref{Diagram comparing theories 2} is
commutative, we warn the reader that the square in 
\eqref{Diagram comparing
theories 2} may not be. This is due to the fact that an arbitrary choice of
a $p$-adic Abel-Jacobi map $\Phi _{\mathcal{MS}}^{AJ}$ has been made.
Fortunately, it can be shown that the image of $\Psi $ in $\mathbf{D}_{%
\mathcal{MS}}\otimes K_{p}/F^{m}$ does not depend on the choice of $\Phi _{%
\mathcal{MS}}^{AJ}$; see \cite[Proposition 2.22]{Se} for more details, where
it is proved that although the square in \eqref{Diagram comparing theories 2}
may not be commutative, the entire diagram is, i.e. 
\begin{equation}
\Phi ^{AJ}(y_{\Psi })=\Phi _{\mathcal{MS}}^{AJ}(y_{\mathcal{MS}}(\Psi )).
\label{eqSH}
\end{equation}

\end{document}